\documentclass[a4paper,onecolumn, superscriptaddress,10pt, shorttitle=papers,allowtoday]{compositionalityarticle}
 \pdfoutput=1
\usepackage[utf8]{inputenc}
\usepackage[english]{babel}
\usepackage{amsmath}
\usepackage{hyperref}
\usepackage{tikz-cd}
\usepackage{quiver}

\usepackage{natbib}
\usepackage{tikz}
\usepackage{lipsum}
\usepackage{dsfont}
\usepackage{nicefrac, xfrac}
\usepackage{comment}
\usepackage{soul}
\usepackage{xcolor}

\newcommand{\jnote}[1]{{\color{red}[JHS: #1]}}

\newcommand{\cO}{\mathcal{O}}
\newcommand{\cW}{\mathfrak{W}}
\newcommand{\Operad}{{\O}perad }
\newcommand{\Operads}{{\O}perads }

\usepackage{natbib}
\usepackage{rotating}
\usepackage{floatpag}
\usepackage{xifthen}

\rotfloatpagestyle{empty}
\renewcommand{\theenumi}{\arabic{enumi}.}
\ifthenelse{\isundefined{\showoptional}}{
  \newcommand{\showoptional}{1}
}{}

\ifthenelse{\isundefined{\ismain}}{
  \newcommand{\ismain}{0}
}{}

\ifthenelse{\isundefined{\lecturenotes}}{
  \newcommand{\lecturenotes}{0}
}{}

\usepackage{hyperref}

\usepackage{amsmath,amsthm,amssymb,mathabx}

\usepackage{xspace,enumerate,color,epsfig}
\usepackage{graphicx}
\graphicspath{{.}{./figures/}}

\usepackage{tikzfig}
\usepackage{stmaryrd}
\usepackage{docmute}
\usepackage{keycommand}

\newkeycommand{\wirelabel}[style=white label]{\,\tikz{\node[style=\commandkey{style}]{};}\,\xspace}

\newkeycommand{\pointmap}[style=point][1]{\,\tikz{\node[style=\commandkey{style}] (x) at (0,-0.3) {$#1$};\node [style=none] (3) at (0, 0.9) {};\node [style=none] (3) at (0, -0.9) {};\draw (x)--(0,0.7);}\,}

\newkeycommand{\copointmap}[style=copoint][1]{\,\tikz{\node[style=\commandkey{style}] (x) at (0,0.3) {$#1$};\draw (x)--(0,-0.7);}\,}

\newkeycommand{\dpoint}[style=point][1]{\,\tikz{\node[style=\commandkey{style},doubled] (x) at (0,-0.3) {$#1$};\draw[boldedge] (x)--(0,0.7);}\,}

\newkeycommand{\kpoint}[style=kpoint,estyle=][1]{\,\tikz{\node[style=\commandkey{style}] (x) at (0,-0.05) {$#1$};\draw [\commandkey{estyle}] (x)--(0,1.05);}\,}

\newkeycommand{\kpointgray}[style=gray kpoint,estyle=][1]{\,\tikz{\node[style=\commandkey{style}] (x) at (0,-0.05) {$#1$};\draw [\commandkey{estyle}] (x)--(0,1.05);}\,}

\newkeycommand{\typedkpoint}[style=kpoint,edgestyle=][2]{%
\,\begin{tikzpicture}
  \begin{pgfonlayer}{nodelayer}
    \node [style=none] (0) at (0, 1) {};
    \node [style=\commandkey{style}] (1) at (0, -0.75) {$#2$};
    \node [style=right label] (2) at (0.25, 0.5) {$#1$};
  \end{pgfonlayer}
  \begin{pgfonlayer}{edgelayer}
    \draw [\commandkey{edgestyle}] (1) to (0.center);
  \end{pgfonlayer}
\end{tikzpicture}\,}

\newkeycommand{\typedkpointdag}[style=kpoint adjoint,edgestyle=][2]{%
\,\begin{tikzpicture}
  \begin{pgfonlayer}{nodelayer}
    \node [style=none] (0) at (0, -1) {};
    \node [style=\commandkey{style}] (1) at (0, 0.75) {$#2$};
    \node [style=right label] (2) at (0.25, -0.5) {$#1$};
  \end{pgfonlayer}
  \begin{pgfonlayer}{edgelayer}
    \draw [\commandkey{edgestyle}] (0.center) to (1);
  \end{pgfonlayer}
\end{tikzpicture}\,}

\newkeycommand{\bistate}[style=kpoint][1]{\,\begin{tikzpicture}
  \begin{pgfonlayer}{nodelayer}
    \node [style=\commandkey{style}, minimum width=1 cm, inner sep=2pt] (0) at (0, -0.25) {$#1$};
    \node [style=none] (1) at (-0.75, 0) {};
    \node [style=none] (2) at (0.75, 0) {};
    \node [style=none] (3) at (-0.75, 0.88) {};
    \node [style=none] (4) at (0.75, 0.88) {};
  \end{pgfonlayer}
  \begin{pgfonlayer}{edgelayer}
    \draw (1.center) to (3.center);
    \draw (2.center) to (4.center);
  \end{pgfonlayer}
\end{tikzpicture}\,}

\newkeycommand{\bistatebig}[style=kpoint][1]{\,\begin{tikzpicture}
  \begin{pgfonlayer}{nodelayer}
    \node [style=\commandkey{style}, minimum width=, minimum width=1.26 cm, inner sep=2pt] (0) at (0, -0.25) {$#1$};
    \node [style=none] (1) at (-1, 0) {};
    \node [style=none] (2) at (1, 0) {};
    \node [style=none] (3) at (-1, 0.88) {};
    \node [style=none] (4) at (1, 0.88) {};
  \end{pgfonlayer}
  \begin{pgfonlayer}{edgelayer}
    \draw (1.center) to (3.center);
    \draw (2.center) to (4.center);
  \end{pgfonlayer}
\end{tikzpicture}\,}

\newkeycommand{\dbistate}[style=dkpoint][1]{\,\begin{tikzpicture}
  \begin{pgfonlayer}{nodelayer}
    \node [style=\commandkey{style}, minimum width=1 cm, inner sep=2pt] (0) at (0, -0.25) {$#1$};
    \node [style=none] (1) at (-0.75, 0) {};
    \node [style=none] (2) at (0.75, 0) {};
    \node [style=none] (3) at (-0.75, 1.0) {};
    \node [style=none] (4) at (0.75, 1.0) {};
  \end{pgfonlayer}
  \begin{pgfonlayer}{edgelayer}
    \draw [boldedge] (1.center) to (3.center);
    \draw [boldedge] (2.center) to (4.center);
  \end{pgfonlayer}
\end{tikzpicture}\,}

\newkeycommand{\bistatebraket}[style1=kpoint,style2=kpointadj][2]{\,\begin{tikzpicture}
  \begin{pgfonlayer}{nodelayer}
    \node [style=\commandkey{style1}, minimum width=1 cm, inner sep=2pt] (0) at (0, -0.75) {$#1$};
    \node [style=none] (1) at (-0.75, -0.5) {};
    \node [style=none] (2) at (0.75, -0.5) {};
    \node [style=none] (3) at (-0.75, 0.5) {};
    \node [style=none] (4) at (0.75, 0.5) {};
    \node [style=\commandkey{style2}, minimum width=1 cm, inner sep=2pt] (5) at (0, 0.75) {$#2$};
  \end{pgfonlayer}
  \begin{pgfonlayer}{edgelayer}
    \draw (1.center) to (3.center);
    \draw (2.center) to (4.center);
  \end{pgfonlayer}
\end{tikzpicture}\,}

\newkeycommand{\dkpoint}[style=dkpoint][1]{\,\tikz{\node[style=\commandkey{style}] (x) at (0,-0.1) {$#1$};\draw[boldedge] (x)--(0,1);}\,}
\newkeycommand{\dkpointadj}[style=dkpointadj][1]{\,\tikz{\node[style=\commandkey{style}] (x) at (0,0.1) {$#1$};\draw[boldedge] (x)--(0,-1);}\,}
\newkeycommand{\dkpointtrans}[style=dkpointtrans][1]{\,\tikz{\node[style=\commandkey{style}] (x) at (0,0.1) {$#1$};\draw[boldedge] (x)--(0,-1);}\,}
\newkeycommand{\dkpointconj}[style=dkpointconj][1]{\,\tikz{\node[style=\commandkey{style}] (x) at (0,-0.1) {$#1$};\draw[boldedge] (x)--(0,1);}\,}

\newcommand{\trace}{\,\begin{tikzpicture}
  \begin{pgfonlayer}{nodelayer}
    \node [style=none] (0) at (0, -0.60) {};
    \node [style=upground] (1) at (0, 0.40) {};
  \end{pgfonlayer}
  \begin{pgfonlayer}{edgelayer}
    \draw [style=boldedge] (0.center) to (1);
  \end{pgfonlayer}
\end{tikzpicture}\,}

\newcommand\discard\trace

\newkeycommand{\pointketbra}[style1=copoint,style2=point][2]{\,%
\begin{tikzpicture}
  \begin{pgfonlayer}{nodelayer}
    \node [style=none] (0) at (0, -1.75) {};
    \node [style=\commandkey{style1}] (1) at (0, -0.75) {$#1$};
    \node [style=\commandkey{style2}] (2) at (0, 0.75) {$#2$};
    \node [style=none] (3) at (0, 1.75) {};
  \end{pgfonlayer}
  \begin{pgfonlayer}{edgelayer}
    \draw (0.center) to (1);
    \draw (2) to (3.center);
  \end{pgfonlayer}
\end{tikzpicture}\,}

\newkeycommand{\twocopointketbra}[style1=copoint,style2=copoint,style3=point][3]{\,%
\begin{tikzpicture}
  \begin{pgfonlayer}{nodelayer}
    \node [style=none] (0) at (-0.75, -1.75) {};
    \node [style=none] (1) at (0.75, -1.75) {};
    \node [style=\commandkey{style1}] (2) at (-0.75, -0.75) {$#1$};
    \node [style=\commandkey{style2}] (3) at (0.75, -0.75) {$#2$};
    \node [style=\commandkey{style3}] (4) at (0, 0.75) {$#3$};
    \node [style=none] (5) at (0, 1.75) {};
  \end{pgfonlayer}
  \begin{pgfonlayer}{edgelayer}
    \draw (0.center) to (2);
    \draw (1.center) to (3);
    \draw (4) to (5.center);
  \end{pgfonlayer}
\end{tikzpicture}\,}

\newkeycommand{\twopointketbra}[style1=copoint,style2=point,style3=point][3]{\,%
\begin{tikzpicture}
  \begin{pgfonlayer}{nodelayer}
    \node [style=none] (0) at (0, -1.75) {};
    \node [style=none] (1) at (-0.75, 1.75) {};
    \node [style=\commandkey{style1}] (2) at (0, -0.75) {$#1$};
    \node [style=\commandkey{style2}] (3) at (-0.75, 0.75) {$#2$};
    \node [style=\commandkey{style3}] (4) at (0.75, 0.75) {$#3$};
    \node [style=none] (5) at (0.75, 1.75) {};
  \end{pgfonlayer}
  \begin{pgfonlayer}{edgelayer}
    \draw (0.center) to (2);
    \draw (1.center) to (3);
    \draw (4) to (5.center);
  \end{pgfonlayer}
\end{tikzpicture}\,}

\newkeycommand{\pointbraket}[style1=point,style2=copoint,estyle=][2]{\,%
\begin{tikzpicture}
  \begin{pgfonlayer}{nodelayer}
    \node [style=\commandkey{style1}] (0) at (0, -0.625) {$#1$};
    \node [style=\commandkey{style2}] (1) at (0, 0.625) {$#2$};
  \end{pgfonlayer}
  \begin{pgfonlayer}{edgelayer}
    \draw [\commandkey{estyle}] (0) to (1);
  \end{pgfonlayer}
\end{tikzpicture}\,}

\newkeycommand{\kpointketbra}[style1=kpointdag,style2=kpoint,estyle=][2]{\,%
\begin{tikzpicture}
  \begin{pgfonlayer}{nodelayer}
    \node [style=none] (0) at (0, -1.9) {};
    \node [style=\commandkey{style1}] (1) at (0, -0.95) {$#1$};
    \node [style=\commandkey{style2}] (2) at (0, 0.95) {$#2$};
    \node [style=none] (3) at (0, 1.9) {};
  \end{pgfonlayer}
  \begin{pgfonlayer}{edgelayer}
    \draw [\commandkey{estyle}] (0.center) to (1);
    \draw [\commandkey{estyle}] (2) to (3.center);
  \end{pgfonlayer}
\end{tikzpicture}\,}

\newkeycommand{\kpointmap}[style1=kpoint,style2=map][2]{\,%
\begin{tikzpicture}
  \begin{pgfonlayer}{nodelayer}
    \node [style=\commandkey{style1}] (0) at (0, -1.1) {$#1$};
    \node [style=\commandkey{style2}] (1) at (0, 0.5) {$#2$};
    \node [style=none] (2) at (0, 1.75) {};
  \end{pgfonlayer}
  \begin{pgfonlayer}{edgelayer}
    \draw (0) to (1);
    \draw (1) to (2.center);
  \end{pgfonlayer}
\end{tikzpicture}%
\,}

\newkeycommand{\dkpointmap}[style1=dkpoint,style2=dmap][2]{\,%
\begin{tikzpicture}
  \begin{pgfonlayer}{nodelayer}
    \node [style=\commandkey{style1}] (0) at (0, -1.2) {$#1$};
    \node [style=\commandkey{style2}] (1) at (0, 0.4) {$#2$};
    \node [style=none] (2) at (0, 1.7) {};
  \end{pgfonlayer}
  \begin{pgfonlayer}{edgelayer}
    \draw[style=boldedge] (0) to (1);
    \draw[style=boldedge] (1) to (2.center);
  \end{pgfonlayer}
\end{tikzpicture}%
\,}

\newkeycommand{\dkpointmapx}[style1=dkpoint,style2=dmap][2]{\,%
\begin{tikzpicture}
  \begin{pgfonlayer}{nodelayer}
    \node [style=\commandkey{style1}] (0) at (0, -1.4) {$#1$};
    \node [style=\commandkey{style2}] (1) at (0, 0.6) {$#2$};
    \node [style=none] (2) at (0, 1.9) {};
  \end{pgfonlayer}
  \begin{pgfonlayer}{edgelayer}
    \draw[style=boldedge] (0) to (1);
    \draw[style=boldedge] (1) to (2.center);
  \end{pgfonlayer}
\end{tikzpicture}%
\,}

\newkeycommand{\kpointsandwichmap}[style1=kpoint,style2=map,style3=kpointdag][3]{\,%
\begin{tikzpicture}
  \begin{pgfonlayer}{nodelayer}
    \node [style=\commandkey{style1}] (0) at (0, -1.5) {$#1$};
    \node [style=\commandkey{style2}] (1) at (0, 0) {$#2$};
    \node [style=\commandkey{style3}] (2) at (0, 1.5) {$#3$};
  \end{pgfonlayer}
  \begin{pgfonlayer}{edgelayer}
    \draw (0) to (1);
    \draw (1) to (2);
  \end{pgfonlayer}
\end{tikzpicture}%
}

\newkeycommand{\sandwichtwo}[style1=point,style2=map,style3=map,style4=copoint][4]{\,%
\begin{tikzpicture}
  \begin{pgfonlayer}{nodelayer}
    \node [style=\commandkey{style2}] (0) at (0, -1) {$#2$};
    \node [style=\commandkey{style3}] (1) at (0, 1) {$#3$};
    \node [style=\commandkey{style1}] (2) at (0, -2.6) {$#1$};
    \node [style=\commandkey{style4}] (3) at (0, 2.6) {$#4$};
  \end{pgfonlayer}
  \begin{pgfonlayer}{edgelayer}
    \draw (2) to (0);
    \draw (0) to (1);
    \draw (1) to (3);
  \end{pgfonlayer}
\end{tikzpicture}\,}

\newkeycommand{\longbraket}[style1=point,style2=copoint][2]{\,%
\begin{tikzpicture}
  \begin{pgfonlayer}{nodelayer}
    \node [style=\commandkey{style1}] (0) at (0, -2.6) {$#1$};
    \node [style=\commandkey{style2}] (1) at (0, 2.6) {$#2$};
  \end{pgfonlayer}
  \begin{pgfonlayer}{edgelayer}
    \draw (0) to (1);
  \end{pgfonlayer}
\end{tikzpicture}\,}

\newkeycommand{\onb}[style=point]{\ensuremath{\{\,\tikz{\node[style=\commandkey{style}] (x) at (0,-0.3) {$j$};\draw (x)--(0,0.7);}\,\}}\xspace}

\newkeycommand{\phase}[style=white dot][1]{\,\begin{tikzpicture}
    \begin{pgfonlayer}{nodelayer}
        \node [style=none] (0) at (0, 1) {};
        \node [style=\commandkey{style}] (2) at (0, -0) {$#1$}; 
        \node [style=none] (3) at (0, -1) {};
    \end{pgfonlayer}
    \begin{pgfonlayer}{edgelayer}
        \draw (2) to (0.center);
        \draw (3.center) to (2);
    \end{pgfonlayer}
\end{tikzpicture}\,}

\newkeycommand{\dphase}[style=white phase ddot][1]{\,\begin{tikzpicture}
    \begin{pgfonlayer}{nodelayer}
        \node [style=none] (0) at (0, 1.25) {};
        \node [style=\commandkey{style}] (2) at (0, -0) {$#1$};
        \node [style=none] (3) at (0, -1.25) {};
    \end{pgfonlayer}
    \begin{pgfonlayer}{edgelayer}
        \draw [boldedge] (2) to (0.center);
        \draw [boldedge] (3.center) to (2);
    \end{pgfonlayer}
\end{tikzpicture}\,}

\newkeycommand{\dphasepoint}[style=white phase ddot][1]{\,\begin{tikzpicture}[yshift=-3mm]
  \begin{pgfonlayer}{nodelayer}
    \node [style=none] (0) at (0, 1) {};
    \node [style=\commandkey{style}] (1) at (0, 0) {$#1$};
  \end{pgfonlayer}
  \begin{pgfonlayer}{edgelayer}
    \draw [style=boldedge] (0.center) to (1);
  \end{pgfonlayer}
\end{tikzpicture}\,}

\newkeycommand{\dphasepointgray}[style=gray phase ddot][1]{\,\begin{tikzpicture}[yshift=-3mm]
  \begin{pgfonlayer}{nodelayer}
    \node [style=none] (0) at (0, 1) {};
    \node [style=\commandkey{style}] (1) at (0, 0) {$#1$};
  \end{pgfonlayer}
  \begin{pgfonlayer}{edgelayer}
    \draw [style=boldedge] (0.center) to (1);
  \end{pgfonlayer}
\end{tikzpicture}\,}

\newkeycommand{\dphasecopoint}[style=white phase ddot][1]{\,\begin{tikzpicture}[yshift=3mm,yscale=-1]
  \begin{pgfonlayer}{nodelayer}
    \node [style=none] (0) at (0, 1) {};
    \node [style=\commandkey{style}] (1) at (0, 0) {$#1$};
  \end{pgfonlayer}
  \begin{pgfonlayer}{edgelayer}
    \draw [style=boldedge] (0.center) to (1);
  \end{pgfonlayer}
\end{tikzpicture}\,}

\newkeycommand{\dphasepointsm}[style=white phase ddot][1]{\,\begin{tikzpicture}[yshift=-3mm]
  \begin{pgfonlayer}{nodelayer}
    \node [style=none] (0) at (0, 1) {};
    \node [style=\commandkey{style}] (1) at (0, 0) {\footnotesize\!$#1$\!};
  \end{pgfonlayer}
  \begin{pgfonlayer}{edgelayer}
    \draw [style=boldedge] (0.center) to (1);
  \end{pgfonlayer}
\end{tikzpicture}\,}

\newkeycommand{\phasepoint}[style=white phase dot][1]{\,\begin{tikzpicture}[yshift=-3mm]
  \begin{pgfonlayer}{nodelayer}
    \node [style=none] (0) at (0, 1) {};
    \node [style=\commandkey{style}] (1) at (0, 0) {$#1$};
  \end{pgfonlayer}
  \begin{pgfonlayer}{edgelayer}
    \draw (0.center) to (1);
  \end{pgfonlayer}
\end{tikzpicture}\,}

\newkeycommand{\kpointbraket}[style1=kpoint,style2=kpoint adjoint,estyle=][2]{\,%
\begin{tikzpicture}
  \begin{pgfonlayer}{nodelayer}
    \node [style=\commandkey{style1}] (0) at (0, -0.735) {$#1$};
    \node [style=\commandkey{style2}] (1) at (0, 0.735) {$#2$};
  \end{pgfonlayer}
  \begin{pgfonlayer}{edgelayer}
    \draw [\commandkey{estyle}] (0) to (1);
  \end{pgfonlayer}
\end{tikzpicture}\,}

\newkeycommand{\kpointbraketx}[style1=kpoint,style2=kpoint adjoint,estyle=][2]{\,%
\begin{tikzpicture}
  \begin{pgfonlayer}{nodelayer}
    \node [style=\commandkey{style1}] (0) at (0, -0.885) {$#1$};
    \node [style=\commandkey{style2}] (1) at (0, 0.885) {$#2$};
  \end{pgfonlayer}
  \begin{pgfonlayer}{edgelayer}
    \draw [\commandkey{estyle}] (0) to (1);
  \end{pgfonlayer}
\end{tikzpicture}\,}

\tikzstyle{cdiag}=[matrix of math nodes, row sep=3em, column sep=3em, text height=1.5ex, text depth=0.25ex,inner sep=0.5em]
\tikzstyle{arrow above}=[transform canvas={yshift=0.5ex}]
\tikzstyle{arrow below}=[transform canvas={yshift=-0.5ex}]

\usetikzlibrary{shapes.multipart}
\usetikzlibrary{decorations.markings}

\tikzset{->-/.style={decoration={
  markings,
  mark=at position .5 with {\arrow{>}}},postaction={decorate}}}
\tikzset{-<-/.style={decoration={
  markings,
  mark=at position .5 with {\arrow{<}}},postaction={decorate}}}
  
\tikzstyle{bwSpider}=[
       rectangle split,
       rectangle split parts=2,
       rectangle split part fill={black,white},
 minimum size=3.6 mm, inner sep=-2mm, draw=black,scale=0.5,rounded corners=0.8 mm
       ]
 \tikzstyle{wbSpider}=[
       rectangle split,
       rectangle split parts=2,
       rectangle split part fill={white,black},
 minimum size=3.6 mm, inner sep=-2mm, draw=black,scale=0.5,rounded corners=0.8 mm
       ]
       \definecolor{quantumviolet}{RGB}{79, 4, 134}
\tikzstyle{oWire}=[line width = 1pt, color=black]
\tikzstyle{lightWire}=[line width = 0.5pt, color=compositionalitydarkpink]
\tikzstyle{QWire}=[line width = 1pt, color=quantumviolet]
\tikzstyle{qWire}=[line width = 1pt, color=compositionalitydarkpink]
\tikzstyle{cWire}=[color=darkgray,line width = .75pt]
\tikzstyle{CqWire}=[line width = 1pt, color=compositionalitydarkpink,->-]
\tikzstyle{CQWire}=[line width = 1pt, color=quantumviolet,->-]
\tikzstyle{CcWire}=[color=darkgray,line width = .75pt,->-]
\tikzstyle{RqWire}=[line width = 1pt, color=compositionalitydarkpink,-<-]
\tikzstyle{RQWire}=[line width = 1pt, color=quantumviolet,-<-]
\tikzstyle{RcWire}=[color=darkgray,line width = .75pt,-<-]

\tikzstyle{epiCopoint}=[regular polygon,regular polygon sides=3,draw,scale=0.8,inner sep=-0.5pt,minimum width=5mm,fill=white,regular polygon rotate=0,line width=.5pt]
\tikzstyle{epiPoint}=[regular polygon,regular polygon sides=3,draw,scale=0.8,inner sep=-0.5pt,minimum width=5mm,fill=white,regular polygon rotate=180,line width=.5pt]
\tikzstyle{epiPointWide}=[regular polygon,regular polygon sides=3,draw,xscale=0.75,yscale=.55,inner sep=-0.5pt,minimum width=8mm,fill=white,regular polygon rotate=180,line width=.5pt]
\tikzstyle{epiCopointWide}=[regular polygon,regular polygon sides=3,draw,xscale=0.75, yscale=.55,inner sep=-0.5pt,minimum width=8mm,fill=white,regular polygon rotate=0,line width=.5pt]
\tikzstyle{epiBox}=[fill=white,draw, line width = .5pt,inner sep=0.6mm,font=\footnotesize,minimum height=3mm,minimum width=3mm]
\tikzstyle{epiBoxWide}=[fill=white,draw, line width = .5pt,inner sep=0.6mm,font=\footnotesize,minimum height=3.5mm,minimum width=5mm]
\tikzstyle{epiBoxVeryWide}=[fill=white,draw, line width = .5pt,inner sep=0.6mm,font=\footnotesize,minimum height=3mm,minimum width=7mm]
\tikzstyle{infpoint}=[regular polygon,regular polygon sides=3,draw,scale=0.75,inner sep=-0.5pt,minimum width=9mm,fill=white,regular polygon rotate=90]
\tikzstyle{infcopoint}=[regular polygon,regular polygon sides=3,draw,scale=0.75,inner sep=-0.5pt,minimum width=9mm,fill=white,regular polygon rotate=270]
 \tikzstyle{infupground}=[circuit ee IEC,thick,ground,rotate=0,xscale=2.5,yscale=2]
\tikzstyle{env}=[copoint,regular polygon rotate=0,minimum width=0.2cm, fill=black]

\tikzstyle{probs}=[shape=semicircle,fill=white,draw=black,shape border rotate=180,minimum width=1.2cm]

\tikzstyle{every picture}=[baseline=-0.25em,scale=0.5]
\tikzstyle{dotpic}=[] 
\tikzstyle{diredges}=[every to/.style={diredge}]
\tikzstyle{math matrix}=[matrix of math nodes,left delimiter=(,right delimiter=),inner sep=2pt,column sep=1em,row sep=0.5em,nodes={inner sep=0pt},text height=1.5ex, text depth=0.25ex]

\tikzstyle{inline text}=[text height=1.5ex, text depth=0.25ex,yshift=0.5mm]
\tikzstyle{label}=[font=\footnotesize,text height=1.5ex, text depth=0.25ex,yshift=0.5mm]
\tikzstyle{left label}=[label,anchor=east,xshift=1.5mm]
\tikzstyle{right label}=[label,anchor=west,xshift=-1mm]
\tikzstyle{up label}=[label,anchor=south,yshift=-1mm]

\tikzstyle{rght label}=[label,anchor=west,xshift=-1mm]

\tikzstyle{braceedge}=[decorate,decoration={brace,amplitude=2mm,raise=-1mm}]
\tikzstyle{small braceedge}=[decorate,decoration={brace,amplitude=1mm,raise=-1mm}]

\tikzstyle{doubled}=[line width=1.6pt] 
\tikzstyle{boldedge}=[doubled,shorten <=-0.17mm,shorten >=-0.17mm]
\tikzstyle{boldedgegray}=[doubled,gray,shorten <=-0.17mm,shorten >=-0.17mm]
\tikzstyle{singleedgegray}=[gray]

\tikzstyle{semidoubled}=[line width=1.4pt] 
\tikzstyle{semiboldedgegray}=[semidoubled,gray,shorten <=-0.17mm,shorten >=-0.17mm]

\tikzstyle{boxedge}=[semiboldedgegray]

\tikzstyle{boldedgedashed}=[very thick,dashed,shorten <=-0.17mm,shorten >=-0.17mm]
\tikzstyle{vboldedgedashed}=[doubled,dashed,shorten <=-0.17mm,shorten >=-0.17mm]
\tikzstyle{left hook arrow}=[left hook-latex]
\tikzstyle{right hook arrow}=[right hook-latex]
\tikzstyle{sembracket}=[line width=0.5pt,shorten <=-0.07mm,shorten >=-0.07mm]

\tikzstyle{causal edge}=[->,thick,gray]
\tikzstyle{causal nondir}=[thick,gray]
\tikzstyle{timeline}=[thick,gray, dashed]

\tikzstyle{cedge}=[<->,thick,gray!70!white]

\tikzstyle{empty diagram}=[draw=gray!40!white,dashed,shape=rectangle,minimum width=1cm,minimum height=1cm]
\tikzstyle{empty diagram small}=[draw=gray!50!white,dashed,shape=rectangle,minimum width=0.6cm,minimum height=0.5cm]

\tikzstyle{arrow}=[->]

\tikzstyle{dot}=[inner sep=0mm,minimum width=2mm,minimum height=2mm,draw,shape=circle]
\tikzstyle{bigdot}=[inner sep=0mm,minimum width=5mm,minimum height=5mm,draw,shape=circle]
\tikzstyle{leak}=[white dot, shape=regular polygon, minimum size=3.3 mm, regular polygon sides=3, outer sep=-0.2mm, regular polygon rotate=270]
\tikzstyle{proj}=[regular polygon,regular polygon sides=4,draw,scale=0.75,inner sep=-0.5pt,minimum width=6mm,fill=white]
\tikzstyle{projOut}=[regular polygon,regular polygon sides=3,draw,scale=0.75,inner sep=-0.5pt,minimum width=7.5mm,fill=white,regular polygon rotate=180]
\tikzstyle{projIn}=[regular polygon,regular polygon sides=3,draw,scale=0.75,inner sep=-0.5pt,minimum width=7.5mm,fill=white]
\tikzstyle{Vleak}=[white dot, shape=regular polygon, minimum size=3.3 mm, regular polygon sides=3, outer sep=-0.2mm, regular polygon rotate=90]
\tikzstyle{dleak}=[white dot, line width=1.6pt, shape=regular polygon, minimum size=3.3 mm, regular polygon sides=3, outer sep=-0.2mm, regular polygon rotate=270]

\tikzstyle{Wsquare}=[white dot, shape=regular polygon, rounded corners=0.8 mm, minimum size=3.3 mm, regular polygon sides=3, outer sep=-0.2mm]
\tikzstyle{Wsquareadj}=[white dot, shape=regular polygon, rounded corners=0.8 mm, minimum size=3.3 mm, regular polygon sides=3, outer sep=-0.2mm, regular polygon rotate=180]
\tikzstyle{ddot}=[inner sep=0mm, doubled, minimum width=2.5mm,minimum height=2.5mm,draw,shape=circle]

\tikzstyle{spider}=[inner sep=0mm,minimum width=1mm,minimum height=1mm,draw=compositionalitydarkpink,shape=circle,fill=compositionalitydarkpink]
\tikzstyle{black dot}=[dot,fill=black]
\tikzstyle{small black dot}=[inner sep=0mm,minimum width=.5mm,minimum height=.5mm,draw,shape=circle,fill=black]
\tikzstyle{white dot}=[dot,fill=white,text depth=-0.2mm]
\tikzstyle{big white dot}=[dot,fill=white,text depth=-0.2mm, minimum width = 2.5mm]
\tikzstyle{white Wsquare}=[Wsquare,fill=gray,,text depth=-0.2mm]
\tikzstyle{white Wsquareadj}=[Wsquareadj,fill=white,,text depth=-0.2mm]
\tikzstyle{green dot}=[white dot] 
\tikzstyle{gray dot}=[dot,fill=gray!40!white,,text depth=-0.2mm]
\tikzstyle{red dot}=[gray dot] 

\tikzstyle{black ddot}=[ddot,fill=black]
\tikzstyle{white ddot}=[ddot,fill=white]
\tikzstyle{gray ddot}=[ddot,fill=gray!40!white]

\tikzstyle{gray edge}=[gray!60!white]

\tikzstyle{small dot}=[inner sep=0.5mm,minimum width=0pt,minimum height=0pt,draw,shape=circle]

\tikzstyle{small white dot}=[small dot,fill=white]
\tikzstyle{small gray dot}=[small dot,fill=gray!40!white]

\tikzstyle{causal dot}=[inner sep=0.4mm,minimum width=0pt,minimum height=0pt,draw=white,shape=circle,fill=gray!40!white]

\tikzstyle{phase dimensions}=[minimum size=5mm,font=\footnotesize,rectangle,rounded corners=2.5mm,inner sep=0.2mm,outer sep=-2mm]
\tikzstyle{dphase dimensions}=[minimum size=5mm,font=\footnotesize,rectangle,rounded corners=2.5mm,inner sep=0.2mm,outer sep=-2mm]

\tikzstyle{white phase dot}=[dot,fill=white,phase dimensions]
\tikzstyle{white phase ddot}=[ddot,fill=white,dphase dimensions]

\tikzstyle{white rect ddot}=[draw=black,fill=white,doubled,minimum size=5mm,font=\footnotesize,rectangle,rounded corners=2.5mm,inner sep=0.2mm]
\tikzstyle{gray rect ddot}=[draw=black,fill=gray!40!white,doubled,minimum size=6mm,font=\footnotesize,rectangle,rounded corners=3mm]

\tikzstyle{gray phase dot}=[dot,fill=gray!40!white,phase dimensions]
\tikzstyle{gray phase ddot}=[ddot,fill=gray!40!white,dphase dimensions]
\tikzstyle{grey phase dot}=[gray phase dot]
\tikzstyle{grey phase ddot}=[gray phase ddot]

\tikzstyle{small phase dimensions}=[minimum size=4mm,font=\tiny,rectangle,rounded corners=2mm,inner sep=0.2mm,outer sep=-2mm]
\tikzstyle{small dphase dimensions}=[minimum size=4mm,font=\tiny,rectangle,rounded corners=2mm,inner sep=0.2mm,outer sep=-2mm]

\tikzstyle{small gray phase dot}=[dot,fill=gray!40!white,small phase dimensions]
\tikzstyle{small gray phase ddot}=[ddot,fill=gray!40!white,small dphase dimensions]

\tikzstyle{small map}=[draw,shape=rectangle,minimum height=4mm,minimum width=4mm,fill=white]

\tikzstyle{cnot}=[fill=white,shape=circle,inner sep=-1.4pt]

\tikzstyle{asym hadamard}=[fill=white,draw,shape=NEbox,inner sep=0.6mm,font=\footnotesize,minimum height=4mm]
\tikzstyle{asym hadamard conj}=[fill=white,draw,shape=NWbox,inner sep=0.6mm,font=\footnotesize,minimum height=4mm]
\tikzstyle{asym hadamard dag}=[fill=white,draw,shape=SEbox,inner sep=0.6mm,font=\footnotesize,minimum height=4mm]

\tikzstyle{hadamard}=[fill=white,draw,inner sep=0.6mm,font=\footnotesize,minimum height=4mm,minimum width=4mm]
\tikzstyle{small hadamard}=[fill=white,draw,inner sep=0.6mm,minimum height=1.5mm,minimum width=1.5mm]
\tikzstyle{small hadamard rotate}=[small hadamard,rotate=45]
\tikzstyle{dhadamard}=[hadamard,doubled]
\tikzstyle{small dhadamard}=[small hadamard,doubled]
\tikzstyle{small dhadamard rotate}=[small hadamard rotate,doubled]
\tikzstyle{antipode}=[white dot,inner sep=0.3mm,font=\footnotesize]

\tikzstyle{scalar}=[diamond,draw,inner sep=0.5pt,font=\small]
\tikzstyle{dscalar}=[diamond,doubled, draw,inner sep=0.5pt,font=\small]

\tikzstyle{small box}=[rectangle,inline text,fill=white,draw,minimum height=5mm,yshift=-0.5mm,minimum width=5mm,font=\small]
\tikzstyle{small gray box}=[small box,fill=gray!30]
\tikzstyle{medium box}=[rectangle,inline text,fill=white,draw,minimum height=5mm,yshift=-0.5mm,minimum width=10mm,font=\small]
\tikzstyle{square box}=[small box] 
\tikzstyle{medium gray box}=[small box,fill=gray!30]
\tikzstyle{semilarge box}=[rectangle,inline text,fill=white,draw,minimum height=5mm,yshift=-0.5mm,minimum width=12.5mm,font=\small]
\tikzstyle{large box}=[rectangle,inline text,fill=white,draw,minimum height=5mm,yshift=-0.5mm,minimum width=15mm,font=\small]
\tikzstyle{large gray box}=[small box,fill=gray!30]

\tikzstyle{Bayes box}=[rectangle,fill=black,draw, minimum height=3mm, minimum width=3mm]

\tikzstyle{gray square point}=[small box,fill=gray!50]

\tikzstyle{dphase box white}=[dhadamard]
\tikzstyle{dphase box gray}=[dhadamard,fill=gray!50!white]
\tikzstyle{phase box white}=[hadamard]
\tikzstyle{phase box gray}=[hadamard,fill=gray!50!white]

\tikzstyle{point}=[regular polygon,regular polygon sides=3,draw,scale=0.75,inner sep=-0.5pt,minimum width=9mm,fill=white,regular polygon rotate=180]
\tikzstyle{point nosep}=[regular polygon,regular polygon sides=3,draw,scale=0.75,inner sep=-2pt,minimum width=9mm,fill=white,regular polygon rotate=180]
\tikzstyle{copoint}=[regular polygon,regular polygon sides=3,draw,scale=0.75,inner sep=-0.5pt,minimum width=9mm,fill=white]
\tikzstyle{dpoint}=[point,doubled]
\tikzstyle{dcopoint}=[copoint,doubled]

\tikzstyle{pointgrow}=[shape=cornerpoint,kpoint common,scale=0.75,inner sep=3pt]
\tikzstyle{pointgrow dag}=[shape=cornercopoint,kpoint common,scale=0.75,inner sep=3pt]

\tikzstyle{wide copoint}=[fill=white,draw,shape=isosceles triangle,shape border rotate=90,isosceles triangle stretches=true,inner sep=0pt,minimum width=1.5cm,minimum height=6.12mm]
\tikzstyle{wide point}=[fill=white,draw,shape=isosceles triangle,shape border rotate=-90,isosceles triangle stretches=true,inner sep=0pt,minimum width=1.5cm,minimum height=6.12mm,yshift=-0.0mm]
\tikzstyle{wide point plus}=[fill=white,draw,shape=isosceles triangle,shape border rotate=-90,isosceles triangle stretches=true,inner sep=0pt,minimum width=1.74cm,minimum height=7mm,yshift=-0.0mm]

\tikzstyle{wide dpoint}=[fill=white,doubled,draw,shape=isosceles triangle,shape border rotate=-90,isosceles triangle stretches=true,inner sep=0pt,minimum width=1.5cm,minimum height=6.12mm,yshift=-0.0mm]

\tikzstyle{tinypoint}=[regular polygon,regular polygon sides=3,draw,scale=0.55,inner sep=-0.15pt,minimum width=6mm,fill=white,regular polygon rotate=180]

\tikzstyle{white point}=[point]
\tikzstyle{white dpoint}=[dpoint]
\tikzstyle{green point}=[white point] 
\tikzstyle{white copoint}=[copoint]
\tikzstyle{gray point}=[point,fill=gray!40!white]
\tikzstyle{gray dpoint}=[gray point,doubled]
\tikzstyle{red point}=[gray point] 
\tikzstyle{gray copoint}=[copoint,fill=gray!40!white]
\tikzstyle{gray dcopoint}=[gray copoint,doubled]

\tikzstyle{white point guide}=[regular polygon,regular polygon sides=3,font=\scriptsize,draw,scale=0.65,inner sep=-0.5pt,minimum width=9mm,fill=white,regular polygon rotate=180]

\tikzstyle{black point}=[point,fill=black,font=\color{white}]
\tikzstyle{black copoint}=[copoint,fill=black,font=\color{white}]

\tikzstyle{tiny gray point}=[tinypoint,fill=gray!40!white]

\tikzstyle{diredge}=[->]
\tikzstyle{ddiredge}=[<->]
\tikzstyle{rdiredge}=[<-]
\tikzstyle{thickdiredge}=[->, very thick]
\tikzstyle{pointer edge}=[->,very thick,gray]
\tikzstyle{pointer edge part}=[very thick,gray]
\tikzstyle{dashed edge}=[dashed]
\tikzstyle{thick dashed edge}=[very thick,dashed]
\tikzstyle{thick gray dashed edge}=[thick dashed edge,gray!40]
\tikzstyle{thick map edge}=[very thick,|->]

\makeatletter
\newcommand{\boxshape}[3]{%
\pgfdeclareshape{#1}{
\inheritsavedanchors[from=rectangle] 
\inheritanchorborder[from=rectangle]
\inheritanchor[from=rectangle]{center}
\inheritanchor[from=rectangle]{north}
\inheritanchor[from=rectangle]{south}
\inheritanchor[from=rectangle]{west}
\inheritanchor[from=rectangle]{east}
\backgroundpath{
\southwest \pgf@xa=\pgf@x \pgf@ya=\pgf@y
\northeast \pgf@xb=\pgf@x \pgf@yb=\pgf@y

\@tempdima=#2
\@tempdimb=#3

\pgfpathmoveto{\pgfpoint{\pgf@xa - 5pt + \@tempdima}{\pgf@ya}}
\pgfpathlineto{\pgfpoint{\pgf@xa - 5pt - \@tempdima}{\pgf@yb}}
\pgfpathlineto{\pgfpoint{\pgf@xb + 5pt + \@tempdimb}{\pgf@yb}}
\pgfpathlineto{\pgfpoint{\pgf@xb + 5pt - \@tempdimb}{\pgf@ya}}
\pgfpathlineto{\pgfpoint{\pgf@xa - 5pt + \@tempdima}{\pgf@ya}}
\pgfpathclose
}
}}

\boxshape{NEbox}{0pt}{3pt}
\boxshape{SEbox}{0pt}{-3pt}
\boxshape{NWbox}{5pt}{0pt}
\boxshape{SWbox}{-5pt}{0pt}
\boxshape{EBox}{-3pt}{3pt}
\boxshape{WBox}{3pt}{-3pt}
\makeatother

\tikzstyle{cloud}=[shape=cloud,draw,minimum width=1.5cm,minimum height=1.5cm]

\tikzstyle{map}=[draw,shape=NEbox,inner sep=1pt,minimum height=4mm,fill=white]
\tikzstyle{dashedmap}=[draw,dashed,shape=NEbox,inner sep=2pt,minimum height=6mm,fill=white]
\tikzstyle{mapdag}=[draw,shape=SEbox,inner sep=1pt,minimum height=4mm,fill=white]
\tikzstyle{mapadj}=[draw,shape=SEbox,inner sep=2pt,minimum height=6mm,fill=white]
\tikzstyle{maptrans}=[draw,shape=SWbox,inner sep=2pt,minimum height=6mm,fill=white]
\tikzstyle{mapconj}=[draw,shape=NWbox,inner sep=2pt,minimum height=6mm,fill=white]

\tikzstyle{medium map}=[draw,shape=NEbox,inner sep=2pt,minimum height=6mm,fill=white,minimum width=7mm]
\tikzstyle{medium map dag}=[draw,shape=SEbox,inner sep=2pt,minimum height=6mm,fill=white,minimum width=7mm]
\tikzstyle{medium map adj}=[draw,shape=SEbox,inner sep=2pt,minimum height=6mm,fill=white,minimum width=7mm]
\tikzstyle{medium map trans}=[draw,shape=SWbox,inner sep=2pt,minimum height=6mm,fill=white,minimum width=7mm]
\tikzstyle{medium map conj}=[draw,shape=NWbox,inner sep=2pt,minimum height=6mm,fill=white,minimum width=7mm]
\tikzstyle{semilarge map}=[draw,shape=NEbox,inner sep=2pt,minimum height=6mm,fill=white,minimum width=9.5mm]
\tikzstyle{semilarge map trans}=[draw,shape=SWbox,inner sep=2pt,minimum height=6mm,fill=white,minimum width=9.5mm]
\tikzstyle{semilarge map adj}=[draw,shape=SEbox,inner sep=2pt,minimum height=6mm,fill=white,minimum width=9.5mm]
\tikzstyle{semilarge map dag}=[draw,shape=SEbox,inner sep=2pt,minimum height=6mm,fill=white,minimum width=9.5mm]
\tikzstyle{semilarge map conj}=[draw,shape=NWbox,inner sep=2pt,minimum height=6mm,fill=white,minimum width=9.5mm]
\tikzstyle{large map}=[draw,shape=NEbox,inner sep=2pt,minimum height=6mm,fill=white,minimum width=12mm]
\tikzstyle{large map conj}=[draw,shape=NWbox,inner sep=2pt,minimum height=6mm,fill=white,minimum width=12mm]
\tikzstyle{very large map}=[draw,shape=NEbox,inner sep=2pt,minimum height=6mm,fill=white,minimum width=17mm]

\tikzstyle{medium dmap}=[draw,doubled,shape=NEbox,inner sep=2pt,minimum height=6mm,fill=white,minimum width=7mm]
\tikzstyle{medium dmap dag}=[draw,doubled,shape=SEbox,inner sep=2pt,minimum height=6mm,fill=white,minimum width=7mm]
\tikzstyle{medium dmap adj}=[draw,doubled,shape=SEbox,inner sep=2pt,minimum height=6mm,fill=white,minimum width=7mm]
\tikzstyle{medium dmap trans}=[draw,doubled,shape=SWbox,inner sep=2pt,minimum height=6mm,fill=white,minimum width=7mm]
\tikzstyle{medium dmap conj}=[draw,doubled,shape=NWbox,inner sep=2pt,minimum height=6mm,fill=white,minimum width=7mm]
\tikzstyle{semilarge dmap}=[draw,doubled,shape=NEbox,inner sep=2pt,minimum height=6mm,fill=white,minimum width=9.5mm]
\tikzstyle{semilarge dmap trans}=[draw,doubled,shape=SWbox,inner sep=2pt,minimum height=6mm,fill=white,minimum width=9.5mm]
\tikzstyle{semilarge dmap adj}=[draw,doubled,shape=SEbox,inner sep=2pt,minimum height=6mm,fill=white,minimum width=9.5mm]
\tikzstyle{semilarge dmap dag}=[draw,doubled,shape=SEbox,inner sep=2pt,minimum height=6mm,fill=white,minimum width=9.5mm]
\tikzstyle{semilarge dmap conj}=[draw,doubled,shape=NWbox,inner sep=2pt,minimum height=6mm,fill=white,minimum width=9.5mm]
\tikzstyle{large dmap}=[draw,doubled,shape=NEbox,inner sep=2pt,minimum height=6mm,fill=white,minimum width=12mm]
\tikzstyle{large dmap conj}=[draw,doubled,shape=NWbox,inner sep=2pt,minimum height=6mm,fill=white,minimum width=12mm]
\tikzstyle{large dmap trans}=[draw,doubled,shape=SWbox,inner sep=2pt,minimum height=6mm,fill=white,minimum width=12mm]
\tikzstyle{large dmap adj}=[draw,doubled,shape=SEbox,inner sep=2pt,minimum height=6mm,fill=white,minimum width=12mm]
\tikzstyle{large dmap dag}=[draw,doubled,shape=SEbox,inner sep=2pt,minimum height=6mm,fill=white,minimum width=12mm]
\tikzstyle{very large dmap}=[draw,doubled,shape=NEbox,inner sep=2pt,minimum height=6mm,fill=white,minimum width=19.5mm]

\tikzstyle{muxbox}=[draw,shape=rectangle,minimum height=3mm,minimum width=3mm,fill=white]
\tikzstyle{dmuxbox}=[muxbox,doubled]

\tikzstyle{box}=[draw,shape=rectangle,inner sep=2pt,minimum height=6mm,minimum width=6mm,fill=white]
\tikzstyle{dbox}=[draw,doubled,shape=rectangle,inner sep=2pt,minimum height=6mm,minimum width=6mm,fill=white]
\tikzstyle{dmap}=[draw,doubled,shape=NEbox,inner sep=2pt,minimum height=6mm,fill=white]
\tikzstyle{dmapdag}=[draw,doubled,shape=SEbox,inner sep=2pt,minimum height=6mm,fill=white]
\tikzstyle{dmapadj}=[draw,doubled,shape=SEbox,inner sep=2pt,minimum height=6mm,fill=white]
\tikzstyle{dmaptrans}=[draw,doubled,shape=SWbox,inner sep=2pt,minimum height=6mm,fill=white]
\tikzstyle{dmapconj}=[draw,doubled,shape=NWbox,inner sep=2pt,minimum height=6mm,fill=white]

\tikzstyle{ddmap}=[draw,doubled,dashed,shape=NEbox,inner sep=2pt,minimum height=6mm,fill=white]
\tikzstyle{ddmapdag}=[draw,doubled,dashed,shape=SEbox,inner sep=2pt,minimum height=6mm,fill=white]
\tikzstyle{ddmapadj}=[draw,doubled,dashed,shape=SEbox,inner sep=2pt,minimum height=6mm,fill=white]
\tikzstyle{ddmaptrans}=[draw,doubled,dashed,shape=SWbox,inner sep=2pt,minimum height=6mm,fill=white]
\tikzstyle{ddmapconj}=[draw,doubled,dashed,shape=NWbox,inner sep=2pt,minimum height=6mm,fill=white]

\boxshape{sNEbox}{0pt}{3pt}
\boxshape{sSEbox}{0pt}{-3pt}
\boxshape{sNWbox}{3pt}{0pt}
\boxshape{sSWbox}{-3pt}{0pt}
\tikzstyle{smap}=[draw,shape=sNEbox,fill=white]
\tikzstyle{smapdag}=[draw,shape=sSEbox,fill=white]
\tikzstyle{smapadj}=[draw,shape=sSEbox,fill=white]
\tikzstyle{smaptrans}=[draw,shape=sSWbox,fill=white]
\tikzstyle{smapconj}=[draw,shape=sNWbox,fill=white]

\tikzstyle{dsmap}=[draw,dashed,shape=sNEbox,fill=white]
\tikzstyle{dsmapdag}=[draw,dashed,shape=sSEbox,fill=white]
\tikzstyle{dsmaptrans}=[draw,dashed,shape=sSWbox,fill=white]
\tikzstyle{dsmapconj}=[draw,dashed,shape=sNWbox,fill=white]

\boxshape{mNEbox}{0pt}{10pt}
\boxshape{mSEbox}{0pt}{-10pt}
\boxshape{mNWbox}{10pt}{0pt}
\boxshape{mSWbox}{-10pt}{0pt}
\tikzstyle{mmap}=[draw,shape=mNEbox]
\tikzstyle{mmapdag}=[draw,shape=mSEbox]
\tikzstyle{mmaptrans}=[draw,shape=mSWbox]
\tikzstyle{mmapconj}=[draw,shape=mNWbox]

\tikzstyle{mmapgray}=[draw,fill=gray!40!white,shape=mNEbox]
\tikzstyle{smapgray}=[draw,fill=gray!40!white,shape=sNEbox]

\makeatletter

\pgfdeclareshape{cornerpoint}{
\inheritsavedanchors[from=rectangle] 
\inheritanchorborder[from=rectangle]
\inheritanchor[from=rectangle]{center}
\inheritanchor[from=rectangle]{north}
\inheritanchor[from=rectangle]{south}
\inheritanchor[from=rectangle]{west}
\inheritanchor[from=rectangle]{east}
\backgroundpath{
\southwest \pgf@xa=\pgf@x \pgf@ya=\pgf@y
\northeast \pgf@xb=\pgf@x \pgf@yb=\pgf@y

\pgfmathsetmacro{\pgf@shorten@left}{\pgfkeysvalueof{/tikz/shorten left}}
\pgfmathsetmacro{\pgf@shorten@right}{\pgfkeysvalueof{/tikz/shorten right}}

\pgfpathmoveto{\pgfpoint{0.5 * (\pgf@xa + \pgf@xb)}{\pgf@ya - 5pt}}
\pgfpathlineto{\pgfpoint{\pgf@xa - 8pt + \pgf@shorten@left}{\pgf@yb - 1.5 * \pgf@shorten@left}}
\pgfpathlineto{\pgfpoint{\pgf@xa - 8pt + \pgf@shorten@left}{\pgf@yb}}
\pgfpathlineto{\pgfpoint{\pgf@xb + 8pt - \pgf@shorten@right}{\pgf@yb}}
\pgfpathlineto{\pgfpoint{\pgf@xb + 8pt - \pgf@shorten@right}{\pgf@yb - 1.5 * \pgf@shorten@right}}
\pgfpathclose
}
}

\pgfdeclareshape{cornercopoint}{
\inheritsavedanchors[from=rectangle] 
\inheritanchorborder[from=rectangle]
\inheritanchor[from=rectangle]{center}
\inheritanchor[from=rectangle]{north}
\inheritanchor[from=rectangle]{south}
\inheritanchor[from=rectangle]{west}
\inheritanchor[from=rectangle]{east}
\backgroundpath{
\southwest \pgf@xa=\pgf@x \pgf@ya=\pgf@y
\northeast \pgf@xb=\pgf@x \pgf@yb=\pgf@y

\pgfmathsetmacro{\pgf@shorten@left}{\pgfkeysvalueof{/tikz/shorten left}}
\pgfmathsetmacro{\pgf@shorten@right}{\pgfkeysvalueof{/tikz/shorten right}}

\pgfpathmoveto{\pgfpoint{0.5 * (\pgf@xa + \pgf@xb)}{\pgf@yb + 5pt}}
\pgfpathlineto{\pgfpoint{\pgf@xa - 8pt + \pgf@shorten@left}{\pgf@ya + 1.5 * \pgf@shorten@left}}
\pgfpathlineto{\pgfpoint{\pgf@xa - 8pt + \pgf@shorten@left}{\pgf@ya}}
\pgfpathlineto{\pgfpoint{\pgf@xb + 8pt - \pgf@shorten@right}{\pgf@ya}}
\pgfpathlineto{\pgfpoint{\pgf@xb + 8pt - \pgf@shorten@right}{\pgf@ya + 1.5 * \pgf@shorten@right}}
\pgfpathclose
}
}

\makeatother

\pgfkeyssetvalue{/tikz/shorten left}{0pt}
\pgfkeyssetvalue{/tikz/shorten right}{0pt}

\tikzstyle{kpoint common}=[draw,fill=white,inner sep=1pt,minimum height=4mm]
\tikzstyle{kpoint sc}=[shape=cornerpoint,kpoint common]
\tikzstyle{kpoint adjoint sc}=[shape=cornercopoint,kpoint common]
\tikzstyle{kpoint}=[shape=cornerpoint,shorten left=5pt,kpoint common]
\tikzstyle{kpoint adjoint}=[shape=cornercopoint,shorten left=5pt,kpoint common]
\tikzstyle{kpoint conjugate}=[shape=cornerpoint,shorten right=5pt,kpoint common]
\tikzstyle{kpoint transpose}=[shape=cornercopoint,shorten right=5pt,kpoint common]
\tikzstyle{kpoint symm}=[shape=cornerpoint,shorten left=5pt,shorten right=5pt,kpoint common]

\tikzstyle{wide kpoint sc}=[shape=cornerpoint,kpoint common, minimum width=1 cm]
\tikzstyle{wide kpointdag sc}=[shape=cornercopoint,kpoint common, minimum width=1 cm]

\tikzstyle{black kpoint}=[shape=cornerpoint,shorten left=5pt,kpoint common,fill=black,font=\color{white}]

\tikzstyle{black kpoint sm}=[shape=cornerpoint,shorten left=5pt,kpoint common,fill=black,font=\color{white},scale=0.75]

\tikzstyle{black kpoint adjoint}=[shape=cornercopoint,shorten left=5pt,kpoint common,fill=black,font=\color{white}]
\tikzstyle{black kpointadj}=[shape=cornercopoint,shorten left=5pt,kpoint common,fill=black,font=\color{white}]

\tikzstyle{black kpointadj sm}=[shape=cornercopoint,shorten left=5pt,kpoint common,fill=black,font=\color{white},scale=0.75]

\tikzstyle{black dkpoint}=[shape=cornerpoint,shorten left=5pt,kpoint common,fill=black, doubled,font=\color{white}]
\tikzstyle{black dkpoint adjoint}=[shape=cornercopoint,shorten left=5pt,kpoint common,fill=black, doubled,font=\color{white}]
\tikzstyle{black dkpointadj}=[shape=cornercopoint,shorten left=5pt,kpoint common,fill=black, doubled,font=\color{white}]

\tikzstyle{black dkpoint sm}=[shape=cornerpoint,shorten left=5pt,kpoint common,fill=black, doubled,font=\color{white},scale=0.75]
\tikzstyle{black dkpointadj sm}=[shape=cornercopoint,shorten left=5pt,kpoint common,fill=black, doubled,font=\color{white},scale=0.75]

\tikzstyle{kpointdag}=[kpoint adjoint]
\tikzstyle{kpointadj}=[kpoint adjoint]
\tikzstyle{kpointconj}=[kpoint conjugate]
\tikzstyle{kpointtrans}=[kpoint transpose]

\tikzstyle{big kpoint}=[kpoint, minimum width=1.2 cm, minimum height=8mm, inner sep=4pt, text depth=3mm]

\tikzstyle{wide kpoint}=[kpoint, minimum width=1 cm, inner sep=2pt]
\tikzstyle{wide kpointdag}=[kpointdag, minimum width=1 cm, inner sep=2pt]
\tikzstyle{wide kpointconj}=[kpointconj, minimum width=1 cm, inner sep=2pt]
\tikzstyle{wide kpointtrans}=[kpointtrans, minimum width=1 cm, inner sep=2pt]

\tikzstyle{wider kpoint}=[kpoint, minimum width=1.25 cm, inner sep=2pt]
\tikzstyle{wider kpointdag}=[kpointdag, minimum width=1.25 cm, inner sep=2pt]
\tikzstyle{wider kpointconj}=[kpointconj, minimum width=1.25 cm, inner sep=2pt]
\tikzstyle{wider kpointtrans}=[kpointtrans, minimum width=1.25 cm, inner sep=2pt]

\tikzstyle{gray kpoint}=[kpoint,fill=gray!50!white]
\tikzstyle{gray kpointdag}=[kpointdag,fill=gray!50!white]
\tikzstyle{gray kpointadj}=[kpointadj,fill=gray!50!white]
\tikzstyle{gray kpointconj}=[kpointconj,fill=gray!50!white]
\tikzstyle{gray kpointtrans}=[kpointtrans,fill=gray!50!white]

\tikzstyle{gray dkpoint}=[kpoint,fill=gray!50!white,doubled]
\tikzstyle{gray dkpointdag}=[kpointdag,fill=gray!50!white,doubled]
\tikzstyle{gray dkpointadj}=[kpointadj,fill=gray!50!white,doubled]
\tikzstyle{gray dkpointconj}=[kpointconj,fill=gray!50!white,doubled]
\tikzstyle{gray dkpointtrans}=[kpointtrans,fill=gray!50!white,doubled]

\tikzstyle{white label}=[draw,fill=white,rectangle,inner sep=0.7 mm]
\tikzstyle{gray label}=[draw,fill=gray!50!white,rectangle,inner sep=0.7 mm]
\tikzstyle{black label}=[draw,fill=black,rectangle,inner sep=0.7 mm]

\tikzstyle{dkpoint}=[kpoint,doubled]
\tikzstyle{wide dkpoint}=[wide kpoint,doubled]
\tikzstyle{dkpointdag}=[kpoint adjoint,doubled]
\tikzstyle{wide dkpointdag}=[wide kpointdag,doubled]
\tikzstyle{dkcopoint}=[kpoint adjoint,doubled]
\tikzstyle{dkpointadj}=[kpoint adjoint,doubled]
\tikzstyle{dkpointconj}=[kpoint conjugate,doubled]
\tikzstyle{dkpointtrans}=[kpoint transpose,doubled]

\tikzstyle{kscalar}=[kpoint common, shape=EBox, inner xsep=-1pt, inner ysep=3pt,font=\small]
\tikzstyle{kscalarconj}=[kpoint common, shape=WBox, inner xsep=-1pt, inner ysep=3pt,font=\small]

\tikzstyle{spekpoint}=[kpoint sc,minimum height=5mm,inner sep=3pt]
\tikzstyle{spekcopoint}=[kpoint adjoint sc,minimum height=5mm,inner sep=3pt]

\tikzstyle{dspekpoint}=[spekpoint,doubled]
\tikzstyle{dspekcopoint}=[spekcopoint,doubled]

 \tikzstyle{upgroundsmall}=[circuit ee IEC,thick,ground,rotate=90,scale=1.25,xshift=-1mm]
 \tikzstyle{upground}=[circuit ee IEC,thick,ground,rotate=90,scale=2.5]
 \tikzstyle{downground}=[circuit ee IEC,thick,ground,rotate=-90,scale=2.5]
 \tikzstyle{bigground}=[regular polygon,regular polygon sides=3,draw=gray,scale=0.50,inner sep=-0.5pt,minimum width=10mm,fill=gray]

\tikzstyle{arrs}=[-latex,font=\small,auto]
\tikzstyle{arrow plain}=[arrs]
\tikzstyle{arrow dashed}=[dashed,arrs]
\tikzstyle{arrow bold}=[very thick,arrs]
\tikzstyle{arrow hide}=[draw=white!0,-]
\tikzstyle{arrow reverse}=[latex-]
\tikzstyle{cdnode}=[]

\let\olddagger\dagger
\renewcommand{\dagger}{\ensuremath{\olddagger}\xspace}

\usepackage{makeidx}
\makeindex

\theoremstyle{definition}
\newtheorem{theorem}{Theorem}[section]

\newtheorem{lemma}[theorem]{Lemma}

\newtheorem{definition}[theorem]{Definition}
\newtheorem{example}[theorem]{Example}

\newcommand{\TODO}[1]{\marginpar{\scriptsize\bB \textbf{TODO:} #1\e}}

\newcommand{\TODOa}[1]{\marginpar{\scriptsize\bM \textbf{TODO:} #1\e}}
\newcommand{\TODOb}[1]{\marginpar{\scriptsize\bB \textbf{TODO:} #1\e}}

\newcommand{\COMMa}[1]{\marginpar{\scriptsize\bM \textbf{COMM:} #1\e}}
\newcommand{\COMMb}[1]{\marginpar{\scriptsize\bB \textbf{COMM:} #1\e}}

\newcommand{\CHECK}[1]{\marginpar{\scriptsize\bR \textbf{CHECK:} #1\e}}

\usepackage[color]{changebar}

\usepackage{color}
\def\bR{\begin{color}{red}}
\def\bB{\begin{color}{blue}}
\def\bM{\begin{color}{magenta}}
\def\bC{\begin{color}{cyan}}
\def\bW{\begin{color}{white}}
\def\bBl{\begin{color}{black}}
\def\bG{\begin{color}{green}}
\def\bY{\begin{color}{yellow}}
\def\e{\end{color}\xspace}
\newcommand{\bit}{\setlist{nolistsep}\begin{itemize}[noitemsep]}
\newcommand{\eit}{\end{itemize}\par\noindent}
\newcommand{\ben}{\setlist{nolistsep}
\begin{enumerate}[noitemsep]
}
\newcommand{\een}{\end{enumerate}\par\noindent}
\newcommand{\beq}{\begin{equation}}
\newcommand{\eeq}{\end{equation}\par\noindent}
\newcommand{\beqa}{\begin{eqnarray*}}
\newcommand{\eeqa}{\end{eqnarray*}\par\noindent}
\newcommand{\beqn}{\begin{eqnarray}}
\newcommand{\eeqn}{\end{eqnarray}\par\noindent}

\if\lecturenotes1

\renewcommand{\TODO}[1]{}
\renewcommand{\TODOa}[1]{}
\renewcommand{\TODOb}[1]{}
\renewcommand{\COMMa}[1]{}
\renewcommand{\COMMb}[1]{}
\renewcommand{\CHECK}[1]{}

\def\bR{\begin{color}{black}}
\def\bB{\begin{color}{black}}
\def\bM{\begin{color}{black}}
\def\bC{\begin{color}{black}}
\def\bW{\begin{color}{black}}
\def\bG{\begin{color}{black}}
\def\bY{\begin{color}{black}}

\fi

\newcommand{\gro}{
    \begin{tikzpicture}[scale=1, baseline=-1mm, thick]
    \draw (0,0) -- +(4.0mm,0mm) {
        +(4mm,-2.5mm) -- +(4mm,2.5mm)
                +(5mm,-1.5mm) -- +(5mm,1.5mm)
                +(6mm,-0.5mm) -- +(6mm,0.5mm)};
        \end{tikzpicture}}
        
\newcommand{\groy}{
    \begin{tikzpicture}[scale=1, baseline=0mm, thick]
    \draw (0,0) -- +(0mm,4mm) {
        +(-2.5mm,4mm) -- +(2.5mm,4mm)
                +(-1.5mm,5mm) -- +(1.5mm,5mm)
                +(-0.5mm,6mm) -- +(0.5mm,6mm)};
        \end{tikzpicture}}
\allowdisplaybreaks
\usepackage{tikzfig}

\usepackage{enumitem}
\usepackage{soul}

\begin{document}

\title{Generalised Process Theories}
\date{\today}

\author{John H. Selby\textsuperscript{*}}
\affiliation{International Centre for Theory of Quantum Technologies, University of Gda\'nsk, 80-309 Gda\'nsk, Poland}
\affiliation{Theoretical Sciences Visiting Program, Okinawa Institute of Science and Technology Graduate University, Onna, 904-0495, Japan}
\author{Maria E. Stasinou\textsuperscript{*}}
\affiliation{Deutsches Elektronen-Synchrotron DESY, Germany}
\author{Matt Wilson}
\affiliation{Universit\'e Paris-Saclay, CNRS, CentraleSup\'elec, ENS Paris-Saclay, Inria, Laboratoire M\'ethodes Formelles, 91190, Gif-sur-Yvette, France}
\author{Bob Coecke}
\affiliation{Quantinuum, 17 Beaumont street, Oxford, OX1 2NA, United Kingdom}

\maketitle

\begin{center}
   \textsuperscript{*}{\textbf{\textcolor{black!85}{The first two authors contributed equally to this work and are listed in alphabetical order.}}}
\end{center}

\begin{abstract}
Process theories have been widely applied, from the foundations of physics through to computational linguistics. They are conceptually based on the idea that we can describe all of these things in terms of \emph{systems} which interact and evolve via \emph{processes}, and that we can explore their behaviour  by considering how these systems and processes \emph{compose}. On a formal level, they are based on the mathematics of symmetric monoidal categories which is known to be broadly applicable within many branches of mathematics. 
There are, however, situations where the formal notion of a process theory does not seem applicable, but where the conceptual idea of a process theory still does. That is, where we still have systems, processes, and their composition, but they do not neatly fit the mould of a symmetric monoidal category. 

In this paper we discuss whether the conceptual idea of a process theory might in fact be better formally
understood in terms of operad algebras. In particular, this view hinges on the works of Ref.~\cite{patterson2021wiring,yau2018operads} in which the authors demonstrate a close connection between symmetric monoidal categories (along with some of their variants) and certain kinds of operad algebras. We build on these previous works by identifying an adaptation needed to recover those symmetric monoidal categories which are causal, and demonstrate the utility of the operad algebra perspective by showing how it can be easily adapted to subsume alternative kinds of process theories such as those which are time-neutral, enriched, or higher-order.

We present all of this using a convenient string-diagrammatic notation, such that this paper should be readily accessible to anyone who has some familiarity with string diagrams in other contexts. That is, we use string diagrams both for the new kinds of process theories that we introduce, as well as for the operads and operad algebras underpinning them.
 \end{abstract}

\tableofcontents

\section{Introduction} 

Process theories are an extremely versatile formalism for describing many situations of interest in many fields of research. They are comprised of a collection of systems and the processes that act on them. A few examples are:

\begin{center}
\begin{tabular}{c|c|c}
\textbf{Research field} & \textbf{Systems} & \textbf{Processes} \\ \hline
Physics & Physical systems & Physical transformations \\
Chemistry & Molecules & Chemical reactions \\
Mathematics & Groups & Group homomorphisms \\
Computer science & Data & Programs \\
Linguistics & Grammatical types & Words \\
... & ... & ... 
\end{tabular}
\end{center}

Whilst the situations of interest in the above table are extremely diverse, they nonetheless share common features. For example, if we have two systems then we can consider them together as a composite system, and if we can process one system into a second, and then the second into a third, then we can consider this as describing an overall process going from the first to the third. It is precisely these common features which are formalised by the definition of a process theory, which here we will call a \emph{traditional} process theory to disambiguate from the other sorts which we will introduce shortly.

\begin{definition}[Traditional process theories]\label{sec:PT}
A traditional process theory is comprised of a collection of \emph{processes}. For example,
\beq\input{figures/process.tikz},\eeq
is a process with \emph{input systems} $A$ and $B$ and \emph{output systems} $A$, $C$, and $C$. This collection of processes must, moreover, be closed under \emph{wirings}. For example, the diagram:
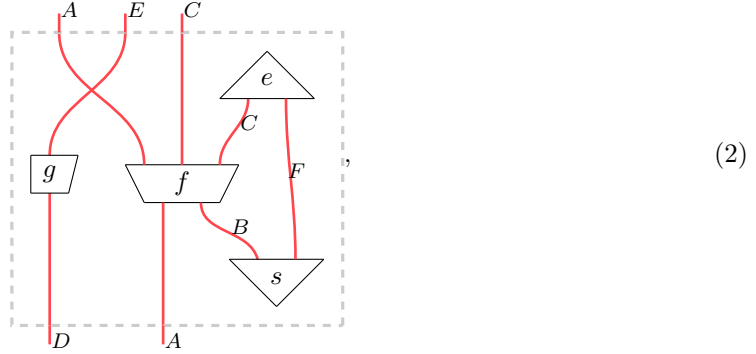
\beq\input{figures/diagram.tikz},\eeq
corresponds to another process in the theory, in this case one with inputs $D$ and $A$ and outputs $A$, $E$, and $C$.
This wiring is subject to the constraints that:
\bit
\item[--] outputs cannot be wired to outputs, and inputs cannot be wired to inputs;
\item[--] the wiring is acyclic;
\item[--] system types match.
\eit
Moreover, we have a notion of diagrammatic equality -- two diagrams are equal if they represent the same wiring, for example,
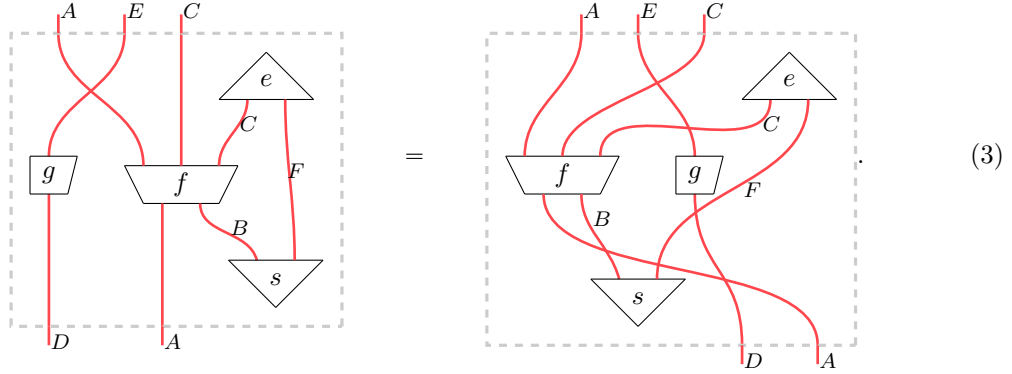
\beq\label{eq:equalDiag}\input{figures/diagram.tikz}\qquad=\qquad\input{figures/diagramRedrawn.tikz}.\eeq
In other words, the precise layout of processes on the page is not important, all that matters is their \emph{connectivity}.
\end{definition}

We've argued that the above definition captures an extremely broad range of theories of interest. Indeed, in the context of the foundations of physics it is known that process theories subsume many other frameworks for describing hypothetical theories such as generalised probabilistic theories \cite{HardyAxiom,Barrett}, operational probabilistic theories \cite{Chiri1,d2017quantum}, epistemically restricted theories \cite{SpekToy,spekkens2016quasi}, ontological theories \cite{harrigan1,schmid2020structure}, and so on. One may therefore be surprised by the ``generalised'' in the title. Why would we need to go to anything even more general than what we have already?

Well, it turns out that in recent years there have been several kinds of theories which have been considered which don't naturally lend themselves to being viewed as process theories. This is not to say that they cannot be shoehorned into the formalism, but that by doing so one would necessarily lose something of the essence of what the theory is about.

\

Let us illustrate this with a simple motivating example.
Suppose that we want to make a fairly subtle modification to our notion of a process theory, such that processes no longer have an input-output distinction. The study of such theories is motivated by the work of  Refs.~\cite{oreshkov2016operational,timesymmetry} which defines a time-neutral version of quantum theory. We will therefore call such process theories time-neutral.
\begin{definition}[Time-neutral process theories]
 A time-neutral theory is comprised of a collection of processes, such as
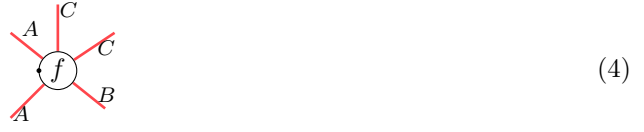
\beq
\input{figures/timeNeutralProcess.tikz}
\eeq
which is closed under wirings. For example,
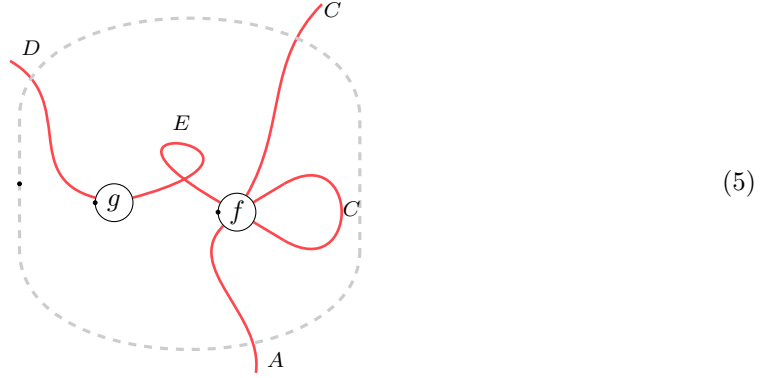
\beq
\input{figures/timeNeutralDiagramNewNotation.tikz}
\eeq
corresponds to anther process in the theory. Moreover, two diagrams are equal if they have the same connectivity.
\end{definition}

One can indeed shoehorn this into the mould of a traditional process theory, namely one which has special kinds of bipartite states and effects known as cups and caps \cite{Kindergarten,CKbook}. Indeed, we will show the formal relationship between these in Sec.~\ref{sec:TimeNeutral}. Cups and caps correspond to particular processes which exist in certain process theories and can be conveniently represented as bendings of wires. Specifically, a \emph{cup}, $\bigcup$, and a \emph{cap}, $\bigcap$ can be represented as:
\beq\input{figures/cupProcess.tikz}\ \ =:\ \ \input{figures/cupWire.tikz} \quad\qquad \& \qquad \quad \input{figures/capProcess.tikz}\ \ =:\ \ \input{figures/capWire.tikz}\ .\eeq
For this representation to make sense, that is, such that we can maintain our rule that only the connectivity of the diagram matters, these must satisfy certain equations, namely:
\beq\label{eq:snake}\input{figures/snake.tikz}\quad\qquad\&\quad\qquad\input{figures/cupSymmetry.tikz}\ .\eeq
Essentially, these cups and caps let us turn an input into an output and vice versa, thereby blurring the distinction between inputs and outputs. For example, diagrams such as
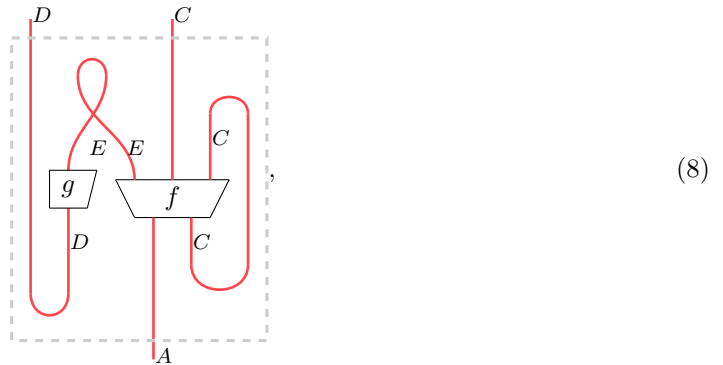
\beq\input{figures/timeNeutralDiagram.tikz},\eeq
are now legitimate diagrams even though they have an apparent cyclicity. 

Intuitively at least, one can therefore see that there is a connection between traditional process theories with cups and caps and the time-neutral process theories that we have already introduced.
A formalisation of this intuitive connection (as is provided Sec.~\ref{sec:TimeNeutral}) means that we don't strictly require going beyond traditional process theories in order to capture time-neutrality. 
This perspective, however, is rather unnatural and certainly loses something of the essence of a time-neutral theory -- it first requires that we introduce an artificial partitioning of systems into inputs and outputs, and then requires us to add in the cups and caps in order to effectively remove the distinction again! 

\

Similar issues can be seen to arise when trying to formulate theories of higher-order processes, i.e., processes with holes into which we can wire in other processes, drawn as:
\beq
\input{new_supermap_1.tikz}.
\eeq
Such generalised processes have been widely studied in the recent literature \cite{chiribella_architecture, chiribella_networks, Chiribella_2008_supermaps, chiribella_causal} for the purpose of studying quantum causal structures. The study of the compositionally of higher order processes has significantly stretched the standard process-theoretic framework in a number of independent directions \cite{kissinger2017categorical, bisio2019theoretical, simmons_bv_logic, wilson_local, wilson_polycat, hefford_coend, earnshaw_produoidal, vanreitvelde_consistent_circuit, boisseau2022cornering}. The first instance in the literature is the development of higher-order causal categories \cite{kissinger2017categorical, simmons_bv_logic}, and higher order quantum theory \cite{bisio2019theoretical, Hoffreumon_projective} where there are at least two distinct notions of tensor product (i.e., putting processes side-by-side) for higher order processes \cite{kissinger2017categorical, simmons_bv_logic, earnshaw_produoidal}. It has been argued that a key conceptual road block to treating higher-order processes as forming standard process theories is the existence of restrictions on the ways in which they can be plugged together without the formation of pathological time loops \cite{wilson_polycat} such as: \beq
\input{time_loop.tikz}.
\eeq 
This restriction means that in some contexts it is more appropriate to model a fragment of the compositionality of higher-order processes using the notion of a \textit{single wire process theory}, referred to in the categorical literature as a (planar) polycategory \cite{szabo_polycats}.

Even further, there are more fine-grained ways in which higher order processes can be composed which are more naturally treated in terms of contractions between wires of a single globular process \cite{apadula_nosig}. In a related strand of work, higher-order processes have recently been connected to the causal box framework \cite{Portmann_2017_causal_box, venkatesh_cyclic_spacetime,salzger2024mapping}, where quantum processes follow composition rules based on contracting wires within globular boxes, ensuring that such contractions are permissible within the given spacetime structure and do not lead to causal loops.

\

Traditional process theories are commonly formulated using the mathematics of category theory \cite{MacLane,coecke2010categories} and in particular, as symmetric monoidal categories \cite{MacLaneCoherence,coecke2010categories} as we briefly review in Sec.~\ref{sec:SMCs}. In this paper we argue, however, that this is not the best mathematical language to capture generalised process theories, and that, following Refs.~\cite{spivak2013operad,patterson2021wiring}, a formulation in terms of (a minor variant of) operad algebras captures them more faithfully. In the context of traditional process theories the distinction is somewhat inconsequential. However, when it comes to generalisations (such as to the aforementioned time-neutral or higher-order process theories) the operadic formalism turns out to be a much more versatile framework.

A noteworthy point on the classes of physical theories which stretch our intuitions away from the traditional process theoretic framework, is that they are typically concerned with stretching our conception of temporal or causal structure. The notion of a prioritised direction of flow of time is typically represented in the traditional process theoretic viewpoint using causal categories \cite{CRCaucat}\footnote{Also referred to in the categorical literature as affine monoidal categories.}. These provide an abstraction of trace-preservation, where discarding the output of a process is considered equivalent to discarding its input: \beq\input{figures/check_causal_1.tikz} \quad = \quad \begin{tikzpicture}
	\begin{pgfonlayer}{nodelayer}
		\node [style=none] (33) at (0, -0.5) {};
		\node [style=none] (34) at (0, 0.25) {};
		\node [style=upground] (40) at (0, 0.5) {};
	\end{pgfonlayer}
	\begin{pgfonlayer}{edgelayer}
		\draw [qWire, in=90, out=-90] (34.center) to (33.center);
	\end{pgfonlayer}
\end{tikzpicture}
 .\eeq  Consequently, to complete the picture and have a fully working framework for generalised processes theories for physicists, we consider how to construct causal categories as operad algebras. To do so, we are led to working with a slightly generalised notion of \textit{operads with empty space}. Standard operads and their algebras can always be re-interpreted in this setting, which additionally allows us to recover the representation of causality through a discarding effect using appropriate algebras. This amendment can be understood purely graphically, and so provides a stable practical framework for physicists who are unfamiliar with category theory. 

\

To conclude then, we will motivate the generalised process theory framework by identifying the following correspondences:
\begin{center}
\begin{tabular}{|c|c|c|}
\hline
\textbf{Process Theoretic Concept} & \textbf{Formalisation as algebras for:}  \\ \hline
 &  \\ 
Traditional Process Theories &  \input{figures/operadexample_poly1.tikz}  \\ 
 &  \\ 
Time-Neutral Process Theories &  \input{figures/intro_tn_new.tikz}   \\
 &  \\ 
Single-Wire Process Theories &  \input{figures/operadexample_poly2.tikz} \\
 &  \\ 
Causal Process Theories &  \input{wiringTest3.tikz}  \\
 &  \\ 
 \hline
\end{tabular}
\end{center}
Moreover, the natural kinds of maps that one would want to consider between these different kinds of process theories, namely, those that preserve the wirings, are captured by natural transformations between the algebras.

Potential future directions include developing a more refined understanding of the compositional structure of higher-order processes, establishing a formal framework for causal boxes \cite{Portmann_2017_causal_box} and related theories grounded primarily in contractions \cite{apadula_nosig}, and exploring both resources within broader compositional settings, as well as investigating compositionality itself as a resource.. The examples that we consider in this work are primarily motivated by problems of quantum theory. However, the tools that we develop  should also be much more broadly applicable. For instance, the limitations of traditional process theories have also become apparent in the field of natural language processing \cite{wang2023distilling}, and it is likely that the more general formalism that we present here will also be useful there.

More broadly though, the success of the process theory framework can be understood as a consequence of having at the right moment in time, formalised a space of possible theories not too broad as to be unwieldy but not so restrictive as to impede exploration. As these examples show, at least in quantum foundations, we have started to hit the boundaries of this space, and so, it is the right time for an expansion. Indeed, one of the key objectives of this paper is to develop graphical techniques for reasoning about relevant classes of operad algebras so as to make them more ``user-friendly''.

\subsection{Traditional process theories as symmetric monoidal categories}
\label{sec:SMCs}

Traditional  process theories are often considered  a convenient way to describe strict symmetric monoidal categories (SMC). Let us briefly recap how this connection is made. Given some process theory, $\mathcal{P}$, we can define an associated SMC, $\mathcal{C}_{\mathcal{P}}$, as follows. First we equip our process theory with certain extra systems and processes. We introduce a fictitious \emph{trivial system}, $I$, which is the input to a process with no input and an output to a system with no output: 
 \beq
 \input{figures/trivSystem2.tikz} \ \ = \ \ \input{figures/trivSystem1.tikz} \qquad,\qquad  \input{figures/trivSystem4.tikz} \ \ = \ \ \input{figures/trivSystem3.tikz}.
 \eeq
 Moreover, we interpret bits of the wirings as processes in their own right, for example, a single wire is taken as an \emph{identity process} and a crossed wire is taken as a \emph{swap process}:
 \beq
 \input{figures/iden1.tikz} \ \ = \ \ \input{figures/iden2.tikz} \qquad , \qquad \input{figures/swap1.tikz} \ \ = \ \ \input{figures/swap2.tikz}
 \eeq
Having added on these extra bits to the process theory we are then in a position that we can define the associated SMC. We take the objects in $\mathcal{C}_{\mathcal{P}}$ to be the systems in $\mathcal{P}$ together with the trivial system $I$. We define the monoidal product, $\otimes$, for objects via:
\beq
A\otimes B := \input{figures/compSys.tikz}
\eeq
 Hence, the trivial system will be the monoidal unit as diagrammatically we have:
\beq
\input{figures/monUnit1.tikz} \ \ = \ \ \begin{tikzpicture}
	\begin{pgfonlayer}{nodelayer}
		\node [style=none] (0) at (0.5, -0.5) {};
		\node [style=none] (1) at (0.5, 0.5) {};
		\node [style=right label] (2) at (0.5, -0.25) {$A$};
	\end{pgfonlayer}
	\begin{pgfonlayer}{edgelayer}
		\draw [qWire, in=-90, out=90] (0.center) to (1.center);
	\end{pgfonlayer}
\end{tikzpicture}
 \ \ = \ \ \input{figures/monUnit3.tikz}.
\eeq

 The morphisms in $\mathcal{C}_{\mathcal{P}}$ are given by the processes in $\mathcal{P}$ together with the identity and swap processes which are taken to be the identity and swap morphisms respectively. We define composition of morphisms, $\circ$, and the monoidal product of morphisms, $\otimes$, via:
 \beq
 g\circ f :=  \input{seqComp1.tikz} \qquad \text{and} \qquad f\otimes h := \input{parComp1.tikz} \qquad \text{respectively.}
 \eeq
 Verifying that this does indeed define a valid SMC is then relatively straightforward. For example, associativity of $\circ$ is immediate from the definition:
\beq
h\circ (g \circ f) = \input{assoc.tikz} = (h\circ g) \circ f.
\eeq
 
The reason that process theories and SMCs are then often conflated is that anything we can do within the process theory $\mathcal{P}$ can equally well be done within the SMC $\mathcal{C_P}$. In particular, any diagram in the process theory can be built up out of the categorical operations, for example:
\beq\label{eq:foliate}
\input{foliatedDiag1.tikz}.
\eeq
The axioms of the SMC ensure that any two equal diagrams (such as shown in eq.~\eqref{eq:equalDiag}) correspond to the same morphism within the SMC.

This categorical perspective on process theories has proved to be extremely useful, as it provides a formal connection between process theories and a well established branch of mathematical research, allowing one to use all of the examples, results and concepts therein. However, as mentioned before, this view is not particularly natural:
\bit
\item[--] Firstly, we had to introduce a fictitious trivial system to our process theory to act as the monoidal unit for the SMC. This is unnatural, however, in that there are sensible process theories in which we have no need for the trivial system. An example is the process theory of unitary transformations on sets of qubits, which is relevant for the study of the circuit model of quantum computation. 
\item[--]  Secondly, we had to turn the wirings into processes to act as the identity and swap morphisms within the SMC. This is also unnatural in the sense that we can think of a process theory in which it does not make sense to consider these as processes (but merely as how processes are composed). An example is the process theory of noisy quantum operations, which is relevant for the study of the quantum processes that can actually be implemented in the lab.
\item[--] Thirdly, in order to convert a diagram in a process theory into composition in terms of the basic categorical operations (as in eq.~\eqref{eq:foliate}), we introduced a foliation of the diagram (the horizontal dashed lines). However, there is no unique way to introduce this (which is closely related to the lack of a notion of simultaneity in relativity). There is therefore a need to introduce equations which relate the various categorical operations in order to make sure that all of the equivalent ways of expressing a diagram in terms of these operations are equal to one another. The privileging of certain types of composition over others, which is somewhat akin to picking a basis for a vector space, it may be useful but it can obfuscate the underlying structure\footnote{This issue was first pointed out to the authors by Lucien Hardy.}.
\eit
In practice, however, none of these issues actually serves as a great impediment to viewing traditional process theories as SMCs. For example, in the SMC corresponding to unitary operations on qubits we would gain a trivial object that we would never use, but its existence isn't really a problem for us. Similarly, adding in identities and swaps as processes rather than just as composition never seems to lead to any problems in practice. The greater issue in reality is the inflexibility of the categorical formalism to capture theories with different notions of composition as we discussed above. 

To elaborate however, the correspondence between process theories and SMCs  no longer holds for time-neutral process theories. Integral to the notion of a morphism in a category, is the specification of its domain and codomain.  However, in a time neutral process theory there is no meaningful way to divide up the systems associated to a process into inputs and outputs, and thus no meaningful way to specify the domain and codomain of the morphism that we would associate to the process. One can get around this -- to some extent -- by working with compact closed SMCs, that is, the sorts of SMCs that we get from process theories with cups and caps. But, as we discussed already, this is not a particularly natural formulation of this kind of theory.

The purpose of this article is to show that there is an alternative way to connect process theories to well studied mathematics, which avoids all of the aforementioned problems. The crux of this is the recent work of Patterson et al.,  Ref.~\cite{patterson2021wiring}, in which the authors show an equivalence between symmetric monoidal categories and the algebras of the \emph{acyclic wiring operad}. This equivalence, together with the relationship we described above between process theories and SMCs, immediately means that we can also view process theories as certain kinds of operad algebras. 

\subsection{Operadic preliminaries}\label{sec:Operads}

In this section we provide a brief introduction to the operadic language by means of an intuitive graphical representation. We then discuss how process theories can be re-expressed in this language and showcase the advantages of this perspective.

Operads, much like categories, consist of a collection of objects, a collection of operations and a means of operad composition obeying the relevant associativity and identity laws. They differ from categories in the sense that their operations can have multiple or no inputs instead of a single input (as in the case of categorical morphisms).
\begin{definition}
An operad $
\cO$ consists of a collection of objects, where an object $t$ is denoted by,
\beq
\begin{tikzpicture}
	\begin{pgfonlayer}{nodelayer}
		\node [style=none] (0) at (1.25, 0) {};
		\node [style=none] (1) at (-0.75, 0) {};
		\node [style=none] (3) at (0.25, 0.25) {$t$};
	\end{pgfonlayer}
	\begin{pgfonlayer}{edgelayer}
		\draw (0.center) to (1.center);
	\end{pgfonlayer}
\end{tikzpicture}
,
\eeq
 and a set of operations with some number of objects as input and a single object as an output. We denote an operation $f$ from objects $t_1,...,t_n$ to $t$ as,
\beq
\input{figures/operadoperation.tikz}.
\eeq
These can be composed provided that types match. For example an operation $f'$ with output $t_i$ can be connected to the $i$th input of $f$ as follows:
\beq\label{eq:opComp}
\input{figures/operadcomp.tikz}.
\eeq
This composition is associative and has a unit operation given by
\beq
\input{figures/operadid2.tikz},
\eeq
which ensures that our diagrammatic representation is consistent.
\end{definition}

As inputs of operations are taken to be lists of objects.  In particular,  the empty list can be treated as an input. Diagrammatically we will draw these as operations without an input:
\beq
\begin{tikzpicture}
	\begin{pgfonlayer}{nodelayer}
		\node [style=none] (4) at (1, 0) {};
		\node [style=infpoint,fill=black!10] (6) at (-0.5, 0) {$\phi$};
		\node [style=up label] (26) at (0.75, 0) {};
	\end{pgfonlayer}
	\begin{pgfonlayer}{edgelayer}
		\draw (4.center) to (6);
	\end{pgfonlayer}
\end{tikzpicture}.
\eeq

For the purposes of this work, we will follow the interpretation of operads of Ref.~\cite{fong2018seven} as being abstract theories of composition. That is, the operations are thought of as describing different ways to compose a bunch of small things into a bigger thing. Note that in general there is a set of different operations with input $t_1,...,t_n$  and output $t$. If this is a singleton set then this means that there is a unique way to compose $t_1,...,t_n$ to make $t$, whilst if it is the empty set then there is no way to do so, and otherwise there are various different ways in which they can be composed. From this perspective, operation composition is very natural. For example, in Eq.~\ref{eq:opComp} we see that we first use $f'$ to combine $g_1,...,g_m$ to make $t_i$ and then use $f$ to combine the $t_i$ that was created together with $t_1,...,t_{i-1},t_{i+1},...,t_n$ to make some $t$. 

Let us now illustrate this abstract idea with some more concrete examples.
\begin{example}[Space-time operad]
Suppose we have some space-time manifold $\mathcal{M}$. Then we can define an operad where the objects correspond to subsets $M_i\subseteq \mathcal{M}$. For every $M_1,...,M_n$ there is a single operad operation which takes $M_1,...,M_n$ to the union of these subsets, $\bigcup_{i=1}^n M_i$. That is, we can think of this operad as describing the sticking together of spacetime regions. 
\end{example}

\begin{example}[Operads from SMCs]
An example of operads relevant to this work is the class of operads $\cO_\mathcal{C}$ that follow from symmetric monoidal categories $\mathcal{C}$. In particular, any symmetric monoidal category $\mathcal{C}$ defines an operad  $\cO_\mathcal{C}$ by  restricting $\mathcal{C}$ to having only morphisms with a single output. That said, morphisms $C_1\otimes C_2\otimes...\otimes C_n\to D$ in $\mathcal{C}$ can be viewed as operad operations from $C_1,...,C_n$ to $D$. Representative examples are the operad $\textsc{set}$ which has sets as objects and functions from the cartesian product of the input sets to the output set as operad operations  as well as the operad $\textsc{vect}_{K}$ which has vector spaces as objects and linear maps from the tensor product of the input spaces to the output space as operations.
\end{example}

Whilst we can capture abstract theories of composition using operads, we will typically be more interested in concrete instantiations of this abstract theory. In order to describe these we first introduce operad functors:
\begin{definition} An operad functor \colorbox{blue!10}{$F:\cO\to \cO'$} maps the objects $t\in \cO$ to the objects  $F(t)\in \cO'$ and the operad operations $o \in \cO$ with inputs $t_1,...,t_n$ and output $t$ to an operation $F(o)\in \cO'$ with inputs $F(t_1),...,F(t_n)$ and output $F(t)$, such that composition and identities are preserved. We diagrammatically denote these morphisms by shaded regions:
\beq
\input{figures/operadFunctor.tikz}
\eeq
 The condition that the functor preserves the operadic composition translates diagramatically as:
\beq
\input{figures/operadFunctor1.tikz}\quad = \quad
\input{figures/operadFunctor2.tikz}.
\eeq
The preservation of identity operations means that:
\beq
\input{opAlg1.tikz} \ \ = \ \ \begin{tikzpicture}
	\begin{pgfonlayer}{nodelayer}
		\node [style=none] (42) at (-1, 0) {};
		\node [style=none] (44) at (1, 0) {};
		\node [style=up label] (62) at (-0.25, 0) {$F(t)$};
	\end{pgfonlayer}
	\begin{pgfonlayer}{edgelayer}
		\draw (42.center) to (44.center);
	\end{pgfonlayer}
\end{tikzpicture}
.
\eeq
\end{definition}

For us the most important class of operad functors are operad algebras:
\begin{definition} An operad algebra for some operad $\cO$ is an operad functor \colorbox{blue!10}{$F:\cO\to \textsc{set}$}.
\end{definition}

Conceptually, again following Ref.~\cite{fong2018seven}, the functor $F$ provides a concrete instantiation of the abstract notion of composition provided by $\cO$. That is, to each object $t\in \cO$ there is some associated set $F(t)$, which we think of as describing the set of possible ways that $t$ can be. 
 Then, to each operad operation $f$ telling us how $t_1,...,t_n$ are to be combined to make $t$, $F$ assigns a function $F(f):F(t_1)\times...\times F(t_n) \to F(t)$. This function indicates for each way that $t_1,..,t_n$ can be how the output $t$ will be. Diagrammatically, we represent the elements of the set $F(t)$ (i.e. the possible ways that $t$ can be) as:
\beq
\begin{tikzpicture}
	\begin{pgfonlayer}{nodelayer}
		\node [style=none] (4) at (1, 0) {};
		\node [style=infpoint,fill=black!10] (6) at (-0.5, 0) {$s$};
		\node [style=up label] (26) at (0.75, 0) {$F(t)$};
	\end{pgfonlayer}
	\begin{pgfonlayer}{edgelayer}
		\draw (4.center) to (6);
	\end{pgfonlayer}
\end{tikzpicture}
,
\eeq
which we think of as the `state' of $t$.
Then, we represent the action of the function on these `states' via
 \beq
 \input{actionOpAlg.tikz} \ \ = \ \ \begin{tikzpicture}
	\begin{pgfonlayer}{nodelayer}
		\node [style=none] (4) at (1, 0) {};
		\node [style=infpoint,fill=black!10] (6) at (-0.5, 0) {$s$};
		\node [style=up label] (26) at (0.75, 0) {$F(t)$};
	\end{pgfonlayer}
	\begin{pgfonlayer}{edgelayer}
		\draw (4.center) to (6);
	\end{pgfonlayer}
\end{tikzpicture}

 \eeq
 where $s = F(f)(s_1\times \cdots \times s_n)$.
Functoriality of $F$ ensures that this is a sensible interpretation. For example, in the case of the ``do nothing'' operation we have that:
\beq
\input{opAlg1.tikz} \ \ = \ \ \begin{tikzpicture}
	\begin{pgfonlayer}{nodelayer}
		\node [style=none] (42) at (-1, 0) {};
		\node [style=none] (44) at (1, 0) {};
		\node [style=up label] (62) at (-0.25, 0) {$F(t)$};
	\end{pgfonlayer}
	\begin{pgfonlayer}{edgelayer}
		\draw (42.center) to (44.center);
	\end{pgfonlayer}
\end{tikzpicture}
.
\eeq
Hence, doing nothing doesn't change the state of the object just as we would expect!

\begin{example}[Fields on space-time]
In this case $F(M)$ are the possible field configurations on the space $M
\subseteq \mathcal{M}$ plus an extra ``fail'' configuration, that is $F(M):=\mathsf{FieldConfig}\big|_M \sqcup \mathtt{fail}$. 
Then $F(\cup): F(M)\times F(M') \to F(M\cup M')$ is a function that glues the state of the field together provided that they are compatible, and otherwise returns $\mathtt{fail}$.
\end{example}

Viewing operad algebras as concrete instantiations of abstract theories of composition will be central to this work. In particular, this perspective places process theories under a new light since they too will be viewed as concrete instantiations of an abstract theory of how boxes can be wired together. More specifically, the abstract theory  will have the form of a particular kind of operad, a `wiring operad', which when varying its precise definition will lead to different kinds of operad algebras and consequently different kinds of process theories.

\section{Traditional process theories}

We begin by defining the acyclic wiring operad \cite{patterson2021wiring} which captures the notion of composition of the traditional process theories that we defined in Def.~\ref{sec:PT}. 

\begin{definition}
An acyclic wiring operad $\cW^A$ consists of:
\begin{itemize}
\item  Boxes as objects. For example:
\beq
\begin{tikzpicture}
	\begin{pgfonlayer}{nodelayer}
		\node [style=none] (28) at (-1, 0) {};
		\node [style=none] (29) at (3, 0) {};
		\node [style=none] (40) at (1.5, 1.25) {\scalebox{0.6}{\input{wboxes.tikz}}};
	\end{pgfonlayer}
	\begin{pgfonlayer}{edgelayer}
		\draw [oWire] (28.center) to (29.center);
	\end{pgfonlayer}
\end{tikzpicture}
.
\eeq
\item Acyclic wirings of boxes as operations.  For example:
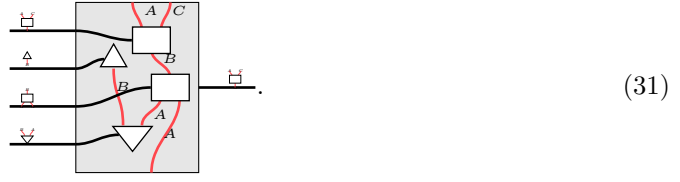
\beq
\input{figures/WiringTest1.tikz}.
\eeq
takes four boxes as input and wires them together to form a new box as an output.
\item The identity operation for every box is the trivial wiring. For instance, the identity for a box $\scalebox{0.6}{\input{wboxesAB.tikz}}$ is given by:
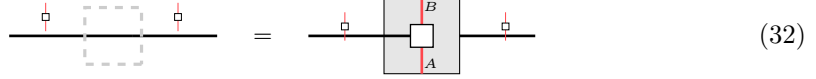
\beq
\input{figures/identitybox.tikz}\quad =\quad \input{figures/wiringIdent.tikz}
\eeq
\item Operation composition is given by diagram substitution.
For instance, 
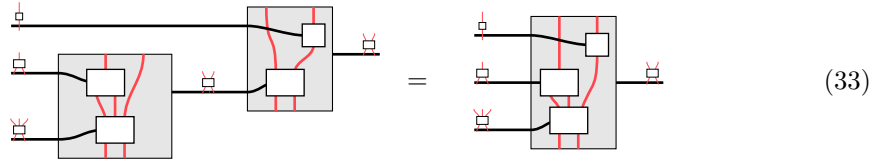
\beq
\input{figures/operadexample.tikz}
\eeq
in which the wiring diagram on the left is substituted into the bottom box of the wiring diagram on the right. In this way diagram substitution provides us with a more detailed wiring diagram than the one we had in the beginning.
\end{itemize}
\end{definition}

The acyclic wiring operad $\cW^A$ allows us to study properties of wirings of abstract boxes. However, typically we are not interested in abstract boxes, but an actual realization of those. In particular, we are concerned with boxes that describe physical processes. The corresponding algebras \colorbox{blue!10}{$F: \cW^A\rightarrow \textsc{set}$}, implement exactly that: They turn the abstract boxes and wirings of $\cW^A$ into actual processes and their composition.  More specifically, it was shown in Ref.~\cite{patterson2021wiring} that there is an equivalence between operad algebras for $\cW^A$ and symmetric monoidal categories. In the following we provide an intuition for this result, aided by the diagrammatic notation that we have set up.

Before getting to the technical details, let us try to give some intuition for this result. Suppose we have a diagram in a process theory, $\mathcal{P}$: 
\beq \input{operadexamplediagram.tikz}. \eeq
Then as on operad algebra $F_\mathcal{P}$ for $\cW^A$ we would instead represent this as:
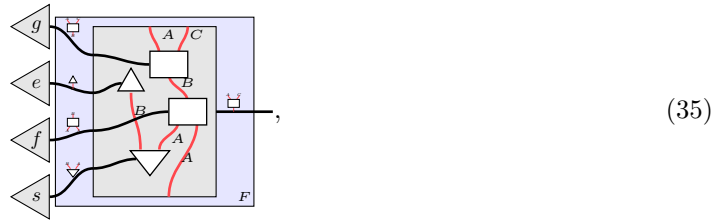
\beq\input{operadexamplealgebra.tikz}, \eeq
that is, rather than directly labelling the boxes we label them via the state of objects in the codomain of the operad algebra. The set of possible states of the boxes therefore corresponds to the set of possible processes in the theory.

In order to formalise this connection between operad algebras and process theories, it is useful to mention a few particular wiring operations. Specifically, sequential, parallel, identity, and swap wirings, which are diagrammatically denoted as:
\beq
\input{wcomp.tikz}\quad\text{,}\qquad \input{parcomp.tikz}\quad \text{,} \qquad
\input{figures/unitA.tikz}\quad \text{, and}\qquad  \input{sym.tikz}\quad \text{respectively.}
\eeq
These last two are interesting as they are operad operations which have no input box and only an output box. By considering the action of \colorbox{blue!10}{$F: \cW^A\rightarrow \textsc{set}$} on boxes, and on these four particular wirings in particular, we will see that it encodes all of the data constituting a symmetric monoidal category, $\mathcal{C}_F$.

On boxes $\scalebox{0.6}{\input{wboxesAB.tikz}}$\hspace{-1.5mm}, $F$ assigns a set $ F\left(\scalebox{0.6}{\input{wboxesAB.tikz}}\hspace{-1.5mm}\right)$ which we will take to be the homset of $\mathcal{C}_F$ with domain $A$ and codomain $B$. That is,
\begin{equation}
\mathcal{C}_F(A, B) := F\left(\scalebox{0.6}{\input{wboxesAB.tikz}}\hspace{-1.5mm}\right).
\end{equation}
This means that the morphisms in $\mathcal{C}_F$, that  process theoretically are denoted by
\[
\input{processAB.tikz},
\]
correspond to the elements of $F\left(\scalebox{0.6}{\input{wboxesAB.tikz}}\hspace{-1.5mm}\right)$ that have the following operadic representation:
\[
\begin{tikzpicture}
	\begin{pgfonlayer}{nodelayer}
		\node [style=none] (11) at (1.75, 0) {};
		\node [style=infpoint, fill={black!10}] (12) at (-1, 0) {$m$};
		\node [style=none] (14) at (0.75, 0.5) {$F\left(\scalebox{0.3}{\input{figures/box1to1.tikz}}\right)$};
	\end{pgfonlayer}
	\begin{pgfonlayer}{edgelayer}
		\draw [oWire] (12) to (11.center);
	\end{pgfonlayer}
\end{tikzpicture}
 \ \ = \ \ \begin{tikzpicture}
	\begin{pgfonlayer}{nodelayer}
		\node [style=none] (11) at (1.75, 0) {};
		\node [style=infpoint, fill={black!10}] (12) at (-1, 0) {$m$};
		\node [style=up label] (14) at (0.75, 0) {$\mathcal{C}_F(A,B)$};
	\end{pgfonlayer}
	\begin{pgfonlayer}{edgelayer}
		\draw [oWire] (12) to (11.center);
	\end{pgfonlayer}
\end{tikzpicture}
 \ \ =\ \  \input{morphRep.tikz} \ .
\]
Similarly, we will take the homsets with composite inputs and outputs to be given by $F$ applied to boxes with multiple inputs and outputs, for example:
\beq
\mathcal{C}_F(A_1\otimes\cdots\otimes A_n, B_1\otimes\cdots\otimes B_m):= F\left(\scalebox{0.6}{\input{wboxes.tikz}}\right).
\eeq
We will also take homsets to and from the monoidal unit to be given by $F$ applied to boxes without an input and without and output respectively, for example:
\beq
\mathcal{C}_F(I,B\otimes A) := F\left( \scalebox{0.6}{\input{Op1.tikz}}\right).
\eeq

The information about various elements of $\mathcal{C}_F$, can then be extracted as we apply $F$ to suitable wiring diagrams of $\cW^A$.
To obtain identities in $\mathcal{C}_F$ we need to pick out particular morphisms $\mathds{1}_A\in \mathcal{C}_F(A,A)$. We can pick out such a morphism via an operad operation in $\textsc{set}$ going from the singleton set $\star=\{*\}$ to $\mathcal{C}_F(A,A)$ . There is an obvious choice for such an operation which we can construct using the operad algebra $F$, namely:
\beq\label{def:morphident}
\begin{tikzpicture}
	\begin{pgfonlayer}{nodelayer}
		\node [style=none] (11) at (1.75, 0) {};
		\node [style=infpoint, fill={black!10}] (12) at (-1, 0) {$\mathds{1}_A$};
		\node [style=up label] (14) at (0.75, 0) {$\mathcal{C}_F(A,A)$};
	\end{pgfonlayer}
	\begin{pgfonlayer}{edgelayer}
		\draw [oWire, in=180, out=0] (12) to (11.center);
	\end{pgfonlayer}
\end{tikzpicture}
 \ \ := \ \ \input{figures/unitAn.tikz}.
\eeq
 Similarly, to obtain the symmetry morphisms we must pick out particular morphisms $\mathds{S}_{AB} \in \mathcal{C}_F(A\otimes B, B\otimes A)$. We can again do this by constructing a suitable operad operation in $\textsc{set}$ via:
\beq\label{def:morphswap}
\begin{tikzpicture}
	\begin{pgfonlayer}{nodelayer}
		\node [style=none] (11) at (1.75, 0) {};
		\node [style=infpoint, fill={black!10}] (12) at (-2, 0) {$\mathds{S}_{AB}$};
		\node [style=up label] (14) at (1.25, 0) {$\mathcal{C}_F(A\otimes B,B\otimes A)$};
	\end{pgfonlayer}
	\begin{pgfonlayer}{edgelayer}
		\draw [oWire, in=180, out=0] (12) to (11.center);
	\end{pgfonlayer}
\end{tikzpicture}
 \ \ := \ \ \input{symn.tikz}.
\eeq
Next let us turn to sequential composition. In order to define sequential composition in $\mathcal{C}_F$ we must define functions 
\beq
\circ_{A,B,C}: \mathcal{C}_F(A, B)\times \mathcal{C}_F(B, C)\rightarrow \mathcal{C}_F(A, C),
\eeq
which can, be thought of as an operad operation in $\textsc{set}$ and the obvious choice for which is given by:
\beq\label{def:seq}
\input{wCompn1.tikz} := \input{wCompn.tikz}.
\eeq
Similarly, we can define parallel composition in $\mathcal{C}_F$ via functions
\beq
\otimes_{A,A';B,B'}:\mathcal{C}_F(A, B)\times \mathcal{C}_F(A', B')\rightarrow \mathcal{C}_F(A\otimes A', B\otimes B')
\eeq
where the obvious choice for this is:
\beq\label{def:par}
\input{parCompn1.tikz} := \input{parCompn.tikz}.
\eeq

Finally, we need to demonstrate that these definitions actually satisfy all of the axioms of a symmetric monoidal category. We will not provide a complete proof here but instead indicate an illustrative example with the aid of our diagrammatic representation.  Specifically, we show unitality of the identity morphisms, that is the following condition:
\[
\input{identUnit1.tikz} =\ \ \input{identUnit2.tikz}.
\]
The proof of this is straightforward. From the definitions given above, the left hand side is equal to:
\begin{align}
\input{identUnit3.tikz} &=  \input{identUnit4.tikz}\\
&= \input{identUnit5.tikz} \\
&= \input{identUnit6.tikz} \\
&= \input{identUnit2.tikz}
\end{align}
The first equality is given by functoriality of $F$, the second by the definition of composition in $\cW^A$, the third from the definition of the identity in $\cW^A$ and the fourth again from functoriality of $F$. 

The proof technique to show the other conditions is practically identical. It follows from straightforward applications of the  definitions that we have set up,  functoriality of $F$, and composition in the wiring operad.

We have so far established that the algebra \colorbox{blue!10}{$F: \cW^A\rightarrow \textsc{set}$} gives rise to a SMC $\mathcal{C}_F$. The reverse also holds, i.e., given a SMC $\mathcal{C}$,  we can define the algebra $F_\mathcal{C}: \cW^A\rightarrow \textsc{set}$ by defining the action of $F_\mathcal{C}$ on boxes and wiring diagrams of $\cW^A$. In particular,  we define
\beq
F_\mathcal{C}\left(\scalebox{0.6}{\input{wboxesAB.tikz}}\hspace{-1.5mm}\right) := \mathcal{C}(A,B).
\eeq
Defining $F_\mathcal{C}$ on the special cases of sequential, parallel, identity, and swap wirings is also straightforward, that is, we simply turn around the definitions in Eqs.~\eqref{def:morphident}, \eqref{def:morphswap}, \eqref{def:seq}, and \eqref{def:par}. The tricky bit, however, is to define the action of $F_\mathcal{C}$ on arbitrary wirings. The solution to this problem was given in Ref.~\cite{patterson2021wiring}, in which they show that any wiring operation in $\cW^A$ can be decomposed in terms of these four particular wiring operations. Hence, once we know how $F_\mathcal{C}$ is defined for these four, we, in principle at least, know how it is defined for arbitrary wirings. The fact that $\mathcal{C}$ is an SMC then ensures that this is independent of how we decompose the given wiring operation.

\

There is, however, an important subtlety that we have so far glossed over. That is, there is not actually a unique wiring operad but one for each possible choice of labels for the inputs and outputs to boxes. The wiring operad that we get from a particular SMC will therefore be the one in which the box labels correspond to the objects of the SMC. For more details on this see Ref.~\cite{patterson2021wiring}.

\subsection{Discussion}

We have therefore seen that we can equivalently describe a traditional process theory either as an SMC, which has been the conventional perspective to date, or as an operad algebra for the acyclic wiring operad $\cW^A$. We therefore propose the following definition: 
\begin{definition}[Traditional process theories]
A traditional process theory is an acyclic wiring operad algebra, \colorbox{blue!10}{$F: \cW^A\rightarrow \textsc{set}$}. 
\end{definition}
This addresses three of the shortcomings of the categorical approach discussed in the introduction. 
\begin{itemize}
\item[--] We do not need to introduce any fictitious trivial system to the process theory in order to make the connection to the operadic language.
\item[--] We do not need to view the identities and swaps as processes in their own right.
\item[--] We do not privilege sequential an parallel composition over all others.
\end{itemize}
However, arguably these advantages alone are not sufficient to make one want to give up the categorical formalism which is now widely used and understood by the foundations of physics community. The major advantage of the formalism should become clear over the remainder of this paper, namely, that it provides a more flexible description which can be more easily adapted to various situations of interest.

\section{Time-neutral process theories} \label{sec:TimeNeutral}

A time-neutral process theory is one in which there is no separation of the systems attached to a given process into inputs and outputs. The need for such theories can be motivated by fundamental questions, such as the question as to how time emerges from fundamentally timeless theories \cite{timesymmetry}, or by practical problems in which we are modelling static situations, for example, formalising electric circuit diagrams \cite{fong2018seven}. In the latter case, for example, there is no reason to ascribe a particular side of the battery as being an input and the other an output, they simply are places that we can wire other things to.

In order to describe theories which do not have an input/output distinction we must revisit the basics of the definition of the wiring operad. In this section, we demonstrate how to do this by defining a new kind of wiring operad whose associated  operad algebras can be thought of as time neutral process theories. In App.~\ref{app:cupsandcaps} we relate this to the ``standard'' approach of introducing time neutrality via cups and caps.

A time-neutral theory is defined by a collection of processes, such as
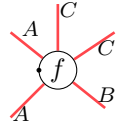
\[
\input{timeNeutralProcess.tikz}
\]
which is closed under wirings. For example,
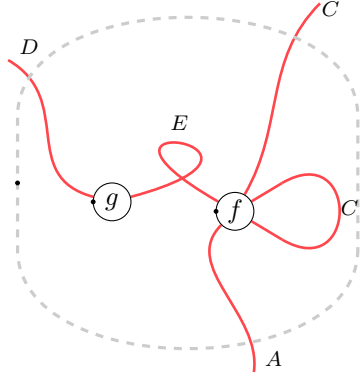
\[
\input{timeNeutralDiagramNewNotation.tikz}
\]
corresponds to another process in the theory. Moreover, two diagrams are equal if they have the same connectivity. Those processes, represented as circles, come with an associated list of systems such that there is no separation of this list into input and output systems -- all systems are on a completely equal footing.

\begin{example}[Tensor networks]
A mathematical example of a time-neutral process theory is the theory of tensors over some field $\mathds{F}$. Wires are labelled by finite index sets, $I$, $J$, ..., and processes are labelled by tensors $T=
\{T_{ij...}\in \mathds{F}|i\in I;j\in J;...\}$. Wiring together these tensors is then simply contraction of the relevant indices. For example if we wire an $I$ system of $T$ to an $I$ system of $S$ we obtain the tensor $\{\sum_{i\in I} T_{...i...}S_{...i...}\}$.
\end{example}

This can therefore be seen as a formalisation of the kinds of diagrams which are used in the study of tensor networks \cite{orus2014practical}. However, in this case there is likely a better formalisation which we discuss in the next subsection.

\begin{example}[ZX calculus (Informal)]
The ZX calculus \cite{van2020zx,coecke2022kindergarden} is typically defined as a traditional process theory with inputs and outputs. However, it is clear from the diagrams which are drawn when actually working with the ZX calculus that there is no distinction between inputs and outputs, and, hence, that it can be viewed as an example of a time-neutral process theory. 
\end{example}


\begin{example}[Time-neutral quantum theory]
We can succinctly represent time-neutral quantum theory \cite{timesymmetry} as a time-neutral process theory by taking systems to be finite-dimensional complex matrix algebras $\mathcal{A}:=\bigoplus_{i\in I_\mathcal{A}} \mathbf{M}_{n_{\mathcal{A}_i}}(\mathds{C})$. Processes, for example with systems $\mathcal{A}$ and $\mathcal{B}$, are non-negative elements of the tensor product algebra, e.g. $\rho \in \mathcal{A}\otimes \mathcal{B}$, considered up to scaling  by non-zero real numbers. Specifically, they are represented  with equivalence classes $[\rho]:=\{r\rho| r \in \mathds{R}^+\}$. Composition is defined by taking a representative element of each equivalence class, tensoring on identities such that they belong to the same algebra, multiplying the matrices, taking the partial trace of the systems which are being composed, and then quotienting to obtain the equivalence class. For example suppose we have $[\rho]$ where $\rho \in \mathcal{A}\otimes \mathcal{B}$ and $[\sigma]$ where $\sigma\in \mathcal{B}\otimes\mathcal{C}$. Then wiring these together by system $\mathcal{B}$ would give $[\mathsf{tr}_\mathcal{B}((\rho\otimes \mathds{1}_\mathcal{C})\cdot (\mathds{1}_\mathcal{A}\otimes \sigma))]\in \mathcal{A}\otimes\mathcal{C}$.
\end{example}

We will now see how to define an operad which faithfully captures the compositional structure of these time-neutral process theories by switching from a wiring operad in which objects are `boxes' to a wiring operad in which objects are `dots'.

\begin{definition}
The wiring operad of dots, $\cW^{D}$, consists of:
\begin{itemize}
\item Dots as objects. For example:
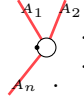
\[
\input{circle.tikz}
\]
where the black point on the LHS of the dot ensures that we still have an order to the attached wires.
\item Wirings of dots as operations. For example:
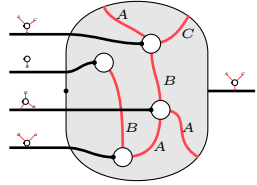
\[
\input{WiringTest2.tikz}
\]
takes four dots as inputs and wires them together to give a single dot as an output.
\item The identity operation for every dot is the trivial wiring. For instance:
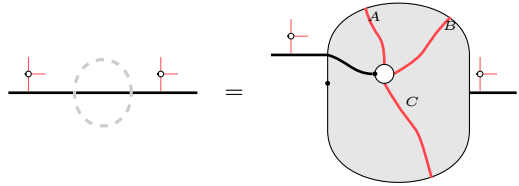
\[
\input{identitytn.tikz}\ \ =\ \ \input{identWiringtn.tikz}. 
\]
\item Operation composition is given by diagram substitution. For instance:
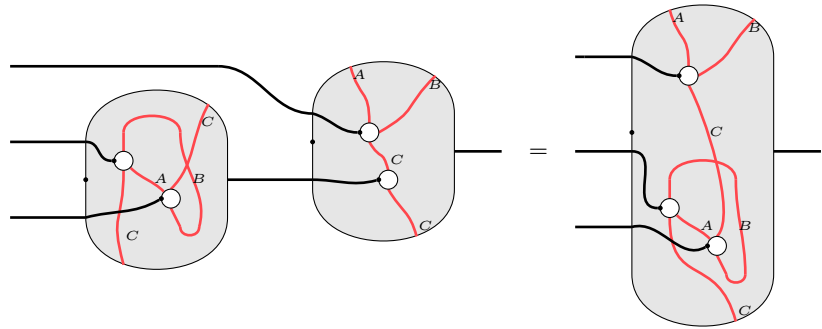
\[
\input{tnOperComp1.tikz}\ \ =\ \ \input{tnOperComp2.tikz}.
\]
in which the wiring diagram on the left is substituted into the bottom dot of the wiring diagram on the right. In this way diagram substitution provides us with a more detailed wiring diagram than the one we had in the beginning.
\end{itemize}
\end{definition}

In analogy with the wiring operad $\cW^A$, there are diagrams in $\cW^{D}$ that remind us of core operations in process theories.
For example we have an analogue of sequential composition:
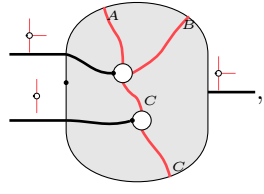
\[
\input{seqcomptn.tikz},
\]
and an analogue of parallel composition:
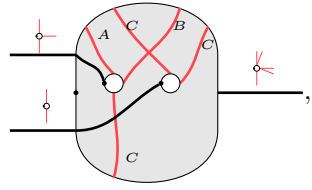
\[
\input{parcomptn.tikz},
\]
as well as identity wires:
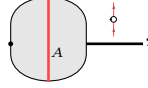
\[
\input{unitAtn.tikz},
\]
which can also be drawn in other equivalent ways which are reminiscent of cups and caps:
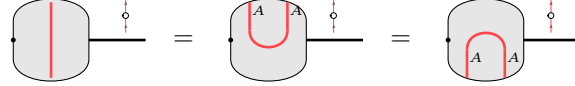
\[
\input{equalityofcupscaps.tikz}.
\]
The above equality constitutes a hint to the connection with compact closed categories as discussed in depth in App.~\ref{app:cupsandcaps}.

We also have swaps:
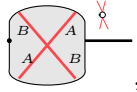
\[
\input{symtn.tikz},
\]
and also wirings, that simply permute the ordering of the systems attached to a dot:
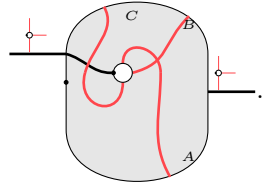
\[
\input{permtn.tikz}.
\]

A time-neutral theory can be formally defined as follows:
\begin{definition}[Time-neutral process theories]
A time-neutral process theory is an wiring operad of dots algebra,\colorbox{blue!10}{$G: \cW^{D}\rightarrow \textsc{set}$}. 
\end{definition}
 Such an algebra, assigns a set to each dot, which we interpret as the set of time-neutral processes for the systems attached to the dot. On wiring diagrams, $G$ assigns composition functions which specify what the process associated to the composite dot will be given the processes associated to the component dots. Functorialiy of $G$ and the definition of the wiring operad $\cW^{D}$ ensure that all of the composition functions will interact in exactly the way we would expect.

\subsection{Partially time-neutral process theories} 

One could also formulate the idea of a process theory in which some systems were treated in a time-neutral way and others were treated as inputs and outputs. We leave a full formalisation of this idea to future work, but give the basic idea here to highlight the flexibility of the operadic approach.

Such theories could be described by a wiring operad for a three dimensional shape such as:
\beq
\input{partialTimeNeutral.tikz}.
\eeq
where the systems running vertically are treated as inputs and outputs. We demand that wirings respect this by satisfying the acyclicity constraints of a traditional process theory; in contrast, the systems in the horizontal plane are treated time-neutrally so arbitrary wirings are possible. For example, this would allow for diagrams such as:
\beq
\input{partialTimeNeutralWiring.tikz}.
\eeq

Algebras of this operad would then be the best way to understand many diagrams used in the context of tensor networks as some connections between nodes correspond to a space-like separation and should be treated time-neutrally whilst others correspond to time-like separation and hence should satisfy the acyclicity constraint of traditional process theories.

\section{Higher-order Processes}

In modern quantum foundations a key conceptual notion is that of \textit{supermap} \cite{Chiribella_2008_supermaps, chiribella_architecture, chiribella_causal, chiribella_networks}, also referred to in the literature as process-matrices \cite{oreshkov2015operational}. In intuitive terms, theories of supermaps are theories of boxes with holes:
\beq
\input{new_supermap_1.tikz}.
\eeq
These pictures whilst intuitive represent a challenge to the standard process theoretic approach, as argued in \cite{Guerin_2019_nogo}. 
A key issue in the compositionally of supermaps is the difference between interpretation of higher-order inputs and higher-order outputs. Inputs represent separate parties which are forbidden from communicating whereas outputs represent parties which are allowed to communicate. At the compositional level the result of this restriction is that one may give meaning to diagrams such as
\beq
\input{new_super_allowed.tikz}.
\eeq
but one may not give meaning to diagrams such as 
\beq
\input{new_super_nowaydude.tikz}.
\eeq
since such diagrams can be used to form non-causal time loops \cite{Guerin_2019_nogo, wilson_polycat}. 
This can be encoded as a more formal restriction, that one must compose along no-more than a single wire at a time \cite{wilson_polycat, hefford_coend}. For instance, representing 
\beq
\input{new_supermap_1.tikz} \quad \textrm{as} \quad \input{super_simple.tikz}
\eeq
where the symbol $\mathbf{Z}$ is used to represent each pair $(Z,Z')$, meaning can be given to
\beq
\input{poly_meaning_yes.tikz} \quad \textrm{ but not } \quad \input{poly_meaning_no.tikz}  \quad \textrm{ or} \quad \input{poly_meaning_no_2.tikz}.
\eeq
In categorical terms this means that supermaps with non-communicating inputs and communicating outputs form not symmetric monoidal categories but instead symmetric polycategories \cite{szabo_polycats}, with this structure being the prop-like analogue to the linearly distributive structure of finite dimensional supermaps constructed in \cite{kissinger_causal}. 
One might note naturally from the pictures associated to supermaps, that there is \textit{some} reasonable notion of parallel composition available, just not one that can be thought of as neatly giving a monoidal structure to all lists of objects. Indeed, we can give meaning to diagrams such as:
\beq
\input{super_space_1.tikz}  \quad \textrm{as} \quad \input{super_space_2.tikz}.
\eeq

Such theories in which polymorphisms can be composed in parallel we shall think of as \textit{polycategories with space}.

Let us now consider the operads who's algebras give rise to these outlined concepts. 
Note that from any directed graph, we can construct an associated undirected graph, by simply replacing each directed edge with an undirected one. We call the sub-operad of the acyclic wiring operad in which all wiring diagrams have underlying undirected acyclic graphs, the undirected acyclic wiring operad $\cW^{UA}$. 
For bare polycategories without a parallel composition rule one can study the operad $\cW^{UAC}$ of connected undirected acyclic graphs, the algebras of this operad have been previously proven to be symmetric polycategories in \cite{yau2018operads}. 

Let us focus on the general undirected acyclic case. We define \textit{polycategories with space} formally to be algebras of type  $\cW^{UA} \rightarrow \mathbf{Set}$. Whilst process theories can be thought of as extracting the diagrammatic part of monoidal category theory, we have begun here instead with purely process-theoretic concepts and so what could remain to do is to find a categorical characterisation. 
\begin{theorem}[Polycategories with space]
A polycategory with space is a polycategory equipped with an identity-on-objects polyfunctor $\otimes: \mathcal{P}_{\times} \rightarrow \mathcal{P} $ such that $(f \otimes g) \otimes h = f \otimes (g \otimes h)$. 
\end{theorem}
\begin{proof}
Given in the appendix, where $\mathcal{P}_{\times}$, the polycategory of pairs of polymorphisms from $\mathcal{P}$ is also defined. 
\end{proof}
What we have seen so far, is that if supermaps are thought of as lower-order processes theories on higher-order objects, then some of their compositional features are beyond the standard process theory formalism but within the general operadic formalism.
A natural open question however, is whether one can do better than thinking in terms of quite mild adaptations of standard pre-defined categorical structures such as the notion of a polycategory. In particular:
\begin{itemize}
\item It seems plausible that an operad could be constructed in which the objects rather than being boxes are given by these more abstract shapes, boxes with holes. Operations in the operad could then be used to construct all of the expected reasonable compositional laws for such boxes, for instance allowing to form diagrams such as this \beq
\input{hopeful_diagram.tikz} 
\eeq
Algebras of such an operad could then give a satisfactory answer to the question \textit{what is a theory of supermaps?}.
\item It might even be possible in principle to express in terms of generalised process theories the full algorithmic characterisation and type-theoretic characterisations of possible composition rules of higher-order quantum processes \cite{apadula_nosig, kissinger_causal, Hoffreumon_projective, simmons_bv_logic}.
\end{itemize}
These are possibilities left as future directions of research, with applications particularly in the definition of resource theories of higher-order processes.

\section{Causal process theories}
An important key point concerning the wiring diagrams of $\cW^A$ is that they are agnostic to a time direction. Here, we take a step further by imposing a (causal) arrow of time by defining a new kind of wiring operad, namely a \emph{causal wiring operad} $\cW^{\groy}$. Specifically, we ask ourselves how we should modify $\cW^A$, so that the corresponding operad algebra will necessarily give rise to a causal SMC $\mathcal{C}$\footnote{That is, one in which the monoidal unit is terminal}. 
It turns out that we can implement this in two steps. Firstly, we relax the constraints on the wirings considered in $\cW^A$, and secondly we impose conditions that these new wirings must satisfy. We will find that to impose these conditions requires us to interpret wiring operads with a notion of empty space, we explain how they arise formally in the appendix, and put them forward as the right notion for generalised process theories.

To understand the first step, note that in $\cW^A$ every wire must begin on some box which could be an input box or the output box. Similarly every wire must end on some box, which again could be an input box or the output box. In order to incorporate causality we break this symmetry by allowing for wires which do not end on any box, as in the following case:
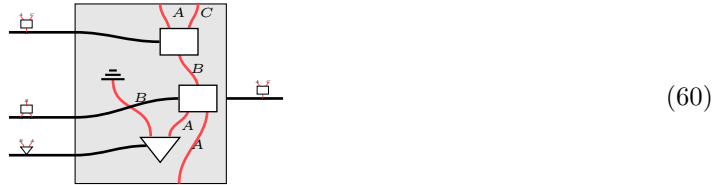
\beq\label{eq:examplediscardwiring}
\input{wiringTest3.tikz}
\eeq
We use the ground symbol to indicate the termination of a wire that does not end on any box. We think of this as \emph{discarding} the system $B$. 

 We could alternatively represent this by allowing wires to end at an inaccessible ``point at infinity'' via diagrams such as:
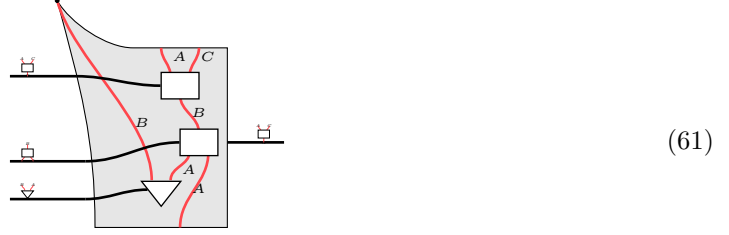
\beq
\input{futureinf.tikz}
\eeq
which provides a passive rather than active perspective on what it means to discard a system. For convenience, however, we will work with the active perspective and leave formally proving an equivalence between these perspectives to future work.

The wiring diagrams in $\cW^{\groy}$ can be built up out of the wirings from $\cW^A$, together with a new wiring diagram:
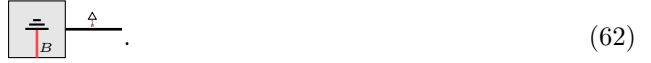
\beq
\input{figures/operadicdisc2.tikz}. 
\eeq
For example, we can construct Eq.~\eqref{eq:examplediscardwiring} via:
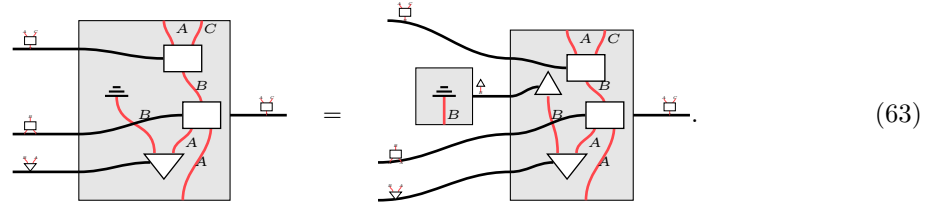
\beq
\input{figures/wiringTest3.tikz}\quad =\quad \input{figures/wiringTest4.tikz}. 
\eeq

In order for $\groy_B$ to fully capture the notion of causality, we must impose an extra condition:  if we `discard' the output of a box then we may as well have directly discarded the inputs, and moreover, we don't care what the box actually was. That is, we want to demand that the following equation is satisfied:
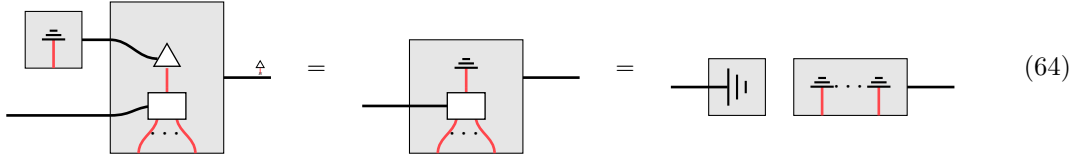
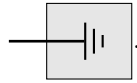
\beq\label{eq:operadicCausality}
\input{discexample.tikz}\quad =\quad \input{discexample1.tikz} \quad = \quad \input{operadicdisccondition.tikz}
\eeq
where the operation
\[
\input{operadicdisc.tikz}. 
\]
is a `discard' of the incoming box, meaning that we are simply not interested in what this box is. We call this \emph{operadic discarding}\footnote{The condition \eqref{eq:operadicCausality} manifests the interplay between operadic and physical discarding operations in a way that is analogous to the ignorability condition of Ref.~\cite[Eq. 96]{schmid2020unscrambling}. In that case, the interplay is between ignoring causal and inferential systems. More generally, it states that if we discard the output of an operation, then the nature of the operation is irrelevant and we may as well have directly discarded the input.}.
This must moreover satisfy the condition that if we are not interested in the output box from some wiring, then we should be equivalently disinterested in all of the input boxes. For example:
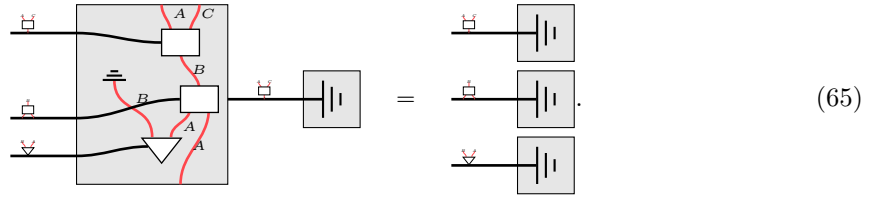
\beq\label{eq:operadDiscardCond}
\input{wiringTest5.tikz}\quad =\quad \input{wiringTest6.tikz}.
\eeq

An important technical point, is that operads do not carry with them a good notion of operation with no-output. As a result, we must adapt our arena in which we design composition operations. On the other hand, we must not lose the foundational results on process theories and time neutral theories already established. 
In the appendix we explain how to freely add empty-outputs and operadic discarding to any operad\footnote{The resulting algebraic structure, is strictly an affine prop, we keep further discussion of the distinction to the appendix.}, and explain that none of the porevious results are lost by instead using these adapted operads. The only thing that needs to be kept in mind regarding diagrammatics, is that functors now have an additional property: they send operations of type $X \rightarrow ()$ to operations of type $F(X) \rightarrow \star$. In pictures, this means that we can now write down the following kinds of diagrams
\beq
\input{figures/operadFunctor_empty.tikz},
\eeq
with the diagrammatic rule for functorality extending to include such assignments. 
Imposing the equation \ref{eq:operadicCausality} can also be done freely, finishing the construction from wiring operads to \textit{causal wiring operads}. 

\begin{definition}
We call a wiring operad $\cW^A$ extended with a discarding diagram and an operadic discarding operation such that conditions \eqref{eq:operadicCausality} and \eqref{eq:operadDiscardCond} are satisfied,  a causal wiring operad, or in short $\cW^{\groy}$.
\end{definition}

We are now in a position to show that an algebra of a causal wiring operad, \colorbox{blue!10}{$F:\cW^{\groy}\to \textsc{set}$} gives rise to a causal SMC $\mathcal{C}_F$. To begin, we can construct the SMC in exactly the same way as for the wiring operad $\cW^A$.  However, we can now specify additional data for the SMC by considering the action of $F$ on the discarding diagram $\groy_B$. This gives a function $F (\groy_B): \star \rightarrow \mathcal{C}_F(B,I)$ where we interpret the image of the function $F (\groy_B)$ to be the discarding effect for system $B$ in $\mathcal{C}_F$.  Diagrammatically this function is denoted:
\beq 
\input{figures/opDiscAlg1.tikz}.
\eeq
Moreover, the action of $F$ on the operadic discarding, $\gro$, is $F(\gro): \mathcal{C}_F(A,B)\rightarrow \star$. Note that for any set $X$, there is a unique function to the singleton set $\star$. Hence, the action of $F$ on operadic discarding is uniquely fixed. These unique functions act as operadic discarding maps for the operad $\textsc{set}$. We therefore have that
\[
\input{opDiscAlg2.tikz} = \begin{tikzpicture}
	\begin{pgfonlayer}{nodelayer}
		\node [style=up label] (44) at (-1.75, 0) {$\mathcal{C}_F(A,B)$};
		\node [style=none] (59) at (0, 0) {};
		\node [style=none] (60) at (-1.75, 0) {};
		\node [style=infupground] (61) at (0.25, 0) {};
	\end{pgfonlayer}
	\begin{pgfonlayer}{edgelayer}
		\draw [oWire] (59.center) to (60.center);
	\end{pgfonlayer}
\end{tikzpicture}
.
\]
Now we prove that the resulting process theory is causal, i.e., if we compose the output of any process with the discarding effect then we obtain the discarding effect on the input. Indeed, for any process $f \in \mathcal{C}_F(A,B)$ we have that
\begin{align}
 \input{opDiscAlg4.tikz}
 &= \input{opDiscAlg5.tikz} \\
 &=  \input{opDiscAlg6.tikz} \\
 &=  \input{opDiscAlg7.tikz} \\
 &=  \input{opDiscAlg8.tikz} \\
 &=   \input{opDiscAlg9.tikz},
\end{align}
for all morphisms $f\in\mathcal{C}(A,B)$.

To sum up, we have extended the wiring operad $\cW^A$ to allow for wirings in which systems do not end on a box, and thus being discarded. We additionally introduced a new type of discarding, operadic discarding, and ensured that it interacts nicely with the discarding diagram. Consequently, the operad $\cW^{\groy}$  gives rise to a causal process theory or SMC $\mathcal{C}$, which is the kind of process theory we typically use to describe physical theories.

\section{Enriched process theories}

In the previous sections we have defined operad algebras as operad functors $F:\mathcal{O}\to \textsc{set}$, where for the various kinds of operads that we have considered we can interpret these algebras as various kinds of process theories. 
In particular, by varying the operad $\mathcal{O}$, we change the notion of compositionality that the resulting theory has; whilst by varying the functor $F$ we obtain different instantiations of the theory, that is, $F$ can be thought of as encoding the particular physics that we are interested in. However, what we have not yet considered varying is the codomain of $F$, we have left this fixed as the operad $\textsc{set}$. 

In the case of traditional process theories, i.e., operad algebras $F:\cW^A\to \textsc{set}$, i.e., symmetric monoidal categories $\mathcal{C}_F$, We have seen that the image of some object in $W$ under the functor $F$ is what determines the hom-set in the SMC, $\mathcal{C}_F$, that is constructed. However, there are many times where the process theories that we are interested in are enriched, in that the processes between a particular set of inputs and outputs do not merely form a set, but themselves have some additional structure. A classic example of enrichment which appears in many fundamental physical frameworks is the notion of convexity, or more concretely, the notion that the processes can be combined using convex combinations to form convex spaces \cite{fritz2015convex}. This can indeed be captured by varying the codomain of $F$ as we now show.

\begin{theorem}
An operad algebra $F:\cW^A\to K$ encodes the information that constitutes a $K$-enriched symmetric monoidal category.
\end{theorem}
\proof 
An enriched category is one in which for each pair of objects $A,B$ there is a hom-object $\mathcal{C}_F(A,B)$ in $K$ along with operations in $K$ of type $(\mathcal{C}_F(A,B), \mathcal{C}_F(B,C)) \rightarrow \mathcal{C}_F(A,C)$ which satisfy laws analogous to associativity in $\textsc{set}$. Naturally a monoidal enriched category is equipped with additional parallel composition morphisms $\otimes : (\mathcal{C}_F(A,A'), \mathcal{C}_F(B,B')) \rightarrow \mathcal{C}_F(AB,A'B')$ which satisfy interchange with the sequential composition morphisms. The construction and well-behaviour of this structure is given in exactly the same way as it was presented for $\textsc{set}$.
\endproof

A key example is convex-enrichment where $K$ can be taken to be the operad $\text{ConvSpc}$ of convex spaces and convex maps. Enrichment here enforces not only that convex combinations of processes exist, but that sequential and parallel composition play well with this notion of convex combination, meaning that all conditions such as $f \circ (pg + p'g') = pf \circ g + p' f \circ g'$ are satisfied by both $\circ$ and $\otimes$. 
Concrete examples of process theories which are enriched in convex spaces include but are not limited to the theories of quantum channels, completely positive maps, unital channels, and classical channels also known as stochastic matrices. 

This operadic view of enriched process theories can be leveraged to define a convenient string diagrammatic language for enriched symmetric monoidal categories see App.~\ref{app:EnrichedStringDiagrams}.

This string diagrammatic language for enriched symmetric monoidal categories bares a striking resemblance to the diagrams with the framework of causal-inferential theories \cite{schmid2020unscrambling}, a connection which we will explore in future work.

The key point of the above theorem, is that we can extend it to express enrichment for generalised process theories, that is,
\begin{definition}
A generalised process theory is an operad algebra 
\beq
F:\mathcal{O}\to K,
\eeq
where $\mathcal{O}$ captures the composition of processes, $K$ captures a theory of enrichment, and $F$ encodes the specific theory of interest.
\end{definition}

\section{Generalised process theories: discussion and future work}

In order to cope with the variety of compositional notions in the foundations of physics, we have discussed a possible formulation of generalised process theories in terms of a minor modification of wiring operads. This view encapsulates causal process theories, time neutral process theories, and aspects of higher-order process theories, and that enrichment arises by simply modifying the codomain of the algebra at hand. Furthermore, the generalised process theory has the potential to provide a shortcut to the definition of new categorical structures associated to new compositional notions. We saw a humble glimpse of this in the characterisation of the undirected acyclic wiring operad in terms of polycategories with space. There is much more to be discovered here, given the recent discoveries and characterisations of elaborate composition rules in higher order quantum theory, in terms of graph structures, proof nets, and general algorithmic procedures \cite{apadula_nosig, kissinger_causal, Hoffreumon_projective, simmons_bv_logic}.

Another key potential application of the move to generalised process theories is in the formalisation of resource theories \cite{CFS} in all of the previously mentioned scenarios in which the notions of compositionality is either more restricted or more relaxed than what is permitted by symmetric monoidal structure. Crucially, it is common practice in the study of resource theories to equate resources in terms of convertibility using by arbitrary compositions of free processes. Clearly then, the allowed compositions of free processes within a theory are likely to have consequences: The more compositional rules that are permitted, there more coarse-grained the equivalence between resources should become. To capture these ideas in an organised way, we will need a definition of resource theory for generalised process theories beyond the standard notion of monodial subcategory as in \cite{CFS}. There is an even more ambitious possibility raised by this line of thought: What if compositionality itself was a resource? To frame this will require a definition of structure preserving map between generalised processes theories over different classes of wiring operads, and will likely require generalisation of the methods for modelling non-identity-on-on-objects functors. 

In addition, the operadic framework for process theories provides a natural setting for describing tensor networks, which are widely used in many body physics and quantum information theory \cite{orus2014practical}. In particular, tensor networks can be understood as time-neutral process theories, where tensors correspond to elements of the operad algebra associated with the wiring operad of dots,  $\cW^{D}$. In this case, the operad  $\cW^{D}$, encodes the abstract rules of composition, specifying how systems can be connected, while the operad algebra assigns to each system a set of tensors that can be composed according to these rules. Tensor contraction in particular, corresponds to operadic composition of processes with summation over contracted indices being identified with connected wires.  

The operadic framework provides a natural setting for ZX calculus \cite{van2020zx,coecke2022kindergarden}. ZX calculus is usually formulated within a SMC. However, spiders suggest a more time-neutral formulation, where inputs and outputs are not fundamentally distinguished as in the wiring operad of dots. In this perspective, ZX spiders correspond to elements of the operad algebra of  $\cW^{D}$, encoding fusion rules as operadic composition and not as categorical morphisms. In the case where causal constraints are introduced , the causal wiring operad provides a way to include discarding operations ensuring that ZX diagrams respect causality. We suggest that operad algebras offer a unifying framework for ZX calculus, leading to new insights in its compositional structure and quantum computation. 

We furthermore believe that the operadic framework shares a deep structural connection with Compositional Quantum Field Theory (CQFT) as presented in \cite{GBQFT}. CQFT aims to provide an axiomatic, compositional approach to quantum field theory. The fundamental building blocks are given by state spaces associated with spacetime regions, whereas physical processes correspond to linear maps between those spaces. Consequently,  gluing spacetimes regions identifies with the composition of linear maps. This approach can be formulated in terms of operad algebras where objects in the operad correspond to spacetime regions and operations in the operad correspond to ways in which those regions can be composed. The operad algebra would then assign to each region a state space and to each operad operation a physcal evolution map. In other words, we presume that CQFT can be framed as an operad algebra over a spacetime operad. 

Moeover, in the context of CQFT quantum processes are related with arbitrary spacetime regions rather than being constrained by an initial and final time-slice. This feature might be one that could be captured by time-neutral process theories, where the standard input-output distinction is abandoned in favour of a more flexible network of interactions.

The purpose of the traditional process theory formalism is to model the purely compositional aspect of physical theories. With it's motivations coming initially from quantum foundations and quantum computation, it is natural that it was  based on the circuit-model, equivalently the model of symmetric monoidal categories, in which there exist simple sequential and parallel composition rules. An important point regarding the adoption of the circuit model in quantum foundations, is that it would have been much more difficult if category-theoretic expertise was a pre-requisite. With this in mind, the generalised process theory formalism, based on algebras of wiring operads, seems to give a suitable meta-framework for defining and using composition rules beyond the circuit model. In particular, whilst based on (and an extension of) the insights of \cite{spivak2013operad,spivak2017string,patterson2021wiring,fong2018seven, yau2018operads} it is framed in such a way as to be usable without expertise in category theory. It is our hope that this framework will hence help to give structure to the ever expanding story of compositionality in the foundations of physics.

\section*{Acknowledgements}

JHS thanks Enzo Pons for many interesting discussions regarding the operadic approach to compositional theories. MW thanks James Hefford for useful conversations. MW acknowledges support by the National Science Centre, Poland (Opus project, categorical foundations of the non-classicality of nature, contract number UMO-2021/41/B/ST2/03149). JHS was funded by the
European Commission by
the QuantERA project ResourceQ under the grant
agreement UMO-2023/05/Y/ST2/00143.
JHS conducted part of this research while visiting the Okinawa Institute of Science and Technology (OIST) through the Theoretical Sciences Visiting
Program (TSVP). BC carried out this work while at Oxford University. MS carried out this work while at Oxford University, was also supported in part by
the Foundation for Polish Science through IRAP project co-financed by EU within Smart
Growth Operational Programme (contract no. 2018/MAB/5), and is currently funded by DESY. 
All of the diagrams within this manuscript were prepared using \href{https://tikzit.github.io/}{TikZit}.

\bibliographystyle{plain}
\bibliography{main}

\begin{thebibliography}{54}
\providecommand{\natexlab}[1]{#1}
\providecommand{\url}[1]{\texttt{#1}}
\expandafter\ifx\csname urlstyle\endcsname\relax
  \providecommand{\doi}[1]{doi: #1}\else
  \providecommand{\doi}{doi: \begingroup \urlstyle{rm}\Url}\fi

\bibitem[Apadula et~al.(2022)Apadula, Bisio, and Perinotti]{apadula_nosig}
Luca Apadula, Alessandro Bisio, and Paolo Perinotti.
\newblock No-signalling constrains quantum computation with indefinite causal
  structure.
\newblock \emph{arXiv}, 2022.
\newblock \doi{10.48550/ARXIV.2202.10214}.
\newblock URL \url{https://arxiv.org/abs/2202.10214}.

\bibitem[Barrett(2007)]{Barrett}
J.~Barrett.
\newblock Information processing in generalized probabilistic theories.
\newblock \emph{Physical Review A}, 75:\penalty0 032304, 2007.

\bibitem[Bisio and Perinotti(2019)]{bisio2019theoretical}
Alessandro Bisio and Paolo Perinotti.
\newblock Theoretical framework for higher-order quantum theory.
\newblock \emph{Proceedings of the Royal Society A}, 475\penalty0
  (2225):\penalty0 20180706, 2019.

\bibitem[Boisseau et~al.(2022)Boisseau, Nester, and
  Roman]{boisseau2022cornering}
Guillaume Boisseau, Chad Nester, and Mario Roman.
\newblock Cornering optics.
\newblock 2022.

\bibitem[Chiribella et~al.(2008{\natexlab{a}})Chiribella,
  D{\textquotesingle}Ariano, and Perinotti]{Chiribella_2008_supermaps}
G.~Chiribella, G.~M. D{\textquotesingle}Ariano, and P.~Perinotti.
\newblock Transforming quantum operations: Quantum supermaps.
\newblock \emph{{EPL} (Europhysics Letters)}, 83\penalty0 (3):\penalty0 30004,
  jul 2008{\natexlab{a}}.
\newblock \doi{10.1209/0295-5075/83/30004}.
\newblock URL \url{https://doi.org/10.12092F0295-50752F832F30004}.

\bibitem[Chiribella et~al.(2008{\natexlab{b}})Chiribella, D'Ariano, and
  Perinotti]{chiribella_architecture}
G.~Chiribella, G.~M. D'Ariano, and P.~Perinotti.
\newblock Quantum circuit architecture.
\newblock \emph{Phys. Rev. Lett.}, 101:\penalty0 060401, Aug
  2008{\natexlab{b}}.
\newblock \doi{10.1103/PhysRevLett.101.060401}.
\newblock URL \url{https://link.aps.org/doi/10.1103/PhysRevLett.101.060401}.

\bibitem[Chiribella et~al.(2010)Chiribella, D'Ariano, and Perinotti]{Chiri1}
G.~Chiribella, G.~M. D'Ariano, and P.~Perinotti.
\newblock Probabilistic theories with purification.
\newblock \emph{Physical Review A}, 81\penalty0 (6):\penalty0 062348, 2010.

\bibitem[Chiribella et~al.(2009)Chiribella, D'Ariano, and
  Perinotti]{chiribella_networks}
Giulio Chiribella, Giacomo~Mauro D'Ariano, and Paolo Perinotti.
\newblock Theoretical framework for quantum networks.
\newblock \emph{Phys. Rev. A}, 80:\penalty0 022339, Aug 2009.
\newblock \doi{10.1103/PhysRevA.80.022339}.
\newblock URL \url{https://link.aps.org/doi/10.1103/PhysRevA.80.022339}.

\bibitem[Chiribella et~al.(2013)Chiribella, D'Ariano, Perinotti, and
  Valiron]{chiribella_causal}
Giulio Chiribella, Giacomo~Mauro D'Ariano, Paolo Perinotti, and Benoit Valiron.
\newblock Quantum computations without definite causal structure.
\newblock \emph{Phys. Rev. A}, 88:\penalty0 022318, Aug 2013.
\newblock \doi{10.1103/PhysRevA.88.022318}.
\newblock URL \url{https://link.aps.org/doi/10.1103/PhysRevA.88.022318}.

\bibitem[Coecke(2005)]{Kindergarten}
B.~Coecke.
\newblock Kindergarten quantum mechanics.
\newblock In A.~Khrennikov, editor, \emph{Quantum Theory: Reconsiderations of
  the Foundations III}, pages 81--98. AIP Press, 2005.
\newblock {a}rXiv:quant-ph/0510032.

\bibitem[Coecke and Kissinger(2016)]{CKbook}
B.~Coecke and A.~Kissinger.
\newblock \emph{Picturing Quantum Processes. A First Course in Quantum Theory
  and Diagrammatic Reasoning}.
\newblock Cambridge University Press, 2016.

\bibitem[Coecke and Lal(2013)]{CRCaucat}
B.~Coecke and R.~Lal.
\newblock Causal categories: relativistically interacting processes.
\newblock \emph{Foundations of Physics}, 43:\penalty0 458--501, 2013.
\newblock arXiv:1107.6019.

\bibitem[Coecke et~al.(2014)Coecke, Fritz, and Spekkens]{CFS}
B.~Coecke, T.~Fritz, and R.~W. Spekkens.
\newblock A mathematical theory of resources.
\newblock \emph{Information and Computation, to appear}, 2014.
\newblock arXiv:1409.5531.

\bibitem[Coecke and Paquette(2010)]{coecke2010categories}
Bob Coecke and Eric~Oliver Paquette.
\newblock Categories for the practising physicist.
\newblock In \emph{New structures for physics}, pages 173--286. Springer, 2010.

\bibitem[Coecke et~al.(2022)Coecke, Horsman, Kissinger, and
  Wang]{coecke2022kindergarden}
Bob Coecke, Dominic Horsman, Aleks Kissinger, and Quanlong Wang.
\newblock Kindergarden quantum mechanics graduates... or how i learned to stop
  gluing lego together and love the zx-calculus.
\newblock \emph{Theoretical Computer Science}, 897:\penalty0 1--22, 2022.

\bibitem[D'Ariano et~al.(2017)D'Ariano, Chiribella, and
  Perinotti]{d2017quantum}
Giacomo~Mauro D'Ariano, Giulio Chiribella, and Paolo Perinotti.
\newblock \emph{Quantum theory from first principles: an informational
  approach}.
\newblock Cambridge University Press, 2017.

\bibitem[Earnshaw et~al.(2023)Earnshaw, Hefford, and
  Rom{\'a}n]{earnshaw_produoidal}
Matt Earnshaw, James Hefford, and Mario Rom{\'a}n.
\newblock The produoidal algebra of process decomposition.
\newblock \emph{arXiv}, 2023.
\newblock \doi{10.48550/ARXIV.2301.11867}.
\newblock URL \url{https://arxiv.org/abs/2301.11867}.

\bibitem[Fong and Spivak(2018)]{fong2018seven}
Brendan Fong and David~I Spivak.
\newblock Seven sketches in compositionality: An invitation to applied category
  theory.
\newblock \emph{arXiv preprint arXiv:1803.05316}, 2018.

\bibitem[Fritz(2015)]{fritz2015convex}
Tobias Fritz.
\newblock Convex spaces i: Definition and examples.
\newblock 2015.

\bibitem[Hardy(2001)]{HardyAxiom}
L.~Hardy.
\newblock Quantum theory from five reasonable axioms.
\newblock \emph{arXiv:quant-ph/0101012}, 2001.

\bibitem[Harrigan and Spekkens(2010)]{harrigan1}
N.~Harrigan and R.~W. Spekkens.
\newblock Einstein, incompleteness, and the epistemic view of quantum states.
\newblock \emph{Foundations of Physics}, 40:\penalty0 125--157, 2010.

\bibitem[Hefford and Comfort(2022)]{hefford_coend}
James Hefford and Cole Comfort.
\newblock Coend optics for quantum combs.
\newblock \emph{arXiv}, 2022.
\newblock \doi{10.48550/ARXIV.2205.09027}.
\newblock URL \url{https://arxiv.org/abs/2205.09027}.

\bibitem[Hermida and Tennent(2012)]{hermida2012monoidal}
Claudio Hermida and Robert~D Tennent.
\newblock Monoidal indeterminates and categories of possible worlds.
\newblock \emph{Theoretical Computer Science}, 430:\penalty0 3--22, 2012.

\bibitem[Hoffreumon and Oreshkov(2022)]{Hoffreumon_projective}
Timoth{\'e}e Hoffreumon and Ognyan Oreshkov.
\newblock Projective characterization of higher-order quantum transformations.
\newblock \emph{arXiv}, 2022.
\newblock \doi{10.48550/ARXIV.2206.06206}.
\newblock URL \url{https://arxiv.org/abs/2206.06206}.

\bibitem[Huot and Staton(2019)]{Huot_2019}
Mathieu Huot and Sam Staton.
\newblock Universal properties in quantum theory.
\newblock \emph{Electronic Proceedings in Theoretical Computer Science},
  287:\penalty0 213--223, jan 2019.
\newblock \doi{10.4204/eptcs.287.12}.

\bibitem[Kissinger and Uijlen(2017)]{kissinger2017categorical}
Aleks Kissinger and Sander Uijlen.
\newblock A categorical semantics for causal structure.
\newblock In \emph{Logic in Computer Science (LICS), 2017 32nd Annual ACM/IEEE
  Symposium on}, pages 1--12. IEEE, 2017.

\bibitem[Kissinger and Uijlen(2019)]{kissinger_causal}
Aleks Kissinger and Sander Uijlen.
\newblock {A categorical semantics for causal structure}.
\newblock \emph{{Logical Methods in Computer Science}}, {Volume 15, Issue 3},
  August 2019.
\newblock \doi{10.23638/LMCS-15(3:15)2019}.
\newblock URL \url{https://lmcs.episciences.org/5681}.

\bibitem[Mac~Lane(1963)]{MacLaneCoherence}
S.~Mac~Lane.
\newblock Natural associativity and commutativity.
\newblock \emph{The Rice University Studies}, 49\penalty0 (4):\penalty0 28--46,
  1963.

\bibitem[Mac~Lane(1998)]{MacLane}
S.~Mac~Lane.
\newblock \emph{Categories for the working mathematician}.
\newblock Springer-verlag, 1998.

\bibitem[Oeckl and Almada(2024)]{GBQFT}
Robert Oeckl and Juan~Orendain Almada.
\newblock Compositional quantum field theory: An axiomatic presentation.
\newblock 2024.
\newblock URL \url{https://arxiv.org/pdf/2208.10385}.

\bibitem[Oreshkov and Cerf(2015)]{oreshkov2015operational}
Ognyan Oreshkov and Nicolas~J Cerf.
\newblock Operational formulation of time reversal in quantum theory.
\newblock \emph{Nature Physics}, 11\penalty0 (10):\penalty0 853--858, 2015.

\bibitem[Oreshkov and Cerf(2016)]{oreshkov2016operational}
Ognyan Oreshkov and Nicolas~J Cerf.
\newblock Operational quantum theory without predefined time.
\newblock \emph{New Journal of Physics}, 18\penalty0 (7):\penalty0 073037,
  2016.

\bibitem[Or{\'u}s(2014)]{orus2014practical}
Rom{\'a}n Or{\'u}s.
\newblock A practical introduction to tensor networks: Matrix product states
  and projected entangled pair states.
\newblock \emph{Annals of physics}, 349:\penalty0 117--158, 2014.

\bibitem[Patterson et~al.(2021)Patterson, Spivak, and
  Vagner]{patterson2021wiring}
Evan Patterson, David~I Spivak, and Dmitry Vagner.
\newblock Wiring diagrams as normal forms for computing in symmetric monoidal
  categories.
\newblock \emph{arXiv preprint arXiv:2101.12046}, 2021.

\bibitem[Portmann et~al.(2017)Portmann, Matt, Maurer, Renner, and
  Tackmann]{Portmann_2017_causal_box}
Christopher Portmann, Christian Matt, Ueli Maurer, Renato Renner, and Bjorn
  Tackmann.
\newblock Causal boxes: Quantum information-processing systems closed under
  composition.
\newblock \emph{{IEEE} Transactions on Information Theory}, pages 1--1, 2017.
\newblock \doi{10.1109/tit.2017.2676805}.

\bibitem[rin et~al.(2019)rin, Krumm, Budroni, and {C}aslav
  Brukner]{Guerin_2019_nogo}
Philippe Allard~Gue rin, Marius Krumm, Costantino Budroni, and {C}aslav
  Brukner.
\newblock Composition rules for quantum processes: a no-go theorem.
\newblock \emph{New Journal of Physics}, 21\penalty0 (1):\penalty0 012001, jan
  2019.
\newblock \doi{10.1088/1367-2630/aafef7}.

\bibitem[Salzger and Vilasini(2024)]{salzger2024mapping}
Matthias Salzger and V~Vilasini.
\newblock Mapping indefinite causal order processes to composable quantum
  protocols in a spacetime.
\newblock \emph{arXiv preprint arXiv:2404.05319}, 2024.

\bibitem[Schmid et~al.(2020{\natexlab{a}})Schmid, Selby, Pusey, and
  Spekkens]{schmid2020structure}
David Schmid, John~H Selby, Matthew~F Pusey, and Robert~W Spekkens.
\newblock A structure theorem for generalized-noncontextual ontological models.
\newblock \emph{arXiv preprint arXiv:2005.07161}, 2020{\natexlab{a}}.

\bibitem[Schmid et~al.(2020{\natexlab{b}})Schmid, Selby, and
  Spekkens]{schmid2020unscrambling}
David Schmid, John~H Selby, and Robert~W Spekkens.
\newblock Unscrambling the omelette of causation and inference: The framework
  of causal-inferential theories.
\newblock \emph{arXiv preprint arXiv:2009.03297}, 2020{\natexlab{b}}.

\bibitem[Selby et~al.(2022)Selby, Stasinou, Gogioso, and Coecke]{timesymmetry}
John~H Selby, Maria~E Stasinou, Stefano Gogioso, and Bob Coecke.
\newblock Time symmetry in quantum theories and beyond.
\newblock \emph{arXiv preprint arXiv:2209.07867}, 2022.

\bibitem[Simmons and Kissinger(2022)]{simmons_bv_logic}
Will Simmons and Aleks Kissinger.
\newblock Higher-order causal theories are models of bv-logic.
\newblock \emph{arXiv}, 2022.
\newblock \doi{10.48550/ARXIV.2205.11219}.
\newblock URL \url{https://arxiv.org/abs/2205.11219}.

\bibitem[Spekkens(2007)]{SpekToy}
R.~W. Spekkens.
\newblock Evidence for the epistemic view of quantum states: A toy theory.
\newblock \emph{Physical Review A}, 75\penalty0 (3):\penalty0 032110, 2007.

\bibitem[Spekkens(2016)]{spekkens2016quasi}
Robert~W Spekkens.
\newblock Quasi-quantization: classical statistical theories with an epistemic
  restriction.
\newblock \emph{Quantum Theory: Informational Foundations and Foils}, pages
  83--135, 2016.

\bibitem[Spivak(2013)]{spivak2013operad}
David~I Spivak.
\newblock The operad of wiring diagrams: Formalizing a graphical language for
  databases, recursion, and plug-and-play circuits.
\newblock \emph{arXiv preprint arXiv:1305.0297}, 2013.

\bibitem[Spivak et~al.(2017)Spivak, Schultz, and Rupel]{spivak2017string}
David~I Spivak, Patrick Schultz, and Dylan Rupel.
\newblock String diagrams for traced and compact categories are oriented
  1-cobordisms.
\newblock \emph{Journal of Pure and Applied Algebra}, 221\penalty0
  (8):\penalty0 2064--2110, 2017.

\bibitem[Szabo(1975)]{szabo_polycats}
M.E. Szabo.
\newblock Polycategories.
\newblock \emph{Communications in Algebra}, 3\penalty0 (8):\penalty0 663--689,
  1975.
\newblock \doi{10.1080/00927877508822067}.
\newblock URL \url{https://doi.org/10.1080/00927877508822067}.

\bibitem[van~de Wetering(2020)]{van2020zx}
John van~de Wetering.
\newblock Zx-calculus for the working quantum computer scientist.
\newblock \emph{arXiv preprint arXiv:2012.13966}, 2020.

\bibitem[Vanrietvelde et~al.(2022)Vanrietvelde, Ormrod, Kristj{\'a}nsson, and
  Barrett]{vanreitvelde_consistent_circuit}
Augustin Vanrietvelde, Nick Ormrod, Hl{\'e}r Kristj{\'a}nsson, and Jonathan
  Barrett.
\newblock Consistent circuits for indefinite causal order.
\newblock \emph{arXiv}, 2022.
\newblock \doi{10.48550/ARXIV.2206.10042}.
\newblock URL \url{https://arxiv.org/abs/2206.10042}.

\bibitem[Vilasini and Renner(2022)]{venkatesh_cyclic_spacetime}
V.~Vilasini and Renato Renner.
\newblock Embedding cyclic causal structures in acyclic spacetimes: no-go
  results for process matrices.
\newblock \emph{arXiv}, 2022.
\newblock \doi{10.48550/ARXIV.2203.11245}.
\newblock URL \url{https://arxiv.org/abs/2203.11245}.

\bibitem[Wang-Mascianica et~al.(2023)Wang-Mascianica, Liu, and
  Coecke]{wang2023distilling}
Vincent Wang-Mascianica, Jonathon Liu, and Bob Coecke.
\newblock Distilling text into circuits.
\newblock \emph{arXiv preprint arXiv:2301.10595}, 2023.

\bibitem[Wilson and Chiribella(2022)]{wilson_polycat}
Matt Wilson and Giulio Chiribella.
\newblock Free polycategories for unitary supermaps of arbitrary dimension.
\newblock \emph{arXiv}, 2022.
\newblock \doi{10.48550/ARXIV.2207.09180}.
\newblock URL \url{https://arxiv.org/abs/2207.09180}.

\bibitem[Wilson et~al.(2022)Wilson, Chiribella, and Kissinger]{wilson_local}
Matt Wilson, Giulio Chiribella, and Aleks Kissinger.
\newblock Quantum supermaps are characterized by locality.
\newblock \emph{arXiv}, 2022.
\newblock \doi{10.48550/ARXIV.2205.09844}.
\newblock URL \url{https://arxiv.org/abs/2205.09844}.

\bibitem[Yau(2008)]{yau2008higher}
Donald Yau.
\newblock Higher dimensional algebras via colored props.
\newblock \emph{arXiv preprint arXiv:0809.2161}, 2008.

\bibitem[Yau(2018)]{yau2018operads}
Donald Yau.
\newblock \emph{Operads of wiring diagrams}.
\newblock Springer, 2018.

\end{thebibliography}

\appendix

\section{Time neutrality from cups and caps}\label{app:cupsandcaps}
In this appendix we extend the acyclic wiring operad, $\cW^A$, by adding in new wirings in which inputs and outputs can be freely connected to one another. For example, unlike in $\cW^A$, we permit wirings such as:
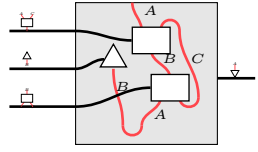
\[
\input{wiringTest7.tikz}.
\]
\begin{definition}
The cyclic wiring operad, $\cW^C$, has the same objects as the acylic wiring operad (i.e., boxes) and the same composition rule (i.e., diagram substitution), but has extra wirings, namely those that are cyclic.
\end{definition}

Note that we can construct any such cyclic wiring from the wirings in $\cW^A$ together with two additional core wirings: the wiring diagram $\textsf{cup}$,
\[
\input{cupop.tikz},
\]
and the wiring diagram $\textsf{cap}$,
\[
\input{capop.tikz}. 
\]
These allow us to freely interchange inputs with outputs of  boxes. For instance,
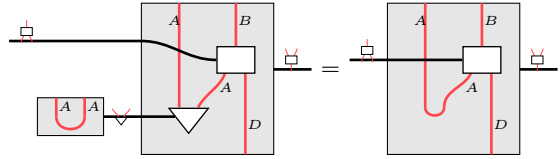
\[
\input{compactoperadexample.tikz} = \input{compactoperadexample1.tikz}, 
\]
in which the input $A$ of the box is turned into an output by applying a cup. 
The wirings $\textsf{cup}$ and $\textsf{cap}$ satisfy the following conditions:
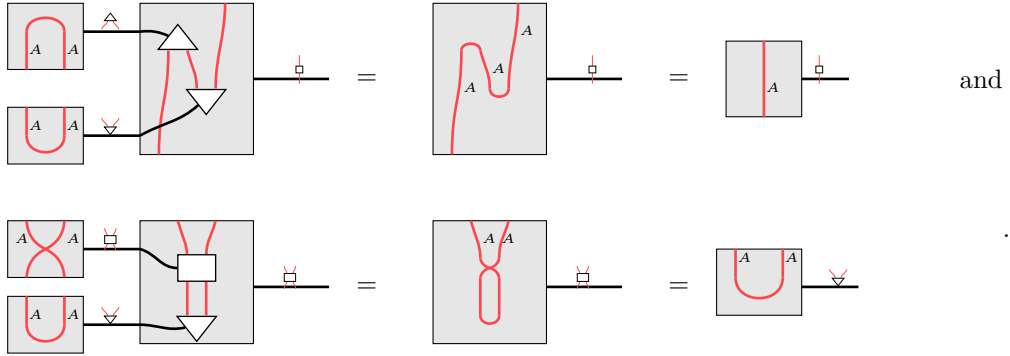
\[
\input{compactclosureoperad.tikz}. 
\]
The first equality in each of these follows from  the definition of composition as diagram substitution, and the second equality in each is really a tautology. For example,
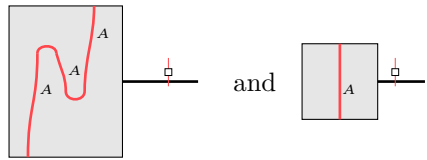
\[
\input{compactclosureoperad1.tikz}
\]
describe the exact same wiring. The reason for picking out these conditions, however, will be clear when we consider the operad algebras for this operad.

The algebras of the cyclic wiring operad $F:\cW^C \to \textsc{set}$ correspond to compact closed categories. In addition to the data specified by $\cW^A$, the functor $F$ also picks out special elements of the homsets $\mathcal{C}(I,A\otimes A)$ and $\mathcal{C}(A\otimes A, I)$. This is achieved via the functions
 $F(\textsf{cup})$ which is diagrammatically drawn as
 \[
 \input{cupFunctor.tikz}
 \]
 and 
$F(\textsf{cap})$ which is diagrammatically drawn as
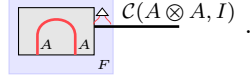
\[
 \input{capFunctor.tikz}.
 \]
These special morphisms in $\mathcal{C}(I,A\otimes A)$ and $\mathcal{C}(A\otimes A,I)$ can easily be shown to satisfy the conditions required to define a compact closed symmetric monoidal category. For example:
\begin{align}
& \input{ccFunct1.tikz} \\
&\quad = \input{ccFunct2.tikz} \\
&\qquad = \input{ccFunct3.tikz} \\
&\quad\qquad = \input{ccFunct4.tikz} 
= \input{ccFunct5.tikz} = \input{ccFunc6.tikz}.
\end{align}
One may be tempted to ask whether we can define a wiring operad which allows for both discarding as well as cups and caps. This, however, quickly runs into difficulties since the cap must be discarding and thus the resulting theory trivialises. In this sense these two extensions of $\cW^A$ are incompatible with one another.

To sum up, we have relaxed the constraints on the wirings in the wiring operad $\cW^A$ to define the cyclic wiring operad $\cW^C$,  such that the algebras of $\cW^C$ correspond to compact closed SMCs. Such process theories have been used, e.g., in Ref.~\cite{timesymmetry}, to describe time-neutral theories of physics. Arguably, however, such theories are not truly time-neutral as there is still a distinction between the inputs and outputs of a box regardless of the fact that they can be freely interchanged via cups and caps. There is however, a close connection to the time-neutral process theories that we defined in the main text as we discuss in the following subsection.

\subsection{Relating time-neutral process theories and compact closed SMCs}

In this section we demonstrate a tight connection between time-neutral process theories and compact closed SMCs (viewed as algebras of $\cW^C$), in order to relate our new definition of time-neutral theories with those in the literature.

More specifically, we will define a pair of operad functors  \colorbox{red!10}{$\alpha:\cW^C\to\cW^D$} and  \colorbox{green!10}{$\beta:\cW^D \to \cW^C$}. Where we find that $\alpha \circ \beta = I_{\cW^D}$ and that $\beta\circ \alpha \cong I_{\cW^C}$, hence showing that $\cW^C$ and $\cW^D$ are equivalent operads.  

At the level of the algebras we can use these functors to turn algebras for $\cW^{D}$ into algebras for $\cW^C$ and vice versa:
\[
\input{tn+cc.tikz}
\]
Specifically, the functor $\alpha$ allows us to map a time-neutral theory $G$ to an associated process theory with compact structure $F_G := G\circ\alpha$, whereas the functor $\beta$ allows us to map a process theory with compact structure $F$ to a time neutral theory  $G_F := F\circ\beta$. The equivalence of the two operads then immediately leads to an equivalence of the associated algebras. 

In this sense, it is therefore simply a matter of presentation as to whether to work with time-neutral process theories or to work with traditional process theories equipped with cups and caps. However, as we argued in the main text, time-neutral process theories appear to be presentationally the cleanest choice.  

To formalise this equivalence we first have to introduce natural transformations for operad functors. 
\begin{definition}\label{def:natTrans} A natural transformation $\eta: \colorbox{blue!10}{$F:\cO\to \cO'$} \implies \colorbox{red!10}{$G:\cO\to \cO'$}$ is defined as a collection of operad operations in $\cO'$ indexed by the objects in $\cO$. We denote these as $\eta_{t}$ and require that they satisfy:
\beq\label{eq:NatTran}
\input{figures/opNatTran.tikz}\quad = \quad \input{figures/opNatTran1.tikz}
\eeq
for all operations $f\in \cO$.
\end{definition}

Next let us define the functor \colorbox{red!10}{$\alpha:\cW^C \to \cW^{D}$}. Its action on objects is given by:
\[
\alpha \left(\input{wboxes.tikz}\!\!\! \right)  = \input{alphaDef1.tikz}.
\]
Note that this is not an injective mapping as it forgets the distinction between inputs and outputs. However, it preserves the planar orientation of systems. For example,
\[
\input{boxestodots1.tikz}\quad \text{whilst}\quad\input{boxestodots2.tikz}.
\]
The action of $\alpha$ on the wirings of $\cW^C$ simply gives the wiring in $\cW^{D}$ which has the same connectivity. For example:
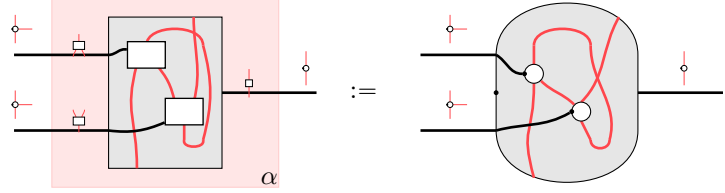
\[
\input{wdboxestodot.tikz}
\]
In particular, note that $\alpha$ acts to identify cups, caps, and identity wirings in $\cW^C$:
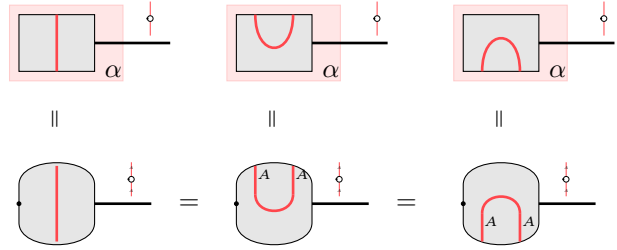
\[
\input{identifyCupsCapsIdents.tikz}
\]
such that cups and caps are simply redundant once working with the wiring operad of dots. 

It is easy to then see that $\alpha$ does indeed define an operad functor, for example:
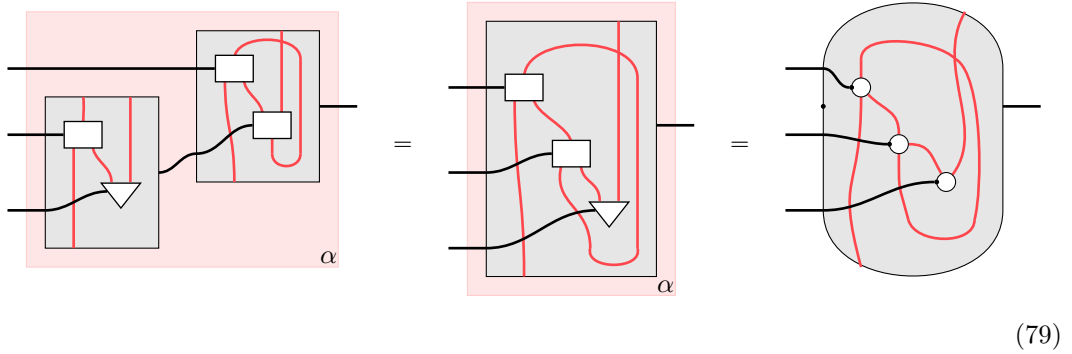
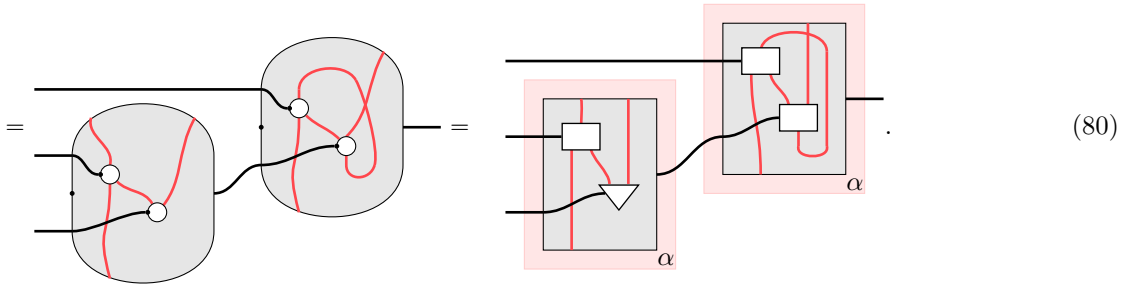
\begin{align}
\input{wdboxestodotFunct1.tikz} \quad &=\quad \input{wdboxestodotFunct2.tikz} \quad
=\quad \input{wdboxestodotFunct3.tikz} \\
= \input{wdboxestodotFunct4.tikz} 
&=\quad \input{wdboxestodotFunct5.tikz}.
\end{align}

Next let us define the functor \colorbox{green!10}{$\beta:\cW^D\to\cW^C$}. This is not quite so straightforward to define -- moving from $\cW^C$ to $\cW^{D}$ amounted to forgetting structure, that is, forgetting the input-output distinction, in order to go back from $\cW^{D}$ to $\cW^C$ we are therefore forced to artificially reintroduce this distinction. Indeed, there is a degree of arbitrariness as to how we should approach the problem. Here we will work with the convention that all of the systems are assigned to be outputs. That is, we take the action of $\beta$ on objects to be:
\[
\beta\left(\input{circle.tikz}\right) := \input{betaOnObjects.tikz}.
\]
It is clear that this fails to be a surjective mapping, as we only obtain boxes which have no inputs, hence really $\beta$ is faithfully mapping $\cW^D$ into a full suboperad of $\cW^C$. 

The action of $\beta$ then maps wirings of dots to wirings of boxes-without-inputs while preserving the connectivity. For example,
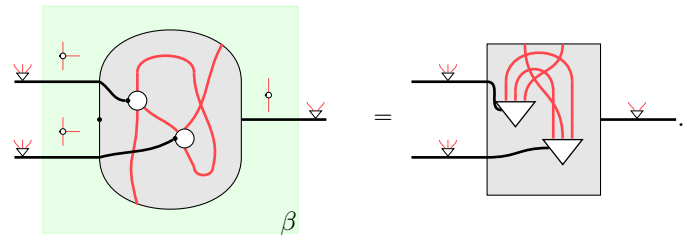
\[
\input{wddotstoboxes.tikz}.
\]
It is again easy to check that this is indeed an operad functor, for example:
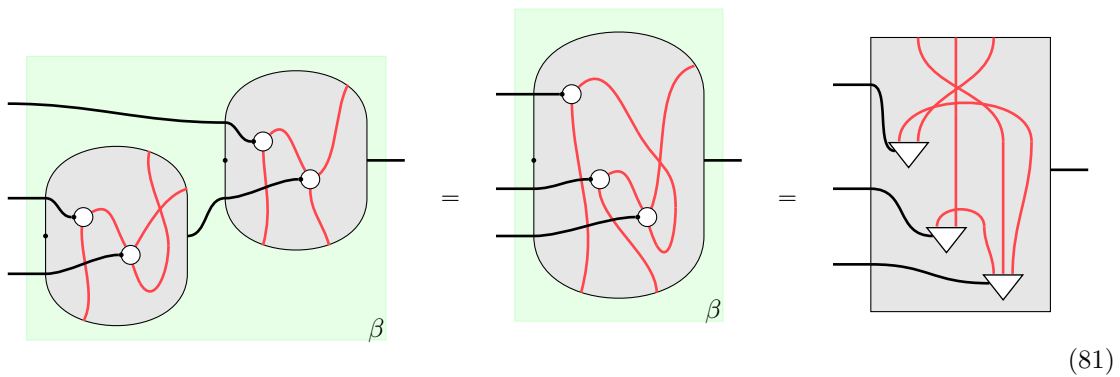
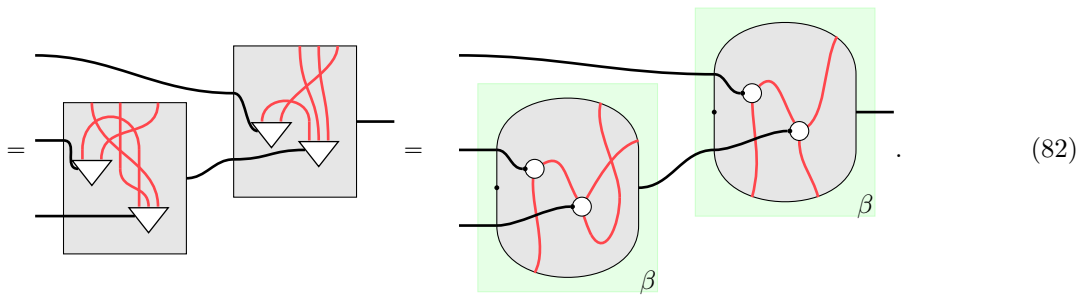
\begin{align}
\input{wddotstoboxesFunct1.tikz} \quad &=\quad \input{wddotstoboxesFunct2.tikz} \quad
=\quad \input{wddotstoboxesFunct3.tikz} \\
= \input{wddotstoboxesFunct4.tikz} 
&=\quad \input{wddotstoboxesFunct5.tikz}.
\end{align}

Given these functors $\alpha$ and $\beta$, as we discussed above, we can now use these to construct $\cW^C$ operad algebras from $\cW^D$ operad algebras and vice versa. What we want to know, however, is whether anything is lost by doing so.
A simple way to check this is to consider the ``round-trips'' that we can take, e.g., converting an algebra for $\cW^D$ into one for $\cW^C$ and back again and vice versa. If, in both cases, we end up with the same operad algebra that we started with then we know that neither mapping lost anything.
In other words, we want to understand the two possible compositions of the functors $\alpha$ and $\beta$.

The first possibility, namely the composite $\alpha\circ\beta$ is straightforward. Specifically, one finds that 
\beq
\alpha \circ \beta = I_{\cW^{D}}.
\eeq
This is not surprising, $\beta$ artificially adds in extra structure that $\alpha$ then forgets about again, leaving us with what we started with.
What this means that if we map some operad algebra for $\cW^C$ to an algebra for $\cW^D$ and back again, then we end up with the same operad algebra that we started with. 

It is clear, however, that the second composite, $\beta \circ \alpha$ is not so simple, as we have that $\beta \circ \alpha \neq I_{\cW^C}$. This is easy to see, for example, as it is not even injective on objects:
\beq\begin{tikzpicture}
	\begin{pgfonlayer}{nodelayer}
		\node [style=epiBox] (0) at (-5.5, 0) {};
		\node [style=none] (1) at (-5.5, 0.25) {};
		\node [style=none] (2) at (-5.5, -0.25) {};
		\node [style=none] (3) at (-5.5, 1.25) {};
		\node [style=none] (4) at (-5.5, -1.25) {};
		\node [style=right label] (5) at (-5.5, 0.75) {\tiny $B$};
		\node [style=right label] (6) at (-5.5, -0.75) {\tiny $A$};
		\node [style=none] (7) at (-4.5, 0) {};
		\node [style=none] (8) at (-4.5, 0) {,};
		\node [style=epiCopointWide] (9) at (-3.25, 0.5) {};
		\node [style=none] (10) at (-3.5, 0.25) {};
		\node [style=none] (11) at (-4, -0.5) {};
		\node [style=none] (12) at (-3, 0.25) {};
		\node [style=none] (13) at (-2.5, -0.5) {};
		\node [style=right label] (14) at (-3.75, -0.25) {\tiny $A$};
		\node [style=right label] (15) at (-2.75, -0.25) {\tiny $B$};
		\node [style=none] (16) at (-2, 0) {};
		\node [style=none] (17) at (-2, 0) {,};
		\node [style=epiPointWide] (18) at (-1, 0) {};
		\node [style=none] (19) at (-1.75, 0.75) {};
		\node [style=none] (20) at (-0.25, 0.75) {};
		\node [style=none] (21) at (-1.25, 0.25) {};
		\node [style=none] (22) at (-0.75, 0.25) {};
		\node [style=left label] (23) at (-1.5, 0.5) {\tiny $B$};
		\node [style=right label] (24) at (-0.5, 0.5) {\tiny $A$};
		\node [style=none] (25) at (0.5, 0) {};
		\node [style=none] (26) at (0.5, 0) {$\stackrel{\alpha}{\mapsto}$};
		\node [style=big white dot] (27) at (1.75, 0) {};
		\node [style=small black dot] (28) at (1.5, 0) {};
		\node [style=none] (30) at (1.75, -0.25) {};
		\node [style=none] (31) at (1.75, -1.25) {};
		\node [style=none] (32) at (1.75, 0.25) {};
		\node [style=none] (33) at (1.75, 1.25) {};
		\node [style=right label] (34) at (1.75, -0.75) {\tiny $A$};
		\node [style=right label] (35) at (1.75, 0.75) {\tiny $B$};
		\node [style=none] (36) at (3, 0) {$\stackrel{\beta}{\mapsto}$};
		\node [style=epiPointWide] (37) at (5, 0) {};
		\node [style=none] (38) at (4.25, 0.75) {};
		\node [style=none] (39) at (5.75, 0.75) {};
		\node [style=none] (40) at (4.75, 0.25) {};
		\node [style=none] (41) at (5.25, 0.25) {};
		\node [style=left label] (42) at (4.5, 0.5) {\tiny $B$};
		\node [style=right label] (43) at (5.5, 0.5) {\tiny $A$};
	\end{pgfonlayer}
	\begin{pgfonlayer}{edgelayer}
		\draw [qWire] (3.center) to (1.center);
		\draw [qWire] (2.center) to (4.center);
		\draw [qWire, bend left=15] (10.center) to (11.center);
		\draw [qWire, bend right] (12.center) to (13.center);
		\draw [qWire, bend left] (19.center) to (21.center);
		\draw [qWire, bend left] (22.center) to (20.center);
		\draw [qWire] (33.center) to (32.center);
		\draw [qWire] (30.center) to (31.center);
		\draw [qWire, bend left] (38.center) to (40.center);
		\draw [qWire, bend left] (41.center) to (39.center);
	\end{pgfonlayer}
\end{tikzpicture}
\eeq
What we can show, however, is the next best thing. That is, that $\beta\circ\alpha$ and $I_{\cW^C}$ are naturally isomorphic to one another -- there is a natural isomorphism $\eta : I_{\cW^C} \to \beta \circ \alpha$ as we will now demonstrate. As discussed in Def.~\ref{def:natTrans} To define $\eta$ we must define a particular family of operations in $\cW^C$ indexed by the objects in $\cW^C$. In particular, we take these to be the operations that map an input box to an output state, in a way that preserves the planar ordering of the systems, for example:
\beq
\input{figures/natTran1.tikz}.
\eeq
Note that such operations are invertible, for example, the above operation has an inverse:
\beq
\input{figures/natTran3.tikz},
\eeq
hence, if these operations do indeed define a natural transformation, then it is in fact a natural isomorphism.
To show that this is the case, we must show that the following holds for all wirings in $\cW^C$: 
\beq
\input{figures/natTran4.tikz}\ \ =\ \ \input{figures/natTran5.tikz}
\eeq
This is straightforwaqrd to verify. The RHS is immediately equal to:
\beq
\input{figures/natTran6.tikz}
\eeq
from the definition of composition in the operad.
Similarly, we can show that the LHS also is equal to this:
\begin{align}
\input{figures/natTran7.tikz} &= \input{figures/natTran8.tikz}\\
= \input{figures/natTran9.tikz} 
&= \input{figures/natTran6.tikz},
\end{align}
which follows immediately from the definitions of the two functors as well as composition in the operad.

It therefore follows that $\alpha$ and $\beta$ are defining an equivalence between $\cW^C$ and $\cW^D$. Moreover, the natural isomorphism $\eta: \mathds{1}_{\cW^C} \to \beta \circ \alpha$ also implies that any algebra $F:\cW^C \to \textsc{set}$ is naturally isomorphic to the algebra $F\circ \beta \circ \alpha:\cW^C \to \textsc{set}$.

It is worth looking somewhat closer at why this is ``just'' an isomorphism rather than them  being identical. Consider some algebra $F:\cW^C\to\textsc{set}$, this maps the box $\input{boxAB.tikz}$ to a homset $\mathcal{C}_F(A,B)$ and maps the box $\input{stateAB.tikz}$ to $\mathcal{C}_F(I,A\otimes B)$. In contrast, $F\circ \beta \circ \alpha$ maps both boxes to $\mathcal{C}_F(I,A\otimes B)$. At first glance this may seem problematic, but, the fact that this category is compact closed means that we have that $\mathcal{C}_F(A,B)\cong\mathcal{C}_F(I,A\otimes B)$. Therefore, what we are really doing when we move from algebras of $\cW^C$ to $\cW^D$ and back is to identify homsets which were related by the isomorphisms provided by the cups and caps.

\section{Operads with a notion of empty space}
In this section we will see how to give meaning to the adapted operad diagrams of the main text, in which the empty list is considered as a legitimate output, with a unique operation into the empty output referred to as \textit{operadic discarding}. In fact, the process can be pieced together from the literature. Beginning with the adjunctions \cite{yau2008higher, Huot_2019, hermida2012monoidal} \beq
\input{affine_prop_1.tikz}
\eeq 
and noting that $\mathcal{R}_2(\textsc{set}_{ \groy \mathbf{Prop}}) =  \textsc{set}_{ \mathbf{Prop}}$ and $\mathcal{R}_1 (\textsc{set}_{\mathbf{Prop}}) = \textsc{set}_{\mathbf{Op}}$, we can see that affine prop algebras over suitably constructed affine props are equivalent to algebras over underlying operads:
\begin{align*}
\mathbf{Alg}_{\groy \mathbf{Prop}}[\mathcal{L}_2(\mathcal{L}_1(\mathcal{O}))] & := \groy\mathbf{Prop}[\mathcal{L}_2(\mathcal{L}_1(\mathcal{O})) , \textsc{set}_{ \groy \mathbf{Prop}}] \\
& \cong  \mathbf{Prop}[\mathcal{L}_1(\mathcal{O})  , \mathcal{R}_2( \textsc{set}_{ \groy \mathbf{Prop}})] \\
& \cong  \mathbf{Prop}[\mathcal{L}_1(\mathcal{O}) , \textsc{set}_{\mathbf{Prop}}] \\
& \cong \mathbf{Op}[\mathcal{O} , \mathcal{R}_1(\textsc{set}_{\mathbf{Prop}})]  \\
& \cong \mathbf{Op}[\mathcal{O} , \textsc{set}_{\mathbf{Op}}] \\
& \cong \mathbf{Alg}_{\mathbf{Op}}[\mathcal{O}] .
\end{align*}
As we will outline below for clarity, $\mathcal{L}_1$ is the free construction from operads to props, and $\mathcal{L}_2$ outlined below is the free affine completion from props to affine props . On the other hand the $\mathcal{R}_i$ are forgetful, $\mathcal{R}_2$ simply forgets the affine structure of a prop and $\mathcal{R}_1$ forgets the multi-output morphisms of a prop. 
For purely pedagogical purposes, in the remainder of this subsection we will outline the main components of the free constructions $\mathcal{L}_i$.  
\begin{lemma}
    Every operad $\mathcal{O}$ can be used to construct a prop $\mathcal{L}_1[\mathcal{O}]$
\end{lemma}
\begin{proof}
    Objects of $\mathcal{L}_1[\mathcal{O}]$ are given by lists of objects of $\mathcal{O}$. Morphisms of type $[A_{11} \dots A_{mn_{m}}] \rightarrow [B_1 \dots B_m]$ in $\mathcal{L}_1[\mathcal{O}]$ are given by lists $[f_1 \dots f_m]$ of operations in $\mathcal{O}$ with $f_{k}:[A_{k1} \dots A_{kn_{k}}] \rightarrow B_{k}$ in $\mathcal{O}$. Sequential composition is defined component-wise. The monoidal product on both objects and morphisms is given by list concatenation, with monoidal unit given by the empty list $[]$. For any non-empty list $L$ the set $\mathcal{L}_1[\mathcal{O}](L,[])$ is defined to be empty, with $\mathcal{L}_1[\mathcal{O}]([],[])$ a singleton set.
\end{proof}
This assignment lifts to functors, hence, it lifts operad algebras to prop algebras. Functoriality will show that $\mathcal{L}_1[F]: \mathcal{L}_1[\mathcal{O}] \rightarrow \mathcal{L}_1[\textsc{set}]$, for pedagogical purposes we will instead outline the induced algebra $\underline{\mathcal{L}_1}[F]: \mathcal{L}_1[\mathcal{O}] \rightarrow \textsc{set}_{\mathbf{Prop}}$.
\begin{lemma}
Let $F: \mathcal{O} \rightarrow \textsc{set}$ be an operad algebra, then $F$ extends to a morphism $\underline{\mathcal{L}_1}[F]: \mathcal{L}_1[\mathcal{O}] \rightarrow \textsc{set}_{\mathbf{Prop}}$ of props.
\end{lemma}
\begin{proof}
On objects define $\underline{\mathcal{L}_1}[F]A = FA$ .
On operations define $\underline{\mathcal{L}_1}[F][f_1 \dots f_m] = Ff_1 \times \dots \times Ff_m$. Functoriality is routine to check.
\end{proof}
We now show how to introduce operadic discarding
\begin{lemma}[Affine Reflection]
    Every monoidal category $\mathcal{M}$ can be used to construct an affine monoidal category $\mathcal{L}_2[\mathcal{M}]$ \cite{Huot_2019}. This construction furthermore preserves props. 
\end{lemma}
\begin{proof}
Let us outline the construction, consider the set $\bigcup_{X}\mathcal{M}(A,B \otimes X)$ and the relation \[f \sim g \iff \exists h:X \rightarrow Y \ s.t \ (\mathds{1}_B \otimes h) \circ f = g \] and take the symmetric-then-transitive closure to construct an equivalence relation $\cong$. The morphisms of $\mathcal{L}_2[\mathcal{M}]$ are given by \[  \mathcal{L}_2[\mathcal{M}](A,B) = \nicefrac{\bigcup_{X}\mathcal{M}(A, B \otimes X )}{\cong}.    \] It easy to check that sequential and parallel composition commute with these equivalence classes, so that one can define \[[f] \circ [g] := \left[\input{figures/affine_completion_2.tikz}\right] \] and \[[f] \otimes [g] := \left[\input{figures/affine_completion_3.tikz}\right].\] Note that $\mathcal{L}_2[\mathcal{M}]$ is an affine monoidal category meaning that $ \mathcal{L}_2[\mathcal{M}](A,[])$ is a singleton for every $A$. Indeed, focusing on the case in which $\mathcal{M}$ is a prop for simplicity, morphisms of type $A \rightarrow []$ in $\groy[\mathcal{M}]$ are morphisms of type $A \rightarrow [] \otimes X = X$ where taking the quotient with respect to $\cong$ makes all such morphisms equal since 
\begin{align*}
  & f :  A \rightarrow [] \otimes X \\
  = &  (\mathds{1}_{[]} \otimes \hat{f}) \circ \mathds{1}_{A}  : A \rightarrow [] \otimes X (= X) \\
   \cong & (\mathds{1}_{[]} \otimes \mathds{1}_{A}) \circ \mathds{1}_{A}:  A \rightarrow [] \otimes A (= A). 
\end{align*}
\end{proof}
This unique morphism into the unit is what we refer to in the main text as the \textit{operadic discarding} and denoted as \[
\input{operadicdisc.tikz}. 
\] In this appendix we will choose to represent this morphism with a single black dot $\bullet : A \rightarrow []$ to save space. Now consider the affine completion, we show how to lift algebras, not that morphism of affine props are defined simply to be morphisms of their underlying props, since monoidal functors always preserve affine structure. 
\begin{lemma}
Let $F: \mathcal{M} \rightarrow \textsc{set}_{\mathbf{Prop}}$ be a morphism of props, then $F$ can be extended to a morphism $\underline{\mathcal{L}}_2[F]: \mathcal{L}_2[\mathcal{M}] \rightarrow  \textsc{set}_{ \groy \mathbf{Prop}}$ of affine props.
\end{lemma}
\begin{proof}
    One can define \[ \underline{\mathcal{L}}_2[F][f:A \rightarrow B \otimes X]:= \input{figures/affine_functor_1.tikz}.  \] This is well-formed since for each $x \sim y$ then by the affine structure of the codomain $ \underline{\mathcal{L}}_2[F][x] =  \groy[F][y]$.
Now, for each $f \cong g$ (meaning there exists some chain $x_{(i)}$ such that $f \sim x_1 \sim \dots \sim x_n \sim g $) then $F(f) = F(x_1) = \dots = F(x_n) = F(g)$. Functorality is easily checked, for sequential composition we have: \[  \input{figures/affine_functor_2.tikz} \ = \ \input{figures/affine_functor_3.tikz} \] \[ = \ \input{figures/affine_functor_4.tikz} \ = \ \input{figures/affine_functor_5.tikz}.    \] For parallel composition we have \[   \input{figures/affine_par_1.tikz} \ = \ \input{figures/affine_par_2.tikz} \  = \ \input{figures/affine_par_3.tikz}.    \] 
\end{proof}
Given an operad $\mathcal{O}$ and an operad algebra $F : \mathcal{O} \rightarrow \textsc{set}$ we will refer to the affine completion of the free prop on $\mathcal{O}$ and $F$ by $\mathcal{L}_2[\mathcal{L}_1 [\mathcal{O}]]$ and $\underline{\mathcal{L}}_2[ \underline{\mathcal{L}}_1[F]]$ respectively.

\subsection{On discard-wiring operads}
The key motivation for working with affine completions of operads, is to give us operadic discarding with respect to which we can define causality. Let us begin by defining the standard operad $\cW_{\groy}$ in which $\groy$ 
is introduced.
\begin{definition}[Discard Wiring Operad]
    $\cW_{\groy}$ has morphisms given by those of the standard wiring operad along with an additional morphism $\groy$
    denoted \beq
\input{figures/operadicdisc2.tikz}. 
\eeq with operad associativity and no other equations imposed. 
\end{definition}

\begin{lemma}
    Every causal category $\mathcal{C}$ can be used to construct an operad algebra $F: \cW_{\groy} \rightarrow \textsc{set}$ on objects by $F(A,B) = \mathcal{C}(A,B)$ and on morphisms by \beq
\input{figures/operadicdisc3.tikz} = :  \groy \in \mathcal{C}(A,I) 
\eeq
\end{lemma}
We want to impose the following equation \beq
\input{discexample.tikz} = \input{operadicdisccondition.tikz}
\eeq 
to construct the causal wiring operad, but it isn't apriori clear that we can just ``impose an equation'' and still have a monoidal category. We confirm that this is in fact always possible, to do so we will need to construct an equivalence relation which plays well with sequential and parallel composition. 
\begin{definition}[Monoidal Congruence]
Let $\mathcal{C}$ be a monoidal category, an equivalence relation $\cong$ is a monoidal congruence on $\mathcal{C}$ if 
\begin{itemize}
    \item $f \cong g \wedge f' \cong g' \implies f \circ f' \cong g \circ g' $
    \item $f \cong g \wedge f' \cong g' \implies f \otimes f' \cong g \otimes g' $
\end{itemize}
Whenever such expression are well-typed. 
\end{definition}

\begin{theorem}[Monoidal Congruence Categories]
Every monoidal congruence $\cong$ on a monoidal category $\mathcal{C}$ be used to construct a monoidal category $\nicefrac{\mathcal{C}}{\cong}$ with \[  \nicefrac{\mathcal{C}}{\cong}(A,B) := \nicefrac{\mathcal{C}(A,B)}{\cong}.   \] All compositional structure is inherited from $\mathcal{C}$.
\end{theorem}
\begin{proof}
    One can define $[f] \otimes [g] := [f \otimes g]$ and $[f] \circ [g] := [f \circ g]$, with $[f]$ these equivalence class on $f$ with respect to $\cong$. These expression are well formed whenever $\cong$ is a monoidal congruence. 
\end{proof}
In light of this theorem, we see that if we can construct a monoidal congruence from an equation we want to impose, we will be able to construct a new monoidal category in which that new equation is satisfied. 

\begin{theorem}[From Equations to Monoidal Congruence]
Let $h,h'$ be two morphisms of a monoidal category and we want to impose the equation $h=h'$. To do so, let us define a relation by $f \approx g \iff \exists a_{i}, z_{i}, z_{i}'$ such that \[ f = \input{figures/quotient_1.tikz}  \] and \[ g = \input{figures/quotient_2.tikz},  \] where for each $i$ then either $z_i = h$ and $z_i ' = h'$ or $z_i = h'$ and $z_i ' = h$. The transitive closure $\cong$ of $\approx$ is a monoidal congruence. We will refer to the resulting quotient monoidal category $\nicefrac{\mathcal{C}}{\cong}$ as the monoidal category in which equation $h = h'$ has been imposed. 
\end{theorem}
\begin{proof}
    We need to prove that when $f \cong g$ and $f' \cong g'$ then $f \circ f' \cong g \circ g'$ and $f \otimes f' \cong g \otimes g'$. It is sufficient the check for sequential composition and to check that $f \cong g \implies \mathds{1} \otimes f \otimes \mathds{1} \cong \mathds{1} \otimes g \otimes \mathds{1}$
    For the sequential composition check, note first that $f \approx g$ and $f' \approx g'$ implies that $f \circ  f' \approx g \circ g'$, indeed the relevant decomposition of $f \circ f'$ is given by separately decomposing $f$ and $f'$, similarly for $g \circ g'$. Now for the transitive closure, note that $f \cong g$ and $f' \cong g'$ implies that $f \approx x_1 \dots \approx x_n \approx g$ and $f' \approx y_1 \dots \approx y_n \approx g'$. Now note that using the compatibility of $\approx$ with sequential composition we have 
    \begin{align*}
    f \circ f' &  \approx f \circ y_1 
     \approx \dots \approx f \circ y_m 
     \approx f \circ g' 
     \approx x_1 \circ g' 
     \dots \approx x_n \circ g' 
     \approx f' \circ g'.
    \end{align*}
    For parallel composition note that $f \approx g \implies \mathds{1} \otimes f \otimes \mathds{1} \approx \mathds{1} \otimes g \otimes \mathds{1}$. Now consider $f \cong g$ note that $ f \approx x_1 \dots \approx x_n \approx g$ and so $f \otimes \mathds{1} \approx x_1 \otimes \mathds{1} \dots \approx x_n \otimes \mathds{1} \approx g \otimes \mathds{1}$ and similarly on the left hand side, this completes the proof. 
\end{proof}

\begin{lemma}
    Let $\cong$ be a monoidal congruence on $\mathcal{C}$ and let $F: \mathcal{C} \rightarrow \mathcal{D}$ be a morphism of props such that \[ f \cong g \implies F(f) = F(g),     \] then $F$ can be extended to a morphism of props $\nicefrac{F}{\cong}: \nicefrac{\mathcal{C}}{\cong} \rightarrow \mathcal{D}$.
\end{lemma}
\begin{proof}
Define $\nicefrac{F}{\cong}[f] := F(f)$, this is well formed since if $[f] = [g]$ then $F[f] = F[g]$. Note that functorality is directly inherited from $F$:
\begin{align*}
    \nicefrac{F}{\cong}([f] \circ [g]) & = \nicefrac{F}{\cong}([f \circ g]) \\
    & = F(f \circ g)\\
    & = F(f) \circ F(g) \\
    & = \nicefrac{F}{\cong}([f]) \circ \nicefrac{F}{\cong}([g]).
\end{align*}
Compatibility with parallel composition follows similarly. 
\end{proof}

\begin{theorem}
    Let $F: \mathcal{C} \rightarrow \mathcal{D}$ be a mophism of props and let $h,h'$ be such that $F(h) = F(h')$, then $F$ can be extended to a morphism of props $\nicefrac{F}{\cong(\approx)}: \nicefrac{\mathcal{C}}{\cong(\approx)} \rightarrow \mathcal{D}$. 
\end{theorem}
\begin{proof}
    We need to show that since $F(h) = F(h')$, it follows that whenever $f \cong g$ then $F(f) = F(g)$. Note that when $f \cong g$ there exists some sequence $x_{(k)}$ of processes such that $f \approx x_1 \approx \dots \approx x_{n} \approx g$, so, if we can show that $F(y) = F(z)$ whenever $y \approx z$ then we have $F(f) = F(x_1) = \dots = F(x_n) = F(g)$. For any $y,z$ such that $y \approx z$ then there exists decomposition in terms of $a_{(i)},z_{(i)},z_{(i)}'$ such that $F(y)$ is equal to \[   \input{figures/affine_sequence_functor_1.tikz}    \] by functorality is equal to
    \[  \ \input{figures/affine_sequence_functor_2.tikz} ,   \]
    which by the property $F(h) = F(h')$ is equal to
    \[    \input{figures/affine_sequence_functor_3.tikz},    \]
    which again by functorality is given by
    \[   \input{figures/affine_sequence_functor_4.tikz},    \] in other words, $F(y)$ is equal to $F(z)$.
\end{proof}

\begin{definition}[The Causal Wiring Operad]
    The causal wiring operad $\cW^{\groy}$ introduced in the main text, is given by imposing the following equation: 
 \beq
\input{discexample.tikz} = \input{operadicdisccondition.tikz}
\eeq 
On the affine completion of the discard wiring operad. 
\end{definition}

We wish to show that every causal category is an algebra for the causal wiring operad. 

\begin{theorem}
    Every causal category defines an algebra $F:\cW^{\groy} \rightarrow \textsc{set}$
\end{theorem}
\begin{proof}
    The operadic algebra we have already constructed on the standard discarding wiring operad satisfies: \[   \input{opDiscAlg5.tikz} 
 =  \input{opDiscAlg6.tikz}. \] Consequently the previous theorem tells us that it extends to an algebra $F:\cW^{\groy} \rightarrow \textsc{set}$.
\end{proof}

\section{Polycategories With Space}

For the following let $\mathcal{P}_{\times }$ be the polycategory with objects given by those of $\mathcal{P}$ and with polymorphisms $p: L \rightarrow R$ given by a permutation $\sigma_{L}: L_1 \cup L_2 \rightarrow L$ for $L$ and analogously for $R$ along with pairs of processes $p_1 :L_1 \rightarrow R_1,p_2: L_2 \rightarrow R_2$ from $\mathcal{P}$. The polycategorical composition rule is inherited directly from polycategory composition of $\mathcal{P}$ up to keeping track of permutations. 
\begin{theorem}[Polycategories with space]
A polycategory with space is a polycategory equipped with an identity-on-objects polyfunctor $\otimes: \mathcal{P}_{\times} \rightarrow \mathcal{P} $ such that $(f \otimes g) \otimes h = f \otimes (g \otimes h)$. 
\end{theorem}
\begin{proof} (Sketch) 
Let us begin by noting that unpacking the polyfunctor returns parallel composition rules \[\begin{tikzcd}
	{\mathbf{P}([a_{(i)}] \star [c_{(k)}] , [b_{(j)}] \star [d_{(l)}])} \\
	{\mathbf{P}([a_{(i)}] , [b_{(j)}]) \times \mathbf{P}([c_{(k)}] , [d_{(l)}]) }
	\arrow["{\otimes_{a_{(i)}b_{(j)}c_{(k)}d_{(l)}}}", from=2-1, to=1-1]
\end{tikzcd}\] along with interchange laws (from poly-functorality of $\otimes$) such as the following commutative diagram \jnote{This diagram doesn't fit very well, but messing around with the scale option here doesn't seem to do anything!}
\[\begin{tikzcd}[scale=0.1]
	& {\mathbf{P}([e_{(m)}]x[e'_{(m')}] , [f_{(n)}]) \times \mathbf{P}([a_{(i)}] , [b_{(j)}]x[b'_{(j')}]) \times \mathbf{P}([c_{(k)}] , [d_{(l)}]) } \\
	{\mathbf{P}([e_{(m)}]x[e'_{(m')}] , [f_{(n)}]) \times\mathbf{P}([a_{(i)}] \star [c_{(k)}] , [b_{(j)}]x[b'_{(j')}] \star [d_{(l)}])} \\
	{ \times\mathbf{P}([e_{(m)}] \star [a_{(i)}] \star [c_{(k)}]\star [e'_{(m')}] , [b_{(j)}]\star [f_{(n)}] \star[b'_{(j')}] \star [d_{(l)}])} \\
	{ \mathbf{P}([e_{(m)}] \star [a_{(i)}] \star [e'_{(m')}] \star [c_{(k)}] , [b_{(j)}] \star [f_{(n)}] \star [b'_{(j')}] \star [d_{(l)}])  } \\
	& { \mathbf{P}([e_{(m)}] \star [a_{(i)}] \star [e'_{(m')}] , [b_{(j)}] \star [f_{(n)}] \star [b'_{(j')}]) \times \mathbf{P}([c_{(k)}] , [d_{(l)}]) }
	\arrow["\otimes", from=1-2, to=2-1]
	\arrow["{{\circ_{x}}}"', from=1-2, to=5-2]
	\arrow[from=2-1, to=3-1]
	\arrow["{{\mathbf{P}[\sigma, i]}}", from=3-1, to=4-1]
	\arrow["\otimes"', from=5-2, to=4-1].
\end{tikzcd}\]
The above commutative diagram is for post-composition on the left-hand side, equivalent diagrams are constructed for pre/post-composition on the left/right-hand side for full generality, and are constructed in a totally analogous way. 

Now, a proof that algebras of the acyclic wiring operad give (symmetric) polycategories can be inherited directly from \cite{yau2018operads}, however it is instructive to briefly outline the point. What must be constructed is $\circ_x$ composition which can be constructed as follows: \beq
\input{poly_proof_1.tikz} 
\eeq
Associativity laws for polycategories then follow directly from functorality of the algebra. For the parallel composition rule, note that we can construct the desired $\otimes$ functions as follows: \beq
\input{poly_proof_2.tikz} 
\eeq
Associativity of spatial composition then follows as
\begin{align*}
& \input{poly_proof_3.tikz} \\
= \quad & \input{poly_proof_5.tikz} \\
= \quad  &  \input{poly_proof_4.tikz} 
\end{align*}
and the interchange laws arise from using functorality on each side of the interchange law (be it left/right and pre/post) to reach a normal form \beq
\input{poly_proof_6.tikz}.
\eeq
Now to go the other way, note that in analogy to the argument of \cite{patterson2021wiring}, by induction, one can show that every acyclic wiring diagram decomposes into the form (up to permutations): 
\begin{align*}
 \input{poly_proof_8.tikz} = \quad  \input{poly_proof_9.tikz}
\end{align*}
where the dashed connecting wire represents the existence of either 1 or 0 wires. Now note that however complex the tail of the inductive decomposition is, its composition with $t_1$ can be formed using either the $\otimes$ or the $\circ_x$ operations. As a result one can construct the entire candidate algebra just by defining the algebra on those generators. The interchange law between composition and $\otimes$ that guarantees functorality of the algebra. 
\end{proof}

To really complete the result, we hope to see a roundtrip relation as proven for symmetric monoidal categories in \cite{patterson2021wiring}. However, in this case the roundtrip holds on the nose, because in this setting there are no coherences to take care of in polycategories, which are morally strict, being closer in spirit already to props than to general symmetric monoidal categories. 

\section{String diagrams for enriched symmetric monoidal categories} \label{app:EnrichedStringDiagrams}

We show in the main text that functors $F:\cW^A\to K$ correspond to $K$-enriched symmetric monoidal categories. What is more, is that the diagrams that we use to represent $F:\cW^A\to K$ can easily be adapted to give a convenient string diagrammatic notation for these enriched SMCs.

Consider the following diagram
\beq\input{operadexamplealgebra.tikz}. \eeq
If we are just interested in a particular $F$ then we can leave the application of $F$ implicit. Similarly, if we are not interested in using this to further refine some larger diagram, then the ``boxing'' is also irrelevant. That is, we don't lose any relevant information by just writing this as
\beq\input{enrichedstringdiagram.tikz},\eeq
where now the horizontal black wires represent objects in $K$ while the vertical red wires represent objects in $\mathcal{C}_F$. That is, this single diagram now captures the interplay between the enriching and enriched categories.

This also interacts nicely with the shorthand notation that we introduced in Eqs.~\eqref{def:seq}\eqref{def:par}. For example, using this we can write equations such as:
\beq\input{enrichedstringdiagram1.tikz}\ \ = \ \ \input{enrichedstringdiagram2.tikz}\ \ = \ \ \input{enrichedstringdiagram3.tikz},
\eeq
that is, we can move composition from being represented as connectivity in the diagram to being represented by operations in $K$.

Note that this is a substantial part of the diagrammatic language which was introduced in Ref.~\cite{schmid2020unscrambling} for studying the interaction between a theory of causation (the red wires) and a theory of inference (the black wires). Exploring these connections more deeply is a project left to future work.

\end{document}